\documentclass[10pt,reqno]{amsart}
\usepackage[utf8]{inputenc}
\RequirePackage[T1]{fontenc}
\RequirePackage{amsthm,amsmath,etoolbox}
\RequirePackage[numbers]{natbib}
\RequirePackage[colorlinks,linkcolor=blue,citecolor=blue,urlcolor=blue]{hyperref}
\usepackage{amssymb}
\usepackage{anysize}
\usepackage{hyperref}
\usepackage{extarrows}
\usepackage{multirow}
\usepackage{natbib}
\usepackage{stmaryrd}
\usepackage{color}
\usepackage{amsmath}
\usepackage{enumitem}
\usepackage{mathtools} %for sub/superstack
\usepackage{verbatim}
\usepackage{soul}
 
\numberwithin{equation}{section}

\theoremstyle{plain} 
\newtheorem{theorem}{Theorem}[section]
\newtheorem{lemma}[theorem]{Lemma}
\newtheorem{corollary}[theorem]{Corollary}
\newtheorem{proposition}[theorem]{Proposition}
\newtheorem{assumption}[theorem]{Assumption}
\newtheorem{definition}[theorem]{Definition}

\theoremstyle{remark}
\newtheorem{remark}[theorem]{Remark}

\newcommand{\sbt}{\,\begin{picture}(-1,1)(-1,-3)\circle*{3}\end{picture}\ }

\newcommand{\bs}{\boldsymbol}

\makeatletter
\newif\if@gather@prefix 
\preto\place@tag@gather{% 
  \if@gather@prefix\iftagsleft@ 
    \kern-\gdisplaywidth@ 
    \rlap{\gather@prefix}% 
    \kern\gdisplaywidth@ 
  \fi\fi 
} 
\appto\place@tag@gather{% 
  \if@gather@prefix\iftagsleft@\else 
    \kern-\displaywidth 
    \rlap{\gather@prefix}% 
    \kern\displaywidth 
  \fi\fi 
  \global\@gather@prefixfalse 
} 
\preto\place@tag{% 
  \if@gather@prefix\iftagsleft@ 
    \kern-\gdisplaywidth@ 
    \rlap{\gather@prefix}% 
    \kern\displaywidth@ 
  \fi\fi 
} 
\appto\place@tag{% 
  \if@gather@prefix\iftagsleft@\else 
    \kern-\displaywidth 
    \rlap{\gather@prefix}% 
    \kern\displaywidth 
  \fi\fi 
  \global\@gather@prefixfalse 
} 
\def\math@cr@@@align{%
  \ifst@rred\nonumber\fi
  \if@eqnsw \global\tag@true \fi
  \global\advance\row@\@ne
  \add@amps\maxfields@
  \omit
  \kern-\alignsep@
  \if@gather@prefix\tag@true\fi
  \iftag@
    \setboxz@h{\@lign\strut@{\make@display@tag}}%
    \place@tag
  \fi
  \ifst@rred\else\global\@eqnswtrue\fi
  \global\lineht@\z@
  \cr
}
\newcommand*{\beforetext}[1]{% 
  \ifmeasuring@\else
  \gdef\gather@prefix{#1}% 
  \global\@gather@prefixtrue 
  \fi
} 
\makeatother
  
\renewcommand{\mathbf}[1]{\bs{#1}}
 
\marginsize{35mm}{35mm}{38mm}{40mm}  

\allowdisplaybreaks
\begin{document}

 \begin{minipage}{0.85\textwidth}
 \vspace{2.5cm}
 \end{minipage}
\begin{center}
\large\bf
Density of the free additive convolution of multi-cut measures
\end{center}

\renewcommand{\thefootnote}{\fnsymbol{footnote}}	
\vspace{1cm}
\begin{center}
 \begin{minipage}{0.40\textwidth}
 \begin{center}
Philippe Moreillon\\ \vspace{0.1cm}
\footnotesize 
{Department of Mathematics, KTH Royal Institute of Technology,
10044 Stockholm, Sweden.}
{\it phmoreil@kth.se}
\end{center}
\end{minipage}

\renewcommand{\thefootnote}{\fnsymbol{footnote}}	

\end{center}

\vspace{1cm}

\begin{center}
 \begin{minipage}{0.83\textwidth}\footnotesize{
 {\bf Abstract.}  
We consider the free additive convolution semigroup $\lbrace \mu^{\boxplus t}:\,t\ge 1\rbrace$ and determine the local behavior of the density of $\mu^{\boxplus t}$ at the endpoints and at any singular point of its support. We then study the free additive convolution of two multi-cut probability measures and show that its density decays either as a square root or as a cubic root at any endpoint of its support.  The probability measures considered in this paper satisfy a power law behavior with exponents strictly between $-1$ and $1$ at the endpoints of their supports. }
\end{minipage}
\end{center}

 \vspace{5mm}
 
 {\small
\footnotesize{\noindent\textit{Date}: October 1, 2024.}
 
 \vspace{2mm}

 }

\thispagestyle{headings}

\section{Introduction}

One of the fundamental concepts of Voiculescu's free probability theory is the free additive convolution of two probability laws; given two probability measures $\mu$ and $\nu$ on the real line, their free additive convolution $\mu\boxplus\nu$ \cite{Voiculescu 1} is the distribution of the sum of the free self-adjoint noncommutative random variables $X$ and $Y$, whose respective distributions are given by $\mu$ and $\nu$. Though conceptually similar to the classical convolution of independent random variables, the free additive convolution is characteristically different from its classical counterpart. For instance, it has  strongly regularizing properties, e.g.\ the free additive convolution of two probability measures always admits a non-vanishing absolutely continuous part \cite{Bel06}. 
Although it only exists in the noncommutative context, the concept of freeness in noncommutative probability theory is seen as an analogue of the notion of independence in the classical setup, since several concepts can be developed both around freeness and independence. For example, as $n$ tends to infinity, the $n$-fold free convolutions $\mu^{\boxplus n}:=\mu\boxplus \dots \boxplus\mu$ satisfy, upon rescaling, a central limit theorem and converges to Wigner's semi-circular distribution \cite{Voi83}. As shown in \cite{Ber95, Nic96}, the $n$-fold free convolutions can be embedded in a semigroup $\lbrace \mu^{\boxplus t}:\ t\geq 1\rbrace$. Moreover, for any $t>1$, the measure $\mu^{\boxplus t}$ does not admit a singular continuous part and the density of its absolutely continuous part is locally analytic~\cite{Bel05,Bel04}. Moreover, as a function of $t$, the number of intervals in the support of~$\mu^{\boxplus t}$ is non-increasing~\cite{Huang15}.

The free additive convolution was first introduced as an algebraic operation through Voiculescu's $R$-transform \cite{Voiculescu 1}. It was later shown that the Cauchy-Stieltjes transform
of the free additive convolution is related to the Cauchy-Stieltjes transforms of the original 
measures through analytic subordination: given two probability measures $\mu$ and $\nu$ on the real line, there exist two analytic self-maps $\omega_\mu$ and $\omega_\nu$ of the upper half-plane such that the Cauchy-Stieltjes transform of the free additive convolution of $\mu$ and $\nu$ is given by \linebreak $m_{\mu}(\omega_\nu)=m_{\nu}(\omega_\mu)$, 
where $m_{\mu}$ and $m_\nu$ are the respective Cauchy-Stieltjes transforms of the probability measures $\mu$ and $\nu$.  Voiculescu and Biane observed the subordination phenomenon in \cite{Voi93} and in \cite{Bia98}, respectively. As was shown by Belinschi and Bercovici ~\cite{Bel Ber} and by Chistyakov and Götze \cite{Chis}, analytic function theory can be used to define the free additive convolution of probability measures. One way is by characterizing the subordination functions $\omega_\mu$, $ \omega_\nu$ as the Denjoy-Wolff fixed points of some analytic self-maps of the upper-half plane \cite{Bel Ber}. Function theory allows to study the free convolutions and their regularization properties for general probability measures; see~\cite{Bel06,Bel1,BB-GG,BeV98,B,BerWangZHong17,Huang15,Hon21,Min17,Wan10} and references therein.

In our preceding paper \cite{MoSc22}, we considered a class of multi-cut probability measures, i.e.\ probability measures supported on several intervals,  that are of Jacobi type. We looked at the number of intervals in the support of the free additive convolution of two such measures and at the number of intervals in the support of the free additive convolution semigroup $\lbrace \mu^{\boxplus t}:\ t\geq 1\rbrace$. Using function theory, we localized the subordination functions on the real line and then solved a combinatorial problem to bound the number of intervals in the support of the free additive convolutions. In the present paper, we turn to the study of the qualitative properties of the densities of the free additive convolutions and of the free additive convolution semigroups. More specifically, we examine the local behaviors of the densities of the free additive convolutions at the endpoints of their supports, at interior isolated zeros of the densities and at the points where they diverge. 

We first provide an analysis of the free additive convolution semigroup $\lbrace \mu^{\boxplus t}:\ t\geq 1 \rbrace$. In Theorem \ref{Theorem 1}, we obtain precise asymptotics of the density $\rho_t$ at the endpoints of its support. We show that it decays either as a square root or as a cubic root, when the original probability measure $\mu$ is absolutely continuous and of Jacobi type, in the sense given in Assumption~\ref{main assumption} with exponents ranging between $-1$ and $1$.  The square root behavior always occurs at the smallest and largest endpoints in the support of $\mu^{\boxplus t}$, for every value of $t>1$. For any other endpoint in the support of $\mu^{\boxplus t}$, there exists a critical value $t^*$ such that $\rho_t$ decays as a square root for any values of $t$ smaller than $t^*$, and admits a cusp-like singularity, i.e.\ the density behaves as a two-sided cubic root, for the value $t^*$. 
Our arguments are based on a quite involved asymptotic analysis of the solutions to several perturbed systems of equations of degree two or three via Cardano's formula, on the power law behavior in Assumption \ref{main assumption}, and on a localization argument to determine the correct branches of the roots of the solution. 
When $t>t^*$ is sufficiently close to $t^*$, we determine in Theorem~\ref{Theorem 4} the order of the density~$\rho_t$ and prove that it behaves locally as a square, near the set of zeros of some harmonic function $f$. To the best of our knowledge, the arguments in the proof of Theorem \ref{Theorem 4} and the connection between~$\rho_t$ and the zeros of $f$ are new.

In the absence of atoms, it was shown in~\cite{BerWangZHong2020} that the density of $\mu^{\boxplus t}$, $t>1$, is H\"older-continuous with exponent one third; see also~\cite{B} for the corresponding statement for the semicircular flow. The square and cubic roots arise naturally in random matrix theory. For generalized Wigner matrices with variance profile~\cite{AEK Universality general}, the limiting eigenvalue distribution is obtained via the solution to the (vector) Dyson equation. In~\cite{AEK} the densities were classified: only one-sided square and two-sided cubic roots occur. The precise transition from small gaps over cusp singularities to local minima were extensively studied in~\cite{AEK singularities,AEK18,EKS-CuspHermitian}. The square and cubic roots appear as solutions to perturbed third order equations and their ubiquity follows from showing that certain derivatives cannot vanish at spectral edges. For the model at hand, this will follow directly from a convexity argument.

Next, we move on to the setup where the probability measure $\mu$ further admits finitely many atoms. As shown in \cite{Bel05, Bel04, Huang15}, the atoms of $\mu^{\boxplus t}$ are characterized by the following:  $E_t:=tE_0$ is an atom of $\mu^{\boxplus t}$ if and only if $E_0$ is an atom of the measure $\mu$ with mass strictly greater than $1-\frac{1}{t}$. In Theorem \ref{Theorem 2}, we determine the local behavior of the density $\rho_t$ in a small neighborhood of the atoms of $\mu^{\boxplus t}$ and show that the density $\rho_t$ behaves according to the exponent of the original measure $\mu$. A consequence of our results is that the density $\rho_t$ does not diverge in a neighborhood of an atom $E_t$ of $\mu^{\boxplus t}$, as long as the density $\rho$ of $\mu$ is continuous in a neighborhood of the atom $E_0$ of $\mu$.

In Theorem \ref{Theorem 22}, we consider the critical case where an atom $E_0$ of $\mu$ has mass equal to $1-\frac{1}{t}$.  We prove that in a small neighborhood of $tE_0$, the density $\rho_t$ diverges at some universal rates, given by the inverse of a square or of a cubic root; or in some more specific situations set out in items (3b), (3c) and (3d) of Theorem \ref{Theorem 22}, the density diverges with respect to an exponent that can take any value in $(-1,0)$. 
In both Theorems \ref{Theorem 2} and \ref{Theorem 22}, we obtain a classification of the orders of decay or divergence of the density $\rho_t$. 
We extend our results to a regime where~$t$ is no longer fixed and evolves according to some dictated rates. We also determine the dependence of the parameter $t$ with the constant before the leading term in the asymptotics of the density $\rho_t$. We further mention that when $\mu(\lbrace E_0 \rbrace)<1-\frac{1}{t}$, it follows from \cite{Bel2} that the density $\rho_t$ is locally analytic in a neighborhood of $tE_0$. For generalized Wigner matrices with a block structure, the limiting eigenvalue distribution can develop singularities or an atom at the origin~\cite{Kolupaiev,KR-singularity}. 

The proofs of Theorems~\ref{Theorem 2} and~\ref{Theorem 22} rely on localizing the subordination function $\omega_t$ on the real line, and on the  resolution of several perturbed systems whose leading terms have a general power ranging between two and three. Since the parameter $t>1$ is not fixed, the resolution of such systems turns out to be quite technical and requires several Taylor expansions and the use of Plemelj's formula.

Next, we consider the free additive convolution of two multi-cut probability measures $\mu$ and $\nu$ satisfying Assumption \ref{main assumption2}. It was proven in \cite{Bao20} that the density of the free additive convolution $\mu\boxplus\nu$ decays as a square root at the endpoints of its support, if $\mu$ and $\nu$ are supported on a single interval. As it was laid out in our preceding paper \cite{MoSc22}, the case of multi-cut measures turns out to be significantly more involved, since the subordinations functions are no longer bounded. A main feature of the multi-cut setup is that one of the subordination functions can diverge in a local neighborhood of an endpoint of the support of the free additive convolution $\mu\boxplus\nu$. This comes from the fact that the subordination functions can approach the zeros of the Cauchy-Stieltjes transforms of the original measures, in the case where~$\mu$ and~$\nu$ are distinct multi-cut measures. Using a localization argument, we extend the result in \cite{Bao20} to the case of multi-cut measures satisfying Assumption \ref{main assumption2} and prove in Theorem~\ref{Theorem 3} that the density $\rho_{\mu\boxplus \nu}$ of $\mu\boxplus\nu$ decays either as one-sided square root or as a two-sided cubic root at each edge point of its support. Under our assumptions, the square root behavior is always observed at the smallest and the largest endpoints in the support of $\mu\boxplus\nu$, and both square or cubic root decays can be observed for the other endpoints in the support. 
In the proof, we used our localization argument to derive precise asymptotics for the subordination functions when one of them diverges. Our method consists in deriving a single equation  from a system of two dependent equations and to perform an asymptotic analysis of its solution. We then study the first three derivatives of a local inverse of the subordination functions. The signs of these derivatives are a priori not clear in the case where $\mu$ and $\nu$ are distinct multi-cut probability measures and showing that only the square and cubic root decays can be observed is one of the achievements of this paper. It amounts to proving that the three first derivatives cannot vanish simultaneously. Since the expressions of these derivatives is quite general, proving that the third derivative has a sign if the two other ones vanish relies on reexpressing it in terms of Nevanlinna functions and using several identities iteratively.  

The density of the free additive convolution~$\mu\boxplus\nu$ was studied in~\cite{BerWangZHong2020} under the assumption that~$\nu$ is $\boxplus$-infinitely divisible. In particular, the authors obtained similar results for internal cusps. They further showed that in case~$\mu$ admits a locally analytic density that decays with an exponent bigger than two at an edge point of the support, internal zeros of the density can arise. At such points the density can vanish with exponents bigger than one, possibly asymmetrically with different exponents; see~\cite[Appendix B]{BerWangZHong2020}. The measures~$\mu$ and~$\nu$ considered in this paper are typically not $\boxplus$-infinitely divisible. When relaxing our assumption that the densities behave with exponents smaller than one, we expect to observe a richer behavior than in Theorem~\ref{Theorem 3}.

\textit{Organization of the paper: } In Section \ref{section: statement of results}, we state our main results.
In Section~\ref{section: preliminaries}, we recall the definitions of several objects from analytic function theory, such as the negative reciprocal Cauchy-Stieltjes transforms and the Nevanlinna representations.
In Section~\ref{section: free additive convolution semigroup}, we define the free additive convolution semigroup  and prove Theorems~\ref{Theorem 1} and~\ref{Theorem 4}. We localize the subordination function on the real line and determine when its imaginary part is strictly positive. We then solve several perturbed systems of degree $2$ or $3$.
In Section~\ref{section: Proofs of Theorems}, we use Plemelj's formula to solve several perturbed systems of general degrees between $2$ and $3$, and determine the asymptotics of the subordination function, when it approaches an atom of the original probability measure. We then  prove Theorems~\ref{Theorem 2} and~\ref{Theorem 22}.
In Section~\ref{section: free addition definition}, we recall the definition of the free additive convolution of a probability measure and some results from \cite{MoSc22}. We then turn to the proof of Theorem~\ref{Theorem 3}.

\textit{Notation: } We define the support $\mathrm{supp}(\mu)$ of a probability measure $\mu$ on the real line, as the complement of the union of open sets having measure $0$ with respect to $\mu$. Similarly, the support of the density $\rho$ of its absolutely continuous part is defined as the closure of the set $\lbrace x\in\mathbb{R}:\ \rho(x)>0\rbrace$. Given a probability measure $\mu$, we denote by $\mu_{\mathrm{ac}}$, $\mu_{\mathrm{pp}}$ and $\mu_{\mathrm{sc}}$ its absolutely, atomic and singular continuous parts respectively. Moreover we use $m_{\mu}(z)$ to denote the Cauchy-Stieltjes transform of a given measure $\mu$ supported on the real line. Throughout this paper, the version of the density of a given measure that we consider is always the one corresponding to the non-tangential limit of the imaginary part of its Cauchy-Stieltjes transform.  Moreover we denote by $c$, $C$ and $S$ strictly positive constants, whose values may change from line to line. Finally, $\mathbb{C}^+$ stands for the complex upper half-plane, i.e.\ $\mathbb{C}^+:=\lbrace z\in\mathbb{C}:\ \mathrm{Im}\,z>0\rbrace$.

\section{Statement of results}\label{section: statement of results}
\subsection{The free additive convolution semigroup}
\ \\
\vspace{-0.32cm}
\\
Let $\mu$ be a Borel probability measure on the real line. Bercovici and Voiculescu proved in~\cite{Ber95} that the free additive convolution powers $\big\lbrace \mu^{\boxplus n}:\ n\in\mathbb{N} \big\rbrace$ can be embedded in a continuous semigroup for large values of $t$, and Nica and Speicher~\cite{Nic96} extended this to all values of $t\ge 1$.  We first study the density of $\mu^{\boxplus t}$, when $\mu$ is a multi-cut measure, i.e.\ a measure whose absolutely continuous part consists of several intervals, and belongs to the same class of measures as in \cite{MoSc22}, characterized as follows.

\begin{assumption}\label{main assumption}
Let $n_{\mathrm{ac}}\geq 1$ and $n_{\mathrm{pp}}\geq 0$ be integers. We consider a centered and compactly supported probability measure $\mu$ on the real line, whose singular part consists of $n_{\mathrm{pp}}$ atoms and whose absolutely continuous part is supported on $n_{\mathrm{ac}}$  disjoint intervals. We assume that the density $\rho$ of $\mu_{\mathrm{ac}}$ satisfies a power law behavior at the endpoints of each interval in its support. More specifically, for every $j=1,...,n_{\mathrm{ac}}$, there exist exponents $-1<\beta_j^{-},\beta_j^{+}<1$ and a constant $C_j\geq 1$ such that
\begin{align}\label{Jacobi}
C_j^{-1}<\frac{\rho(x)}{(x-E_j^{-})^{\beta_j^{-}}(E_j^{+}-x)^{\beta_j^{+}}}<C_j,\quad\text{ for a.e.\ }x\in[E^{-}_j,E^{+}_j].
\end{align}
\end{assumption}

\begin{remark}
(1.) The assumption that $\mu$ is centered is just for simplicity. By a simple shift, all our results extend to non-centered measures. 
(2.) We notice that \eqref{Jacobi} ensures that~$\rho$ is a.e.\ uniformly bounded from below and above in any compact set in the interior of $\mathrm{supp}(\mu_{\mathrm{ac}})$.
(3.) \eqref{Jacobi} in Assumption \ref{main assumption} is natural in random matrix theory, since the spectral measures of a wide class of random matrix ensembles, including Wigner and Wishart matrices, converge to probability distributions satisfying \eqref{Jacobi}, e.g.\ the Wigner's semicircle and  Marchenko–Pastur laws.
Other examples of random matrix models whose equilibrium measures satisfy Assumption \ref{main assumption} include Muttalib--Borodin and unitary invariant ensembles: in the case of Muttalib--Borodin ensemble, the equilibrium measure blows up with a general exponent $\beta\in (-1,0)$ near the origin, while in the case of unitary invariant ensembles, it may be  supported on several intervals and either decays as a square root or blows up as the inverse of a square root, see e.g.\ \cite{CR14}, \cite{CF}, \cite{CG21}. 
Assumption~\ref{main assumption} is also natural to study universal rates of decay for the density of free convolutions. Indeed, when $\rho$ decays according to an exponent $\beta\geq 1$ at a given endpoint, one can show that the density $\rho_t$ still decays according to the exponent~$\beta$ of~$\rho$ for small values of $t\geq 1$, and will exhibit the universal one-sided square and two-sided cubic root decays, as $t$ becomes larger.
\end{remark}

Let $\mu$ be a probability measure on the real line that is not a point mass and let $t>1$. As it was proven in Theorem 4.1 in \cite{Bel1}, the absolutely continuous part of $\mu^{\boxplus t}$ is always nonzero. Moreover its density $\rho_t$ is continuous whenever it is finite, and analytic at any point where it does not vanish nor diverge. A formula for $\rho_t$ is given in terms of the Nevanlinna representation of $F_{\mu}$ and of the subordination function $\omega_t$ in Theorem $3.8$ of~\cite{Huang15}. Moreover it is shown in \cite{Bel05, Bel04, Huang15} that $E_t$ is an atom of $\mu^{\boxplus t}$ if and only if  $E_0:=\frac{E_t}{t}$ is an atom of~$\mu$ satisfying~$\mu(\lbrace E_0 \rbrace)>1-\frac{1}{t}$. In the present paper, we derive precise asymptotics for the density $\rho_t$ near the points where it vanishes or diverges. 
Given $t>1$, we define the set
\begin{align*}
\mathcal{V}_t:=\partial\big\lbrace x\in\mathbb{R}:\ 0<\rho_t(x)<\infty \big\rbrace.
\end{align*}
The set of points in $\mathcal{V}_t$ was analysed in \cite{Huang15, MoSc22} and an explicit upper bound was given in Theorem 1.2 in \cite{MoSc22}. 
Let $t>1$ and $E_t$ be a real point in $\mathcal{V}_t$. If $E_t$ verifies 
\begin{align}\label{exterior point}
E_t-\varepsilon\notin\mathrm{supp}\big(\mu_{\mathrm{ac}}^{\boxplus t}\big)
\text{ or }
E_t+\varepsilon\notin\mathrm{supp}\big(\mu_{\mathrm{ac}}^{\boxplus t}\big),\ 
\qquad \text{for all $0<\varepsilon<\varepsilon_0$,}
\end{align}
for some $\varepsilon_0>0$, then we shall call $E_t$ an \textit{exterior point} of $\mathcal{V}_t$. On the contrary, if $E_t$ verifies 
\begin{align}\label{interior point}
E_t-\varepsilon\in\mathrm{supp}\big(\mu_{\mathrm{ac}}^{\boxplus t}\big)
\text{ and  }
E_t+\varepsilon\in\mathrm{supp}\big(\mu_{\mathrm{ac}}^{\boxplus t}\big),\ 
\qquad \text{for all $0<\varepsilon<\varepsilon_0$,}
\end{align}
for some $\varepsilon_0>0$, then we shall call $E_t$ an \textit{interior point} of $\mathcal{V}_t$.  
The fact that a point in $\mathcal{V}_t$ is exterior or interior greatly depends on the value of $t$. Indeed, let $\mathcal{K}$ be a bounded interval in $\mathbb{R}\backslash \mathrm{supp}(\mu_{\mathrm{ac}})$. We prove in Section~\ref{section: free additive convolution semigroup} that there exists at least one critical value $t^{*}>1$ depending on $\mathcal{K}$ (see \eqref{definition omega0}) such that the following holds:  
\begin{itemize}[leftmargin=0.32cm]
\item[\sbt] When $t<t^*$, the density $\rho_t$ admits at least two exterior endpoints of $\mathcal{V}_t$ in $\mathcal{K}$. 
\item[\sbt] As $t$ converges to $t^*$, two exterior endpoints in $\mathcal{K}$ merge to form  an interior endpoint $E_*\in\mathcal{K}$ of $\mathcal{V}_{t^*}$. 
\item[\sbt] As $t$ becomes greater than $t^*$, there are no longer points of $\mathcal{V}_t$ in a small neighborhood~of~$E_*$. 
\end{itemize}

In our main results, we distinguish two distinct cases. (1.) In Theorems \ref{Theorem 1} and \ref{Theorem 4}, we look at the local behavior of $\rho_t$ near any point $E_t$ in $\mathcal{V}_t$ such that~$E_0:=\frac{E_t}{t}$ is not an atom of~$\mu$ with $\mu(\lbrace E_0 \rbrace)\geq 1-\frac{1}{t}$. In particular, we prove that~$\rho_t$ vanishes at universal rates near such regular points.  (2.) In Theorems \ref{Theorem 2} and \ref{Theorem 22}, we determine precise asymptotics for~$\rho_t$ near any point~$E_t$ such that $E_0:=\frac{E_t}{t}$ is an atom of $\mu$ with $\mu(\lbrace E_0 \rbrace)\geq 1-\frac{1}{t}$. As we show in Section~\ref{section: free additive convolution semigroup}, a point $E_t\in\mathcal{V}_t$ at which $\rho_t$ diverges is necessarily of this form. When~$\mu(\lbrace E_0 \rbrace)> 1-\frac{1}{t}$, we prove that $\rho_t$ and $\rho$ behave according to a similar exponent. When $\mu(\lbrace E_0 \rbrace)=1-\frac{1}{t}$, we show that $\rho_t$ diverges at $E_t$. In particular, we prove that the rates of divergence of $\rho_t$ are universal and given by the inverse of a square or of a cubic root, when $E_0$ does not lie at an endpoint of $\mathrm{supp}(\mu_{\mathrm{ac}})$.  

To simplify matters, we will employ the following notations in the subsequent theorems. Given $t>1$, Let $\mathcal{S}_t$ be the set of real points $E_t$ such that $E_0:=\frac{E_t}{t}$ is an atom of $\mu$ with~$\mu(\lbrace E_0 \rbrace)\geq 1-\frac{1}{t}$ and $\mathcal{R}_t$ be its complement.

\begin{theorem}\label{Theorem 1}
Let $\mu$ satisfy Assumption \ref{main assumption}, $t>1$, $E_t\in\mathcal{R}_t$ be an exterior point of $\mathcal{V}_t$ and~$\mathcal{U}_t$ be a small neighborhood of $E_t$. There exists a positive constant $S_1$ depending on $t$ such that for every~$x$ in~$\mathcal{U}_t$, we either have 
\begin{align}\label{statement 1 2.2}
&\rho_t(x)
=
S_1\,\frac{t}{\sqrt{(t-1)^{3}}}\sqrt{x-E_t}
+
\mathcal{O}\big(x-E_t \big),\qquad x\geq E_t,
\end{align}
or we have
\begin{align}\label{statement 2 2.2}
&\rho_t(x)
=
S_1\,\frac{t}{\sqrt{(t-1)^{3}}}\sqrt{E_t-x}
+
\mathcal{O}\big(E_t-x\big),\qquad x\leq E_t.
\end{align}
Whether \eqref{statement 1 2.2} or \eqref{statement 2 2.2} occurs depends on the sign of the quantity $C_1(t)$ in~\eqref{definition C1}. Moreover an expression for the constant $S_1$ is given in equations \eqref{omega solution finale 1} and~\eqref{omega solution finale 2}. 
If $E_t$ is not the leftmost or rightmost endpoint of $\mathrm{supp}(\mu^{\boxplus t})$, then there exists $t^*>t$ and an interior point~$E_*$ of $\mathcal{V}_{t^*}$ such that $E_t$ converges to $E_*$, as $t$ approaches $t^*$.  We assume that $E_{*}\in\mathcal{R}_{t^*}$ and look at two distinct regimes: 
\vspace{0.1cm}

(1.) If $|x-E_t|^2= o\big( (t^*-t)^{3} \big)$, then there exists a positive constant $S_2$ that is independent of $t>1$ and such that for all $x$ in $\mathcal{U}_t$, we either have
\begin{align}\label{statement 3 2.2}
\rho_t(x)
=
S_2
\frac{t^*}{\sqrt[4]{3(t-1)(t^*-1)^3}}
\frac{\sqrt{x-E_{t}}}{\sqrt[4]{t^*-t}}
+
\mathcal{O}\bigg(\frac{x-E_{t}}{\sqrt{t^*-t}} \bigg)
+
\mathcal{O}\Big(\sqrt{x-E_{t}}\sqrt[4]{t^*-t}\Big),
\end{align}
when $x\geq E_t$, and $\rho_t(x)=0$ when $x< E_t$,  or we have 
\begin{align}\label{statement 4 2.2}
&\rho_t(x)
=
S_2\,\frac{t^*}{\sqrt[4]{3(t-1)(t^*-1)^3}}
\frac{\sqrt{E_{t}-x}}{\sqrt[4]{t^*-t}}
+
\mathcal{O}\bigg(\frac{E_t-x}{\sqrt{t^*-t}} \bigg)
+
\mathcal{O}\Big(\sqrt{E_t-x}\sqrt[4]{t^*-t}\Big),
\end{align}
when $x\leq E_t$, and $\rho_t(x)=0$ when $x> E_t$. Again, whether \eqref{statement 3 2.2} or \eqref{statement 4 2.2} occurs depends on the sign of $C_1(t)$ in~\eqref{definition C1}. Moreover $S_2$ is determined in equations \eqref{omega solution finale 3} and \eqref{omega solution finale 4}. 
\vspace{0.1cm}

(2.) If $t<t^{*}$ converges to $t^*$ in a way such that $(t^*-t)^3=o\big( |x-E_{t}|^2 \big)$, then there exists a positive constant $S_3$ such that for all $x$ in $\mathcal{U}_t$,
\begin{align}\label{statement 5 2.2}
\rho_t(x)
=
S_3\,
\frac{t^*}{\sqrt[3]{(t^*-1)^4}}
{\sqrt[3]{|x-E_{t}|}} 
+ 
\mathcal{O}\Big( | x-E_{t}|^{\frac{2}{3}}\Big)
+
\mathcal{O}\Big( \sqrt{t^*-t}\Big).
\end{align}
The constant $S_3$ is independent of $t>1$ and is given explicitly in equation  \eqref{solution omega finale 3}.
\end{theorem}
\vspace{0.1cm}

\begin{remark}
(1.) When $t<t^*$ is fixed, we show in \eqref{statement 1 2.2} and \eqref{statement 2 2.2} that the density~$\rho_t$ exhibits a decay of (one-sided) square root type near the exterior point $E_t$. When $t=t^*$, $E_{t^*}=E_*$ is an interior point of $\mathcal{V}_{t^*}$ and we see in \eqref{statement 5 2.2} that the density $\rho_{t^*}$ exhibits a (two-sided) cubic root decay near $E_{*}$. The analysis is then extended to the regimes where $|x-E_t|^2=o\big( (t^*-t)^{{3}} \big)$ or where $(t^*-t)^{{3}}=o\big( |x-E_t|^2 \big)$. (2.) The case where the interior point $E_*$ lies in $\mathcal{S}_{t^*}$ is treated in Theorem \ref{Theorem 22}.
\end{remark}

In the next theorem, we consider any interior point $E_*$ of $\mathcal{V}_{t^*}$ that lies in $\mathcal{R}_{t^*}$ and determine the behavior of the density $\rho_t$ in a small neighborhood of $E_*$, when $t>t^*$. 

\begin{theorem}\label{Theorem 4}
Let $\mu$ be a probability measure that satisfies Assumption \ref{main assumption} and let $t^*>~\hspace{-0.1cm}1$ and $E_*\in\mathcal{R}_{t^*}$ be such that $E_*$ is an interior point of $\mathcal{V}_{t^*}$. There exist $\varepsilon>0$ and a neighborhood $\mathcal{U}_\varepsilon$ of $E_*$ such that
\begin{enumerate} 
\item for any $0<t-t^*<\varepsilon$, there exist $\widetilde{E_t}\in\mathcal{U}_{\varepsilon}$ and a constant $C_{t}$  so that when $x\in\mathcal{U}_\varepsilon$ and $|x-\widetilde{E_t}|^2=o\big( (t-t^*)^{{3}} \big)$, 
\begin{align}\label{*1}
\rho_t(x)-\rho_t(\widetilde{E_{t}})=C_{t} (x-\widetilde{E_{t}})^2+ \mathcal{O}\bigg({(t-t^*)^{-4}|x-\widetilde{E_{t}}|^3}\bigg),
\end{align}
\item for any $0<t-t^*<\varepsilon$, there exist $\widetilde{E_t}\in\mathcal{U}_{\varepsilon}$ and a constant $\widetilde{C}_{t}$  so that when $x\in\mathcal{U}_\varepsilon$ and $(t-t^*)^3=o\big( |x-\widetilde{E_t}|^{{2}} \big)$,
\begin{align}
\rho_t(x)-\rho_t(\widetilde{E_{t}})
=
\widetilde{C_t}\sqrt[3]{{x-\widetilde{E_t}}}
+
\mathcal{O}\big( \sqrt{t-t^*} \big).
\end{align}
\end{enumerate}
\end{theorem}

\begin{remark}
Equation \eqref{*1} still holds when $t>t^*$ is small but fixed and the constant $C_t$ in \eqref{*1} behaves as $(t-t^*)^{-\frac{5}{2}}$, as~$t>t^*$ approaches $t^*$. Moreover it follows from \eqref{asymptotics Re and Im} that~$\rho_t(x)$ is of order $\sqrt{t-t^*}$, when~$x$ lies in a neighborhood of $E_*$.
\end{remark}

We next turn to the study of the local behavior of $\rho_t$ near the points in $\mathcal{S}_t$, i.e.\ the points~$E_t$ such that $E_0:=\frac{E_t}{t}$ is an atom of $\mu$ with $\mu(\lbrace E_0 \rbrace)\geq 1-\frac{1}{t}$. The analysis performed in Theorems \ref{Theorem 2} and \ref{Theorem 22} below relies on a version of Plemelj's formula (see Lemma \ref{Plemelj}) that holds along any sequences that are not necessarily non-tangential. In order to use Lemma~\ref{Plemelj}, we make the following assumption on the probability measure $\mu$.

\begin{assumption}\label{Assumption Holder  continuity}
Let $\mu$ be a probability measure that satisfies Assumption \ref{main assumption} and $\rho$ denote the density of its absolutely continuous part.  We assume that for any interval $[E_{-},E_{+}]$ in the support of $\mu$, the function 
\begin{align}\label{Holder  continuity}
f(x)
:=
\frac{\rho(x)}{(x-E_{-})^{\beta^{-}}(x_{+}-E)^{\beta^{+}}},\qquad x\in[E_{-},E_{+}],
\end{align}
is Hölder continuous. We use the notations $\beta^{-}$, $\beta^{+}$ for the exponents of $\rho$ corresponding respectively to $E_{-}$ and $E_{+}$ in \eqref{Jacobi}.
\end{assumption}

The local behavior of $\rho_t$ near a point $E_t=tE_0$ in $\mathcal{S}_t$ crucially depends on the behavior of $\rho$ near $E_0$. In Theorems \ref{Theorem 2} and~\ref{Theorem 22}, we consider three distinct cases : 
\begin{enumerate}[leftmargin=1.7cm]
\item[(Case 1.)] $E_0$ is at positive distance from $\mathrm{supp}(\mu_{\mathrm{ac}})$, 
\item[(Case 2.)] $E_0$ lies in the interior of $\mathrm{supp}(\mu_{\mathrm{ac}})$, 
\item[(Case 3.)] $E_0$ is the left endpoint $E_{-}$ of an interval $[E_{-},E_+]$ in $\mathrm{supp}(\mu_{\mathrm{ac}})$ and $\beta$ denotes the corresponding exponent of~$\rho$ in \eqref{Jacobi}.
\end{enumerate}
To simplify the statements and proofs of the subsequent theorems, we define for any $t>1$, 
\begin{align}\label{notation A}
a(t)&:=t\mu(\lbrace E_0 \rbrace)-(t-1),\\ \label{notation B}
b(\omega)&:=\int_{\mathbb{R}\backslash\lbrace E_0 \rbrace}\frac{1}{u-\omega}\,\mu(\mathrm{d}u),\qquad \omega\in\mathbb{C}^+.
\end{align} 
We remark that $a(t)$ is non-negative for any $E_t=tE_0\in\mathcal{S}_t$ and that $b(\omega)$ extends non-tangentially to the real line by Plemelj's formula.
In Theorem \ref{Theorem 2} below, we consider the local behavior~of $\rho_t$ in the regime where either $a(t)>0$ (i.e.\ $E_t$ is an atom of~$\mu^{\boxplus t}$) or~$a(t)$ tends to $0$ at a sufficiently slow rate. In this regime, we show that the local behaviors of $\rho_t$ and $\rho$ are the same.

\begin{theorem}\label{Theorem 2} Let $\mu$ be a probability measure that satisfies Assumption \ref{Assumption Holder  continuity}, $t>1$, $E_t:=~tE_0$ be a point in $\mathcal{S}_t$ and $\mathcal{U}_t$ be a small neighborhood of $E_t$. Then the following claims hold. 
\begin{enumerate}[leftmargin=0.8cm]
\item In (Case 1.), $\rho_t(x)=0$ for any $x\in\mathcal{U}_t$ such that $|x-E_t|=o(a(t)^2)$.  

\item In (Case 2.),
$
\rho_t(x)
=
\frac{ t \mu(\lbrace E_0 \rbrace)  \rho(E_0) }{a(t)}+
\mathcal{O}\Big( \frac{|x-E_t|}{a(t)^3}\Big)$ if $x\in\mathcal{U}_t$ verifies $|x-E_t|=o(a(t)^2)$.  

\item[(3a)] In (Case 3.), when $\beta$ is nonzero, there exists a constant $C_\beta$, that is independent of $t>1$ and such that for all $x\in\mathcal{U}_t$ that verify $|x-E_t|=o(a(t)^2)$ and $|x-E_t|^{1+\beta}=o\big( a(t)^{2+\beta}\big)$,
\begin{align}\label{claim 2}
\rho_t(x)
&= 
\frac{t \mu(\lbrace E_0 \rbrace)^{1+\beta} C_\beta}{a(t)^{1+\beta}}(x-E_t)^{\beta}
+
\mathcal{O}\bigg( \frac{(x-E_t)^{1+\beta}}{a(t)^{3+\beta}}\bigg)
+
\mathcal{O}\bigg( \frac{(x-E_t)^{1+2\beta}}{a(t)^{3+2\beta}}\bigg),
\end{align}
when $x\geq E_t$ and $\rho_t(x)=0$, when $x<E_t$. 

\item[(3b)] In (Case 3.), when the exponent $\beta$ is zero, there exists a constant $C_0$, that is independent of $t>1$ and such that for all $x\in\mathcal{U}_t$ that verify $\big|\log(\frac{x-E_t}{a(t)})(x-E_t)\big|=o\big( a(t)^2 \big)$,
\begin{align} 
\rho_t(x)
=
\frac{t \mu(\lbrace E_0 \rbrace) C_0}{a(t)}
+
\mathcal{O}\bigg(\frac{x-E_t}{a(t)^3}\log\Big(\frac{x-E_t}{a(t)}\Big)^2\bigg),
\end{align}
when $x\geq E_t$ and $\rho_t(x)=0$, when $x<E_t$.
\end{enumerate}
Moreover (3a) and (3b) also hold when $E_0$ is a right endpoint of $\mathrm{supp}(\mu_{\mathrm{ac}})$, by replacing $x-E_t$ by $E_t-x$. 
\end{theorem}

\begin{remark} It follows from our proof that all the statements of Theorem \ref{Theorem 2} remain valid when $a(t)$ remains strictly positive and $x\in\mathcal{U}_t$.  In that situation, we proved in (2) that $\rho_t$ is locally constant near $E_t$ and in (3a) and (3b) that it behaves according to the same exponent as $\rho$.
\end{remark}

In Theorem \ref{Theorem 22}, we consider the local behavior of $\rho_t$ in the regime where either $a(t)=0$ (and consequently $E_t$ is not an atom of $\mu^{\boxplus t}$) or $a(t)$ converges to $0$ rapidly. In this regime, we obtain some universal rates of divergence in (Case 1.) and (Case 2.) depending on the value of $b(E_0)$, and especially on its sign. In (Case 3.), the rates of divergence that we obtain heavily depend on the exponent $\beta$ of $\rho$ and on $b(E_0)$.

\begin{theorem}\label{Theorem 22}
Let $\mu$ be a probability measure that verifies Assumption \ref{Assumption Holder  continuity}, $t>1$,~$E_t:=~\hspace{-0.15cm}tE_0$ be a point in $\mathcal{S}_t$ and $\mathcal{U}_t$ be a small neighborhood of $E_t$. For any $x\in\mathcal{U}_t$, the following holds.
\begin{enumerate}[leftmargin=0.8cm]
\item[(1a)] In (Case 1.), if $b(E_0)$ is strictly positive and $a(t)^2=o\big(|x-E_t|\big)$, then 
\begin{align}
\rho_t(x)
=
\frac{1}{\pi}\sqrt{\frac{t\mu(\lbrace E_0 \rbrace)b(E_0)}{x-E_t}}
+
\mathcal{O}\big(1\big)
+
\mathcal{O}\Big(\frac{a(t)}{{x-E_t}}\Big), 
\end{align}
when $x\geq E_t$ and $\rho_t(x)=0$, when $x<E_t$.

\item[(1b)]
In (Case 1.), if $b(E_0)$ is strictly negative and  $a(t)^2=o\big(|x-E_t|\big)$, then 
\begin{align}
\rho_t(x)
=
\frac{1}{\pi}\sqrt{\frac{t\mu(\lbrace E_0 \rbrace)b(E_0)}{x-E_t}}
+
\mathcal{O}\big(1\big)
+
\mathcal{O}\Big(\frac{a(t)}{{E_t-x}}\Big), 
\end{align}
when $x\leq E_t$ and $\rho_t(x)=0$, when $x>E_t$.
\item[(1c)]
In (Case 1.), if $b(E_0)$ is zero and $a(t)^3=o\big(|x-E_t|^2\big)$, then 
\begin{align}
\rho_t(x)
=
\frac{\sin(\frac{\pi}{3})}{\pi}\sqrt[3]{\frac{t\mu(\lbrace E_0 \rbrace)^2 b'(E_0)}{|x-E_t|}}
+
\mathcal{O}(1)
+
\mathcal{O}\bigg(\frac{a(t)}{|x-E_t|}\bigg).
\end{align}

\item[(2)] In (Case 2.), when $a(t)^2=o\big(|x-E_t|\big)$, there exist some strictly positive constants~$S_1^{+}$ and~$S_1^{-}$ that are independent of~$t>1$ and such that 
\begin{align}
\rho_t(x)
=
{S_1^{\pm}}\,\sqrt{\frac{t\mu(\lbrace E_0 \rbrace)|b(E_0)|}{|x-E_t|}}
+
\mathcal{O}(1)
+
\mathcal{O}\bigg(\frac{a(t)}{|x-E_t|}\bigg), 
\end{align}
where $S_1^{\pm}$ stands for $S_1^{+}$ if $x>E_t$, and for $S_1^{-}$ if $x<E_t$.

\item[(3a)] In (Case 3.), when $\beta$ and $b(E_0)$ are strictly positive and $a(t)^2=o\big(|x-E_t|\big)$,
\begin{align}
\rho_t(x)
=
\frac{1}{\pi}
\sqrt{\frac{t\mu(\lbrace E_0 \rbrace)b(E_0)}{x-E_t}} 
+
\mathcal{O}\Big( (x-E_t)^{\frac{\beta-1}{2}}\Big)
+
\mathcal{O}\bigg( \frac{a(t)}{x-E_t} \bigg),
\end{align}
when $x\geq E_t$ and $\rho_t(x)=0$, when $x<E_t$.
\item[(3b)] In (Case 3.), when $\beta$ is strictly positive and $b(E_0)$ is strictly negative, there exists a positive constant~$S_2$ that is independent of $t>1$ and such that when $a(t)^2=o\big(|x-E_t|\big)$,
\begin{align}\label{non sym1}
\rho_t(x)\nonumber
&= 
S_2 \Bigg({\frac{\mu(\lbrace E_0 \rbrace)}{|b(E_0)|}}\Bigg)^{\frac{1+\beta}{2}}
\Bigg({\frac{t}{x-E_t}} \Bigg)^{\frac{1-\beta}{2}} 
\hspace{-0.5cm}
+
\mathcal{O}\Big( (x-E_t)^{\beta-\frac{1}{2}}\Big)
+
\mathcal{O}\bigg( \frac{a(t)}{(x-E_t)^{1-\frac{\beta}{2}}} \bigg)\\
&+
\mathcal{O}\bigg( \frac{a(t)^2}{(x-E_t)^{\frac{3}{2}}} \bigg),
\end{align}
when $x\geq E_t$ and when $x<E_t$,
\begin{align}\label{non sym2}
\rho_t(x)
\hspace{-0.05cm}
= 
\hspace{-0.05cm}
\frac{1}{\pi}
\sqrt{\frac{t\mu(\lbrace E_0 \rbrace)b(E_0)}{x-E_t}} 
\hspace{-0.05cm}
+
\hspace{-0.05cm}
\mathcal{O}\Big( (E_t-x)^{\frac{\beta-1}{2}}\Big)
\hspace{-0.05cm}
+
\hspace{-0.05cm}
\mathcal{O}\bigg( \frac{a(t)}{E_t-x} \bigg).
\end{align}
\item[(3c)] In (Case 3.), when $\beta$ is strictly positive and $b(E_0)$ is zero, there exist some positive constants~$S_3^{-}, S_3^+$ that are independent of $t>1$ so that when~$a(t)^{2+\beta}=o\big(|x-E_t|^{1+\beta}\big)$, we have
\begin{align}
\rho_t(x) 
=
S_3^{\pm}  
\Bigg( \frac{t \mu(\lbrace E_0 \rbrace)^{1+\beta}}{|x-E_t|} \Bigg)^{\frac{1}{2+\beta}}
\hspace{-0.5cm}
+
\mathcal{O}\Big( |x-E_t|^{\frac{\beta}{2+\beta}} \Big)
+
\mathcal{O}\big(1\big)
+
\mathcal{O}\bigg( \frac{a(t)}{|x-E_t|} \bigg),
\end{align}
where $S_3^{\pm}$ stands for $S_3^{+}$ if $x\geq E_t$ and for $S_3^{-}$ if $x<E_t$.
\item[(3d)] In (Case 3.), when $\beta$ is strictly negative, there exists some positive constant $S_4$ that is independent of $t>1$ and such that when~$a(t)^{2+\beta}=o\big(|x-E_t|^{1+\beta}\big)$, we have
\begin{align}
\rho_t(x) 
=
S_4  
\Bigg( \frac{t \mu(\lbrace E_0 \rbrace)^{1+\beta}}{x-E_t} \Bigg)^{\frac{1}{2+\beta}}
\hspace{-0.5cm}
+
\mathcal{O}\Big( (x-E_t)^{\frac{\beta}{2+\beta}} \Big)
+
\mathcal{O}\big(1\big)
+
\mathcal{O}\bigg( \frac{a(t)}{x-E_t} \bigg),
\end{align}
when $x\geq E_t$ and $\rho_t(x)=0$, when $x<E_t$.
\item[(3e)] In (Case 3.), if $\beta$ is zero and $a(t)^2=o\big({\big|(x-E_t)\log|x-E_t|\big|}\big)$, then there exists some positive constant $S_5$ that is  independent of $t>1$ and such that
\begin{align}
\hspace{-0.05cm} \rho_t(x)\nonumber
&=
S_5\sqrt{\frac{t\mu(\lbrace E_0 \rbrace)|\log(x-E_t)|}{2(x-E_t)}}
\hspace{-0.05cm}
+
\hspace{-0.05cm}
\mathcal{O}\Big( \frac{\big|\log| \log(x-E_t) |\big|}{\sqrt{(x-E_t)|\log(x-E_t)|}} \Big)
\hspace{-0.05cm}
+
\hspace{-0.05cm}
\mathcal{O}\bigg( \frac{a(t)\sqrt{|\log(x-E_t)|}}{x-E_t} \bigg),
\end{align}
when $x\geq E_t$ and $\rho_t(x)=0$, when $x<E_t$.
\end{enumerate}

\end{theorem}

\begin{remark} It follows from our proof that all the statements of Theorem \ref{Theorem 22} remain valid when $a(t)$ is zero and $x\in\mathcal{U}_t$. In that situation, we show in (Case 1.) and (Case 2.) that when $x$ approaches $E_t$, $\rho_t(x)$ diverges either as the inverse of a one- or two-sided square root, or as the inverse of a cubic root. In (Case 3.), we show that $\rho_t$ behaves according to an exponent in $[-\frac{1}{2},0)$, when $\beta$ is strictly positive and according to an exponent in $(-1,-\frac{1}{2})$, when $\beta$ is strictly negative. The case where $\beta$ is zero is treated similarly. Interestingly, we notice in \eqref{non sym1} and \eqref{non sym2} that the density $\rho_t$ does not diverge at the same rate on the left and on the right of $E_t$.
\end{remark}

\subsection{The free additive convolution of two measures}\ \\
We consider two absolutely continuous probability measures on the real line, that satisfy a power law behavior near the endpoints of their respective supports. A study of the qualitative properties of the density of $\mu_{\alpha}\boxplus\mu_{\beta}$ was carried out in \cite{Bao20}, whenever $\mu_\alpha$ and~$\mu_\beta$ are absolutely continuous, supported on a single interval and are of Jacobi type \eqref{Jacobi}. One of our purposes here is to extend these properties to the set-up of multi-cut measures, i.e.\ measures whose densities are supported on several intervals. We consider probability measures of the same form as in \cite{MoSc22}, characterized as follows.

\begin{assumption}\label{main assumption2}
Let $n_{\alpha},\,n_{\beta}\geq 1$ be integers and let $\mu_\alpha$ and $\mu_\beta$ be two absolutely continuous probability measures on the real line such that
\begin{enumerate}[label=(\roman*)]
\item $\mu_\alpha$ is centered and supported on $n_{\alpha}$ disjoint intervals $[A_j^{-},A_j^{+}],\ 1\leq j \leq n_{\alpha}$.
\item $\mu_\beta$ is centered and supported on $n_{\beta}$ disjoint intervals $[B^{-}_j,B^{+}_j],\ 1\leq j \leq n_{\beta}$.
\end{enumerate}
We assume that their respective densities $\rho_\alpha$ and $\rho_\beta$ behave as power laws with exponents strictly between $-1$ and $1$ near their endpoints. More precisely,
\begin{enumerate}[label=(\roman*)]
\item for every $j=1,...,n_\alpha$, there exist exponents $-1<s_j^{-},s_j^{+}<1$ and a constant $C_j\geq 1$ such that
\begin{align}\label{Jacobi2}
C_j^{-1}<\frac{\rho_{\alpha}(x)}{(x-A_j^{-})^{s_j^{-}}(A_j^{+}-x)^{s_j^{+}}}<C_j,\qquad \text{for a.e }x\in[A^{-}_j,A^{+}_j];
\end{align}
\item for every $j=1,...,n_\beta$, there exist exponents $-1<t_j^{-},t_j^{+}<1$ and a constant $\widetilde{C_j}\geq 1$ such that
\begin{align}\label{Jacobi22}
\widetilde{C_j}^{-1}<\frac{\rho_{\beta}(x)}{(x-B_j^{-})^{t_j^{-}}(B_j^{+}-x)^{t_j^{+}}}<\widetilde{C_j},\qquad \text{for a.e }x\in[B_j^{-},B_j^{+}].
\end{align}
\end{enumerate}
\end{assumption}

Under these conditions, $\mu_{\alpha}\boxplus\mu_{\beta}$ is absolutely continuous since the atoms of $\mu_{\alpha}\boxplus\mu_{\beta}$ are determined as follows. A point $w\in\mathbb{R}$ is an atom of $\mu_{\alpha}\boxplus\mu_{\beta}$ if and only if $w=x+y$ with~$\mu_{\alpha}(\lbrace x \rbrace)+\mu_{\beta}(\lbrace y \rbrace)>1$, see e.g.\ Theorem 7.4 in \cite{BeV98}. Moreover the density $\rho_{\alpha\boxplus\beta}$ of the free additive convolution $\mu_{\alpha}\boxplus\mu_{\beta}$ is continuous, real analytic in the interior of its support and Assumption \ref{main assumption2} ensures that it is bounded. As motivated in \cite{Bao20}, \eqref{Jacobi2} and~\eqref{Jacobi22} are natural to assume. They imply that the Cauchy-Stieltjes transform of $\mu_{\alpha\boxplus\beta}$ is locally invertible at the endpoints of the support of $\mu_{\alpha}\boxplus\mu_{\beta}$ and as a consequence, that in the one-cut case $\rho_{\alpha\boxplus\beta}$ decays as a square root at the endpoints of its support. When the exponents in~\eqref{Jacobi2} and \eqref{Jacobi22} are bigger than $1$, other local behaviors have been observed, see e.g.\ ~\cite{BerWangZHong2020,LeSc16}. In Theorem \ref{Theorem 3} below, we prove that \eqref{Jacobi2} and \eqref{Jacobi22} imply that $\rho_{\alpha\boxplus\beta}$ decays at universal rates at its endpoints in the multi-cut case. We denote by $\mathcal{V}_{\alpha\boxplus\beta}$ the set of points where the density~$\rho_{\alpha\boxplus\beta}$ vanishes. More precisely,  
\begin{align*}
\mathcal{V}_{\alpha\boxplus\beta}:=\partial\lbrace x\in\mathbb{R}:\ \rho_{\alpha\boxplus\beta}(x)>0 \rbrace.
\end{align*}
We characterize the points in the set $\mathcal{V}_{\alpha\boxplus\beta}$ in Section \ref{section: free addition definition} and define the notions of \textit{interior} and \textit{exterior points} of $\mathcal{V}_{\alpha\boxplus\beta}$, as in \eqref{exterior point} and \eqref{interior point}.

\begin{theorem} \label{Theorem 3}
Let $\mu_\alpha$ and $\mu_\beta$ be two probability measures that satisfy Assumption \ref{main assumption2},~$E_{0}$ be a real point in $\mathcal{V}_{\alpha\boxplus\beta}$ and $\mathcal{U}_{0}$ denote a small neighborhood of $E_{0}$. 
\begin{enumerate}
\item If~$E_{0}$ is an exterior point of $\mathcal{V}_{\alpha\boxplus\beta}$, then there exists a strictly positive constant~$C_1$ such that for every $x$ in $\mathcal{U}_{0}$, we either have
\begin{align} \label{claim alpha beta 1}
\rho_{\alpha\boxplus\beta}(x)
=
C_1\sqrt{x-E_{0}}+\mathcal{O}\big(x-E_{0}\big),
\end{align}
when $x \geq E_{0}$ and $\rho_{\alpha\boxplus\beta}(x)=0$, when $x<E_{0}$, or we have
\begin{align}\label{claim alpha beta 12}
\rho_{\alpha\boxplus\beta}(x)
=
C_1\sqrt{E_{0}-x}+\mathcal{O}\big(E_{0}-x\big),
\end{align}
when $x \leq E_{0}$ and $\rho_{\alpha\boxplus\beta}(x)=0$, when $x>E_{0}$. 
\item If $E_{0}$ is an interior point of $\mathcal{V}_{\alpha\boxplus\beta}$, then there exists a strictly positive constant~$C_2$  such that for every $x$ in $\mathcal{U}_{0}$, 
\begin{align}\label{claim alpha beta 2}
\rho_{\alpha\boxplus\beta}(x)
=
C_2\sqrt[3]{|x-E_{0}|}+\mathcal{O}\Big( |x-E_{0}|^{\frac{2}{3}} \Big).
\end{align}
\end{enumerate}
\end{theorem}

\begin{remark}
It follows from our proof that the sign of the function $H''$ in \eqref{second} determines whether \eqref{claim alpha beta 1}, \eqref{claim alpha beta 12} or \eqref{claim alpha beta 2} happens. 
\end{remark}

\section{Preliminaries}\label{section: preliminaries}

In this section, we review the main properties of the Cauchy-Stieltjes transform of a probability measure $\mu$ on the real line. We then recall the definition of the free additive convolution semigroup and of the free additive convolution of two distinct probability measures. 
For any finite measure $\mu$ supported on the real line, the Cauchy-Stieltjes transform of $\mu$, denoted by $m_{\mu}$, is the analytic function defined~by
\begin{align}
m_\mu(z):=\int_{\mathbb{R}} \frac{1}{u-z}\,\mu(\mathrm{d}u),\qquad z\in\mathbb{C}^+.
\end{align}

As we will see, it is convenient to consider the negative reciprocal Cauchy-Stieltjes transform of $\mu$, denoted by $F_{\mu}$, in order to study the free additive convolution. We therefore~define
\begin{align}\label{reciprocal}
F_\mu(z):=\frac{-1}{m_{\mu}(z)},\qquad z\in\mathbb{C}^+.
\end{align}

\begin{remark}\label{corresponding measure}
The function $F_{\mu}:\mathbb{C}^+\to \mathbb{C}^+$ is analytic and as ${\eta\to\infty}$, $\frac{F_{\mu}(\mathrm{i}\eta)}{\mathrm{i}\eta}$ converges to~$\mu(\mathbb{R})$. It turns out that any analytic function $F$ verifying these two prior conditions is necessarily the negative reciprocal Cauchy-Stieltjes transform of a unique measure $\nu$ on $\mathbb{R}$. We refer to \cite{Akh65} for a more a detailed exhibition.
\end{remark}

The Cauchy-Stieltjes transform of a measure $\mu$ on the real line encloses the information of its distribution function and in particular of its Lebesgue decomposition.  In order to retrieve such information, we introduce the notion of non-tangential sequences.

\begin{definition}
Let $(z_n)_{n\geq 1}$ be a converging sequence of the upper half-plane such that $\lim_n z_n\in\mathbb{R}$. The sequence $(z_n)_{n\geq 1}$ is called non-tangential if there exist some constants $C>0$ and $N_0\geq 1$ such that
\begin{align}
\Big|\frac{\mathrm{Re}\,z_n-\lim_n z_n}{\mathrm{Im}\,z_n}\Big|\leq C,\qquad\text{for every }n\geq N_0.
\end{align}
\end{definition}

\begin{lemma}\label{singular continuous}
Let $\nu$ be a finite measure on the real line and let $\nu_{\mathrm{ac}}$, $\nu_{\mathrm{pp}}$, and $\nu_{\mathrm{sc}}$ denote the absolutely continuous, atomic and singular continuous parts of $\nu$ respectively. Let $E\in\mathbb{R}$ and let $(z_n)_{n\geq 1}$ be a sequence converging non-tangentially to $E$. Then 
\begin{enumerate}[label=(\roman*)]
\item if $\rho_{\nu}$ denotes the density of $\nu_{\mathrm{ac}}$, $
\rho_{\nu}(E)=\frac{1}{\pi}\lim_{n} \mathrm{Im}\,m_{\nu}(z_n),
$ for a.e.\ $E\in\mathrm{supp}(\nu_{\mathrm{ac}})$,
\item 
$
\nu(\lbrace E \rbrace)= \lim_n(E-z_n)\,m_{\nu}(z_n),
$
\item the non-tangential limit of $m_{\nu}$ at $x$ is infinite for $\nu_{\mathrm{sc}}$-almost every $x\in\mathbb{R}$. 
\end{enumerate}
\end{lemma}
We refer the reader to e.g. \cite{Bel06}, Lemma $1.10$ for more detail. 
Next we enumerate some properties of Nevanlinna functions, i.e.\ analytic functions from $\mathbb{C}^+$ to $\mathbb{C}^+\cup\mathbb{R}$. We will repeatedly use the integral representation given in the following lemma.

\begin{lemma}\label{Pick}
Let $f:\mathbb{C}^+\to\mathbb{C}^+\cup\mathbb{R}$ be an analytic function. Then there exists a measure $\nu$ on the real line, $a\in\mathbb{R}$ and $b\geq 0$ such that 
\begin{align}
f(\omega)
=
a
+
b\omega
+
\int_{\mathbb{R}}\Big(\,\frac{1}{u-\omega}-\frac{u}{1+u^2}\Big)\,\nu(\mathrm{d}u),\qquad \omega\in\mathbb{C}^+.
\end{align}
Moreover
\begin{align}
\int_{\mathbb{R}}\frac{1}{1+u^2}\,\nu(\mathrm{d}u)<\infty.
\end{align}
\end{lemma}

As a consequence, the negative reciprocal Cauchy-Stieltjes transform $F_{\mu}$ admits the following representation, referred to as Nevanlinna representation of $F_{\mu}$. 

\begin{proposition}(Proposition 2.5 in \cite{MoSc22})\label{Nevanlinna representation}
Let $\mu$ be a probability measure satisfying Assumption \ref{main assumption}.  There exists a finite measure $\widehat{\mu}$ on the real line such that
\begin{align} \label{Nevan}
F_{\mu}(\omega)-\omega
= 
\int_{\mathbb{R}} \frac{1}{u-\omega}\,\widehat{\mu}(\mathrm{d}u),\qquad \omega\in\mathbb{C}^+.
\end{align}
Moreover
\begin{align}\label{finiteness}
&0<\widehat{\mu}(\mathbb{R})=\int_{\mathbb{R}} u^2\mu(\mathrm{d}u)<\infty\\
\beforetext{and }\label{support}
&\mathrm{supp}(\widehat{\mu})=\mathrm{supp}(\mu_{\mathrm{ac}})\cup\lbrace y_1,...,y_j \rbrace,
\end{align}
for some collection of points $\lbrace y_1,...,y_j\rbrace\subset\lbrace y\in\mathbb{R}:\ m_{\mu}(y)=0\rbrace$ and $1\leq j\leq n_{\mathrm{ac}}+n_{\mathrm{pp}}^{\mathrm{out}}-1$,
where $n_{\mathrm{pp}}^{\mathrm{out}}$ stands for the number of atoms of $\mu$ that are not in the interior of $\mathrm{supp}(\mu_{\mathrm{ac}})$.
\end{proposition}

\begin{remark}\label{pure point de mu hat}
(1.) Equation \eqref{Nevan} is well-known and holds as long as $\mu$ is not a single point mass, see e.g.\ Proposition~2.2 in \cite{Maa92} or Chapter 3 in~\cite{Min17}.
(2.) In the statement of Proposition~\ref{Nevanlinna representation}, $m_{\mu}(y)$ refers to the~non-tangential extension of $m_\mu$ at $y\in\mathbb{R}$ and we remark that by Assumption \ref{main assumption}, $m_{\mu}$ does not admit zeros in the interior of the support of~$\mu$. Indeed~\eqref{Jacobi2} entails the existence of a strictly positive constant $C$ such that for any $y$ in the interior of~$\mathrm{supp}(\mu)$, 
\begin{align}
\lim_{\eta\searrow 0} \mathrm{Im}\,m_{\mu}(y+\mathrm{i}\eta)
\geq C.
\end{align} 
(3.) We further observe that since the exponent $\beta$ in Assumption \ref{main assumption} is smaller than 1, the atoms of $\widehat{\mu}$ do not lie in the support of $\mu$. This is a consequence of Lemma 3.4 in \cite{Bao17}.\linebreak
(4.) As it was laid out in \cite{MoSc22},  any zero $y$ of $m_{\mu}$ that lies at positive distance from $\mathrm{supp(\mu_{\mathrm{ac}})}$ corresponds to an atom of $\widehat{\mu}$. Indeed, since by \eqref{Nevan} $\widehat{\mu}$ does not admit
a singular continuous part, $m_{\mu}$ extends analytically through $y$ by the Schwarz reflection principle. In addition, since $F_{\mu}$ preserves the upper and lower half-planes, we conclude that $y$ is a simple pole of~$F_{\mu}$ and therefore an atom of $\widehat{\mu}$.
\end{remark}

\section{Free additive convolution semigroup}\label{section: free additive convolution semigroup}
In this section, we recall the definition of the free additive convolution semigroup, discuss its main properties and give the proof of Theorem \ref{Theorem 1}.

\begin{proposition}(\cite{Bel05,Bel04}, equations (2.4) and (2.5) in \cite{Huang15})\label{definition mu t}
Let $t\geq 1$ and $\mu$ be a probability measure supported on the real line. There exists an analytic function $\omega_t:\mathbb{C}^+\to~\mathbb{C}^+$, referred to as subordination function, such that for all $z\in\mathbb{C}^+$,
\begin{align}\label{freeconv1}
&\mathrm{Im}\,\omega_t(z)\geq \mathrm{Im}\,z\text{ and }\frac{\omega_t(\mathrm{i}\eta)}{\mathrm{i}\eta}\to 1\text{ as }\eta\to\infty,\\\label{freeconv}
&t\omega_t(z)-z=(t-1)F_{\mu}\big(\omega_t(z)\big).
\end{align} 
By Remark \ref{corresponding measure}, the function $F(z):=F_{\mu}\big(\omega_t(z)\big)$ corresponds to the negative reciprocal Cauchy-Stieltjes transform of some probability measure $\mu^{\boxplus t}$ supported on the real line. These probability measures form the free additive convolution semigroup~$\big\lbrace \mu^{\boxplus t}:\ t\geq 1 \big\rbrace.$
\end{proposition}

The function $\omega_t$ is defined using the theory of Denjoy-Wolff fixed points, see e.g.\ \cite{Bel Ber, Huang15}. A point $\omega\in\mathbb{C}^+\cup\mathbb{R}\cup\lbrace \infty \rbrace$ is called a Denjoy-Wolff point of an analytic function $f:\mathbb{C}^+\rightarrow \mathbb{C}^+\cup\mathbb{R}$ if one of the following conditions is verified:
\begin{enumerate}
\item $\mathrm{Im}\,\omega>0$ and $f(\omega)=\omega$,
\item $\mathrm{Im}\,\omega=0$ and $\lim_n f(\omega_n)=\omega$ and $\lim_n \big|\frac{f(\omega_n)-\omega}{\omega_n-\omega}\big|\leq 1,$ where $(\omega_n)_{n\geq 1}$ converges non-tangentially to $\omega$.
\end{enumerate}

By the Nevanlinna representation of $F_{\mu}$ in Proposition \ref{Nevanlinna representation}, we observe that for every~$z$ in the upper half-plane,  $\omega_t(z)$ is the Denjoy-Wolff fixed point of the function 
\begin{align}\label{Denjoy wolff function group}
f(z,\omega):=z+(t-1)\int_{\mathbb{R}}\frac{1}{u-\omega}\,\widehat{\mu}(\mathrm{d}u),\qquad \omega\in\mathbb{C}^+. 
\end{align}
By the Denjoy-Wolff Theorem, the fact that $f(z,\cdot)$ maps the upper-half plane to itself for every $z\in\mathbb{C}^+$ guarantees the existence and uniqueness of its fixed point, as long as $\mu$ is not a single point mass. This determines the subordination function $\omega_t(z)$, whenever $z\in\mathbb{C}^+$.

\begin{definition}
Given a finite measure $\mu$ on $\mathbb{R}$, we define the function $I_{\mu}$ by
\begin{align}\label{definition Imu}
I_{\mu}(\omega)
:=
\int_{\mathbb{R}} \frac{1}{|u-\omega|^2}\,\mu(\mathrm{d}u),\qquad \omega\in\mathbb{C}^+\cup\mathbb{R}\backslash{\mathrm{supp}(\mu)}. 
\end{align}
\end{definition}

\begin{remark}\label{comparaisonomegaStieltjes}
Using Proposition \ref{definition mu t}, we observe that
\begin{align}\label{connexionomegat}
\mathrm{Im}\, m_{\mu^{\boxplus t}}(z)
=
\mathrm{Im}\, \omega_t(z)\, I_{\mu}\big(\omega_t(z)\big), \qquad z\in\mathbb{C}^+.
\end{align}
In other words,  using Lemma \ref{singular continuous}, we can  deduce qualitative properties of the density of~$\mu^{\boxplus t}$ from the study of the subordination function $\omega_t$ and of the function~$I_{\mu}$. Combining the Nevanlinna representation of $F_{\mu}$ in Lemma \ref{Nevanlinna representation} with equation \eqref{freeconv}, we see that
\begin{align}
\mathrm{Im}\,\omega_t(z)-\mathrm{Im}\,z
=
(t-1)\,\mathrm{Im}\,\omega_t(z)\, I_{\widehat{\mu}}\big(\omega_t(z)\big),\qquad z\in\mathbb{C}^+.
\end{align}
Moreover, recalling from \eqref{freeconv1} that $\mathrm{Im}\,\omega_t(z)\geq\mathrm{Im}\,z$, we obtain
\begin{align}\label{pluspetit}
I_{\widehat{\mu}}\big(\omega_t(z)\big)
=
\frac{1}{t-1}\bigg(1-\frac{\mathrm{Im}\,z}{\mathrm{Im}\,\omega_t(z)}\bigg)
\leq \frac{1}{t-1}, \qquad z\in\mathbb{C}^+.
\end{align}
Therefore the inequality in \eqref{pluspetit} gives a necessary condition for a point $\omega\in\mathbb{C}^+\cup\mathbb{R}$ to be in the range of $\omega_t$. In Proposition \ref{range omegat} below, we prove that \eqref{pluspetit} is also sufficient for $\omega$ to be in the range of $\omega_t$. The proof of Proposition \ref{range omegat} mainly relies on the Denjoy-Wolff characterisation of $\omega_t$ in \eqref{Denjoy wolff function group}. 
\end{remark}

We now review the main properties of the subordination function  $\omega_t$ and of the free additive convolution semigroup $\big\lbrace \mu^{\boxplus t}:\ t\geq 1 \big\rbrace$.

\begin{proposition}\label{continuous}(\cite{Bel05,Bel04} and Proposition 3.5 in \cite{Huang15})
Let $\mu$ be a compactly supported probability measure on the real line, which is not a single point mass. Then for any $t>1$, the subordination function $\omega_t$ extends continuously to $\mathbb{C}^+\cup\mathbb{R}$ as a function taking values in~$\mathbb{C}^+\cup\mathbb{R}$.
\end{proposition}

Since the subordination function $\omega_t$ naturally extends to the real line, we will simply write $\omega_t$ to refer to its extension. We recall the following consequence of Theorem 4.1 in \cite{Bel1}.

\begin{proposition}\label{compactsupp}
For every $t>1$, $\mu^{\boxplus t}$ admits an absolutely continuous part with respect to the Lebesgue measure and an atomic part. Moreover its density is real analytic wherever it is strictly positive and finite and $\mu^{\boxplus t}$ is compactly supported since the support of $\mu$ is compact.
\end{proposition}

We next review the main properties of the subordination function $\omega_t$, proven in \cite{MoSc22}.

\begin{proposition}\label{general results}
Let $\mu$ satisfy Assumption \ref{main assumption} and $t>1$. The subordination function $\omega_t$ satisfies the following properties.
\begin{enumerate}
\item $\omega_t$ is bounded in every compact set of the extended complex upper-half plane $\mathbb{C}^+\cup\mathbb{R}$.
\item $\omega_t$ stays at positive distance from the zeros of $m_\mu$.
\item $\mathrm{Re}\,\omega_t$ is strictly increasing on the real line and $\omega_t$ is a conformal map defined on $\mathbb{C}^+$ that extends to $\mathbb{C}^+\cup\mathbb{R}$ homeomorphically. 
\end{enumerate}
\end{proposition}

\begin{remark}
We refer the reader to Lemmas 3.1, 3.3 and Proposition 3.4 in \cite{MoSc22} for a  proof and remark that (1) and (3) do not rely on the power law behavior in \eqref{Jacobi}. 
\end{remark}

In Lemma \ref{main results 2} below, we prove that under Assumption \ref{main assumption} the subordination function~$\omega_t$ does not approach the non-atomic points in the support of $\mu$. This result is crucial in proving Theorem \ref{Theorem 1} since our argument is based on a Taylor expansion of $F_{\mu}(\omega)-\omega$.  
We know from Proposition \ref{compactsupp} that $\mu^{\boxplus t}$ is compactly supported. We can therefore consider a compact interval~$\mathcal{E}$ on the real line such that $\mathrm{supp}(\mu^{\boxplus t})\subset~\mathcal{E}$. We further let $E_1,...,E_{n_{\mathrm{pp}}}$ denote the atoms of the measure $\mu$ and fix $\delta_1,...,\delta_{n_{\mathrm{pp}}}>0$ sufficiently small. We define the spectral domain
\begin{align}
\mathcal{J}:=\big\lbrace z=E+\mathrm{i}\eta:\ E\in\mathcal{E}\text{ and }0\leq \eta \leq 1 \big\rbrace\backslash\cup_{j=0}^{n_{\mathrm{pp}}}\big(B_{\delta_j}(tE_j)\big),
\end{align}
where $B_\delta(x)$ stands for the ball of radius $\delta$ centered at $x$. Because of the evident similarities with \cite{MoSc22}, we will omit some proofs in what follows.

\begin{lemma}\label{main results 2}[Proposition 3.5 and Lemma 3.7 in \cite{MoSc22}]
Let $\mu$ be a probability measure satisfying Assumption \ref{main assumption} and let $t>1$.
\begin{enumerate}
\item There exist two constants $C_1,$ $C_2\geq 1$ such that 
\begin{align*}
C_1^{-1}\leq I_{\mu}\big(\omega_t(z)\big)\leq C_1\text{ and } 
C_2^{-1}\leq I_{\widehat{\mu}}\big(\omega_t(z)\big)\leq C_2,\quad z\in\mathcal{J}.
\end{align*}
Moreover there exists a constant $c>0$ such that 
\begin{align}
\inf_{z\in\mathcal{J}}\mathrm{dist}\big(\omega_t(z),\,\mathrm{supp}(\widehat{\mu})\big)\geq c.
\end{align}
\item There exists a constant $C\geq 1$ depending on  $\mathcal{J}$ such that 
\begin{align}
C^{-1}\mathrm{Im}\, m_{\mu^{\boxplus t}}(z) 
\leq 
\mathrm{Im}\, \omega_t(z) 
\leq 
C\, \mathrm{Im}\, m_{\mu^{\boxplus t}}(z),\qquad z\in\mathcal{J}. 
\end{align} 
\end{enumerate}
In other words, as long as $\omega_t$ does not approach an atom of the measure $\mu$, $\mathrm{Im}\,m_{\mu^{\boxplus t}}(z)$ and $\mathrm{Im}\,\omega_t(z)$ are comparable and $\omega_t$ stays at positive distance from the support of $\widehat{\mu}$.
\end{lemma}

\begin{remark}
This lemma extends Proposition 3.5 and Lemma $3.12$ in \cite{Bao20} to the case where~$\mu$ is supported on several intervals and admits an atomic part. Its proof heavily relies on the power law behavior in \eqref{Jacobi} in Assumption \ref{main assumption}. For instance, when $\beta\geq 1$, the subordination function $\omega_t$ does not stay at positive distance from $\mathrm{supp}(\mu_{\mathrm{ac}})$ for small values of $t\geq 1$. The fact that $I_{\widehat{\mu}}\big( \omega_t (z) \big)$ is uniformly bounded from below and above, as long as $z$ lies in $\mathcal{J}$, is a consequence of the Nevanlinna representation in Proposition \ref{Nevanlinna representation} and of equation \eqref{freeconv}. 
\end{remark}

Using the Cauchy-Stieltjes inversion formula, we observe that a real point $x$ lies in the interior of the support of $\mu^{\boxplus t}$ if and only if $\mathrm{Im}\,\omega_t(x)>0$, as long as $\omega_t(x)$ is not an atom of the measure $\mu$. Recalling the fact that $\omega_t$ extends homeomorphically to the real line and the inequality in \eqref{pluspetit}, we see that $I_{\widehat{\mu}}\big( \omega_t(z) \big)\leq\frac{1}{t-1}$ for every $z\in\mathbb{C}^+\cup\mathbb{R}$. As it turns out, the converse is also true (see Lemmas 3.3 and 3.4 in \cite{Huang15}) and is the content of the first claim of the next proposition.

\begin{proposition}[Lemma 3.4 in \cite{Huang15} and Theorem 4.1 in \cite{Bel1}]\label{range omegat}
Let $\mu$ be a probability measure satisfying Assumption \ref{main assumption} and let $t>1$. Every $\omega\in\mathbb{C}^+\cup\mathbb{R}$ such that $I_{\widehat{\mu}}(\omega)\leq \frac{1}{t-1}$ is in the range of $\omega_t$.
Moreover, 
\begin{enumerate}
\item the non-tangential derivative of $\omega_t$ is well-defined at $z_0$ if $I_{\widehat{\mu}}\big(\omega_t(z_0)\big)<\frac{1}{t-1}$,
\item  $\omega_t$ is locally analytic at $z_0$ if $\mathrm{Im}\,\omega_t(z_0)>0$,
\item $\omega_t$ is locally analytic at $z_0$ if $I_{\widehat{\mu}}\big(\omega_t(z_0)\big)<\frac{1}{t-1}$ and  $\omega_t(z_0)$ lies outside of $\mathrm{supp}({\mu_{\mathrm{ac}}})$. 
\end{enumerate}
\end{proposition}

\begin{proof}
To prove the claims of the analyticity of $\omega_t$, we will follow the lines of reasoning of \cite{B} and of Theorem 4.1 in \cite{Bel1}.
First, we address the claim about the range of the subordination function $\omega_t$. We define 
\begin{align}
H(\omega):= \omega-(t-1)\int_{\mathbb{R}}\frac{1}{u-\omega}\,\widehat{\mu}(\mathrm{d}u),\qquad \omega\in\mathbb{C}^+\cup\mathbb{R}\backslash\mathrm{supp}(\widehat{\mu}).
\end{align}
By the Nevanlinna representation in \eqref{Nevan} and by the defining equation of the free convolution semigroup in \eqref{freeconv}, we observe that $H\big(\omega_t(z)\big)=z$ for every $z\in\mathbb{C}^+\cup\mathbb{R}$. We first consider a point $\omega_0\in\mathbb{C}^+$ such that $I_{\widehat{\mu}}(\omega_0)\leq\frac{1}{t-1}$ and notice that 
\begin{align}
\mathrm{Im}\,H(\omega_0)=\big(1 - (t-1) \,I_{\widehat{\mu}}(\omega_0) \big)\mathrm{Im}\,\omega_0
\end{align}
is strictly positive or vanishes if $I_{\widehat{\mu}}(\omega_0)=\frac{1}{t-1}$. 
If we define $z_0:=H(\omega_0)$, then the function $f(z_0,\omega)
:=
z_0+(t-1)\int_{\mathbb{R}}\frac{1}{u-\omega}\,\widehat{\mu}(\mathrm{d}u)$ has positive imaginary part, whenever $\omega\in\mathbb{C}^+$. Indeed  
\begin{align}\label{denjoy}
\mathrm{Im}\,f(z_0,\omega)=\mathrm{Im}\,z_0+(t-1)\,\mathrm{Im}\,\omega\, I_{\widehat{\mu}}\big( \omega \big)>0,\qquad \omega\in\mathbb{C}^+.
\end{align}
Since $f$ is analytic, \eqref{denjoy} guarantees the existence and the uniqueness of a fixed point $\omega_0=f(z_0,\omega_0)$.
This point is by construction $\omega_t(z_0)$ and thus $\omega_t(z_0)=\omega_0$. This proves the claim on the range of the subordination function $\omega_t$, whenever $\omega_0$ is in the upper-half plane.

Next, we consider $\omega_0$ on the real line such that $I_{\widehat{\mu}}(\omega_0)\leq\frac{1}{t-1}$. By Lemma~\ref{main results 2}, such a point lies either outside of the support of $\widehat{\mu}$ or corresponds to an atom of the measure $\mu$. We take a sequence $(\omega_n)_{n\geq 1}$ that converges non-tangentially to $\omega_0$ and look at  $f(z_0,\omega_n)$.
If~$\omega_0$ is not in the support of $\widehat{\mu}$, then $f(z_0,\omega_n)$ converges to $f(z_0,\omega_0)$ as~$n$ tends to infinity, by the bounded convergence theorem. If $\omega_0$ corresponds to an atom of $\mu$, then by Lemma \ref{singular continuous} and by the Nevanlinna representation in \eqref{Nevan}, the function $H(\omega)$ extends non-tangentially to $\omega_0$, $(\omega_0-\omega_n)m_{\mu}(\omega_n)$ converges to $\mu\big(\lbrace \omega_0 \rbrace\big)$ and, using $F_{\mu}(\omega_0)=0$,
\begin{align}\label{extension}
f(z_0,\omega_n)-f(z_0,\omega_0)
&=\nonumber
(t-1)\Big(\int_{\mathbb{R}}\frac{1}{u-\omega_n}\,\widehat{\mu}(\mathrm{d}u)-\int_{\mathbb{R}}\frac{1}{u-\omega_0}\,\widehat{\mu}(\mathrm{d}u)\Big)\\
&=\nonumber
(t-1)\Big(F_{\mu}(\omega_n)-F_{\mu}(\omega_0)-\big(\omega_n-\omega_0\big)\Big)\\
&=
(t-1)\big(\omega_n-\omega_0\big)\bigg(\frac{1}{(\omega_0-\omega_n)\,m_{\mu}(\omega_n)}-1\bigg),
\end{align}
tends to zero as $n$ tends to infinity. Therefore, $\omega_0= \lim_n f(z_0,\omega_n)$. Moreover, since by assumption $I_{\widehat{\mu}}(\omega_0)\leq\frac{1}{t-1}$, we see that
\begin{align}
\lim_n \Big|\frac{f(z_0,\omega_n)-f(z_0,\omega_0)}{\omega_n-\omega_0}\Big|\leq 1.
\end{align}
Therefore $\omega_0$ is the unique Denjoy-Wolff fixed point of $f(z_0,\omega)$ and thence corresponds to~$\omega_t(z_0)$. This concludes the proof of the claim on the range of $\omega_t$.

We move on to proving the remaining three claims.  We recall that 
\begin{align}
I_{\widehat{\mu}}\big( \omega_t(z) \big)=\frac{1}{t-1}\bigg(1-\frac{\mathrm{Im}\,z}{\mathrm{Im}\,\omega_t(z)} \bigg).
\end{align} 
Consequently, if $\omega_0=\omega_t(z_0)\in\mathbb{C}^+$, then $I_{\widehat{\mu}}(\omega_0)<\frac{1}{t-1}$ and  
\begin{align}
\left.\frac{\partial}{\partial\omega}\right|_{\omega=\omega_0} \Big(f(z_0,\omega)-\omega\Big)
= 
(t-1)\int_{\mathbb{R}}\frac{1}{(u-\omega_0)^2}\,\widehat{\mu}(\mathrm{d}u)-1<0,
\end{align}
which proves that $\omega_t$ is analytic in an open neighborhood of $z_0$, by the analytic inverse function theorem. This proves claim (2).

If $I_{\widehat{\mu}}( \omega_0 )<\frac{1}{t-1}$ and $\omega_0$  lies outside of the support of ${\mu}_{\mathrm{ac}}$, then $\omega_0$ lies outside of the support of $\widehat{\mu}$ by  Propositions \ref{Nevanlinna representation} and \ref{general results}. Thus $f(z_0,\omega)-\omega$ is analytic in a neighborhood of $\omega_0$ and  
\begin{align}
\left.\frac{\partial}{\partial\omega}\right|_{\omega=\omega_0} \Big(f(z_0,\omega)-\omega\Big)
= 
(t-1)\int_{\mathbb{R}}\frac{1}{(u-\omega_0)^2}\,\widehat{\mu}(\mathrm{d}u)-1<0
\end{align}
guarantees the analyticity of $\omega_t$ in a neighborhood of $z_0$. This proves claim (3).

If we next assume that $I_{\widehat{\mu}}( \omega_0 )<\frac{1}{t-1}$ and that $\omega_0$ lies in the support of ${\mu_{\mathrm{ac}}}$, then $\omega_0$ is an atom of $\mu$ by Lemma \ref{main results 2} and by \eqref{extension}, $f(z_0,\omega)$ extends non-tangentially to $\omega_0$. Using \eqref{extension} and the fact that $\widehat{\mu}$ is not a simple point mass, we then see that for any sequence~$(\omega_n)_{n\geq 1}$ that converges non-tangentially to $\omega_0$, 
\begin{align}
\lim_n \frac{f(z_0,\omega_n)-\omega_n - f(z_0,\omega_0) + \omega_0}{\omega_n-\omega_0}
= 
(t-1)\int_{\mathbb{R}}\frac{1}{(u-\omega_0)^2}\,\widehat{\mu}(\mathrm{d}u)-1<0.
\end{align}
The non-tangential derivative of $\omega_t$ is therefore well-defined at $z_0$. This proves (1) and concludes the proof of this proposition.
\end{proof}

Using Proposition \ref{range omegat}, we will determine when an atom $E$ of $\mu$ lies in the range of the subordination function $\omega_t$. We observe that for any sequence $(\omega_n)_{n\geq 1}$ that converges non-tangentially to $E$, the Nevanlinna representation in Proposition \ref{Nevanlinna representation} and Lemma \ref{singular continuous}~entail
\begin{align}\label{limiteImuhat} 
\lim_n I_{\widehat{\mu}}(\omega_n)
&=
\lim_n\frac{\int_{\mathbb{R}}\frac{1}{|u-\omega_n|^2}\,\mu(\mathrm{d}u)}{\big|\int_{\mathbb{R}}\frac{1}{u-\omega_n}\,\mu(\mathrm{d}u)\big|^2}-1
=
\frac{1}{\mu(\lbrace E \rbrace)}-1,
\end{align}
where the last equality follows from the fact that~$(\omega_n)_{n\geq 1}$ is a non-tangential sequence and the bounded convergence theorem. We conclude by Proposition \ref{range omegat} that $E$ belongs to the range of $\omega_t$ if and only if $\mu(\lbrace E \rbrace)\geq 1-\frac{1}{t}$. We also recall from \cite{Bel04} that $tE$ is an atom of $\mu^{\boxplus t}$ if and only if $\mu(\lbrace E \rbrace) > 1-\frac{1}{t}$. 
We next consider the endpoints of the support of~$\mu^{\boxplus t}$ and refer to  \cite{Huang15} for a detailed exhibition.

\begin{proposition}(\cite{Huang15}, Proposition 4.3 in \cite{Bao20} and Proposition 3.15 in \cite{MoSc22})\label{edgecaract}
Let $\mu$ be a probability measure on the real line that satisfies Assumption \ref{main assumption}. For any $z\in\mathbb{C}^+\cup\mathbb{R}$, we have the inequality 
\begin{align}\label{edge}
\big|F'_{\mu}\big(\omega_t(z)\big)-1\big|\leq \frac{1}{t-1}.
\end{align}
Moreover equality is achieved for $z=tE+\mathrm{i}\eta$ if and only if one of the following occurs:
\begin{enumerate}
\item $E$ is not an atom of $\mu$, $\rho_t$ vanishes at $tE\in\mathcal{V}_t$ and $\eta=0$,  
\item $E$ is an atom of $\mu$ such that $\mu(\lbrace E \rbrace)=1-\frac{1}{t}$ and  $\eta=0$.
\end{enumerate}
In fact we have that for any such  $z$, $F'_{\mu}\big(\omega_t(z)\big)=\frac{t}{t-1}$. 
\end{proposition}

\begin{remark}
(1.) The quantity $F_{\mu}'\big(\omega_t(z)\big)$ denotes the Julia-Carathéodory derivative of~$F_{\mu}$, whenever $\omega_t(z)$ lies in the support of $\mu$. (2.) The proof of \eqref{edge} does not rely on the local behavior in~\eqref{Jacobi}. However the claim on when equality is achieved in \eqref{edge} does. In general, when $\beta\geq 1$, a point $E$ at which $\rho_t$ vanishes no longer verifies~$F'_{\mu}\big(\omega_t(E)\big)=~\frac{t}{t-1}$.
\end{remark}

In summary, we have proven the following. 
\begin{enumerate}[leftmargin=0.8cm]
\item If $z$ is at positive distance from the support of $\mu_{\mathrm{ac}}^{\boxplus t}$, then $I_{\widehat{\mu}}\big(\omega_t(z)\big)<\frac{1}{t-1}$. 
\item If $z$ is in the support of $\mu_{\mathrm{ac}}^{\boxplus t}$ but is not an atom of $\mu^{\boxplus t}$, then $I_{\widehat{\mu}}\big(\omega_t(z)\big)=\frac{1}{t-1}$.
\end{enumerate}
Moreover at any non-atomic endpoint $E_t$ of the support of $\mu^{\boxplus t}_{\mathrm{ac}}$, $F'_{\mu}\big(\omega_t(E_t)\big)=\frac{t}{t-1}$. We now turn to the proof of Theorem~\ref{Theorem 1}.
\begin{proof}[Proof of Theorem \ref{Theorem 1}]
First we fix $t>1$ and define the function 
\begin{align}
H(\omega):= \omega-(t-1)\int_{\mathbb{R}}\frac{1}{u-\omega}\,\widehat{\mu}(\mathrm{d}u),\qquad \omega\in\mathbb{C}^+\cup\mathbb{R}\backslash\mathrm{supp}(\widehat{\mu}).
\end{align}
The function $H$ is the inverse of the subordination function $\omega_t$, since by  \eqref{Nevan} and \eqref{freeconv},
\begin{align}\label{inverse}
H\big(\omega_t(z)\big)
=
\omega_t(z)-(t-1)\int_{\mathbb{R}}\frac{1}{u-\omega_t(z)}\,\widehat{\mu}(\mathrm{d}u)
=
z.
\end{align}

Let $E_t\in\mathcal{R}_t$ be a point in $\mathcal{V}_t$. 
By \eqref{limiteImuhat}, since $\omega_t$ does not approach any atom of $\mu$ with mass smaller than $1-\frac{1}{t}$, we can assume that $E_0:=\frac{E_t}{t}$ is not an atom of $\mu$. By Lemma~\ref{main results 2}, $E_0$ lies at positive distance from $\mathrm{supp}(\widehat{\mu})$. Moreover by \eqref{connexionomegat} and Proposition~\ref{general results}, $\rho_t$ necessarily vanishes at $E_t$. We set  $w_{t}:=\omega_t(E_t)$ and remark that the function $H$ is analytic in some  neighborhood of $w_t$ and for any $\omega$ in the range of $\omega_t$,
\begin{align*}
H(\omega)
=
E_t
+
H'(w_t)(\omega-w_t)
+
\frac{H''(w_t)}{2}(\omega-w_t)^2
+
\frac{H'''(w_t)}{6}(\omega-w_t)^3
+ \mathcal{O}\big((\omega-w_t)^4\big).
\end{align*}
Moreover, by the Nevanlinna representation in \eqref{Nevan},
\begin{align}
H'(w_t)
&=
1-(t-1)\int_{\mathbb{R}}\frac{1}{(u-w_t)^2}\,\widehat{\mu}(\mathrm{d}u),\\ \label{equation 2 0}
H''(w_t)
&=
-2(t-1)\int_{\mathbb{R}} \frac{1}{(u-w_t)^3}\,\widehat{\mu}(\mathrm{d}u),\\ \label{equation 3 0}
H'''(w_t)
&=
-6(t-1)\int_{\mathbb{R}} \frac{1}{(u-w_t)^4}\,\widehat{\mu}(\mathrm{d}u),
\end{align}
and by Proposition \ref{edgecaract}, the fact that $\rho_t$ vanishes at $E_t$ entails 
\begin{align}
\int_{\mathbb{R}}\frac{1}{(u-w_t)^2}\,\widehat{\mu}(\mathrm{d}u)=F'_{\mu}(w_t)-1=\frac{1}{t-1}.
\end{align}
Therefore $H'(w_t)=0$ and by \eqref{inverse}, we see that $\omega_t(z)$ verifies
\begin{align}\label{equation omega loin support}
z-E_t
=
\frac{H''(w_t)}{2} \big(\omega_t(z)-w_t\big)^2
+
\frac{H'''(w_t)}{6} \big(\omega_t(z)-w_t\big)^3
+
\mathcal{O}\Big( \big(\omega_t(z)-w_t\big)^4 \Big).
\end{align}

By equations \eqref{equation 2 0} and \eqref{equation 3 0}, the function $F_{\mu}'-1$ is strictly convex on each bounded interval in~$\mathbb{R}\backslash\mathrm{supp}(\widehat{\mu})$. We consider the interval of  $\mathbb{R}\backslash\mathrm{supp}(\widehat{\mu})$ in which~$w_t$ lies. If $E_t$ is not the leftmost or rightmost point in the support of $\mu^{\boxplus t}$, then the convexity of $F_{\mu}'-1$ implies the existence of a critical value $t^{*}>1$ and a point $E_*\in\mathbb{R}$ such that 
\begin{align}\label{definition omega0}
1-(t^{*}-1)\int_{\mathbb{R}}\frac{1}{\big(u-\omega_{t^*}(E_*)\big)^2}\,\widehat{\mu}(\mathrm{d}u)=0,\quad
\int_{\mathbb{R}}\frac{1}{\big(u-\omega_{t^*}(E_*)\big)^3}\,\widehat{\mu}(\mathrm{d}u)=0.
\end{align}
Moreover $w_t$ and $\omega_*:=\omega_{t^*}(E_{*})$ lie in the same interval in $\mathbb{R}\backslash\mathrm{supp}(\widehat{\mu})$ and as $t\leq t^{*}$ converges to~$t^{*}$, $w_t$ converges to $\omega_*$. We will determine the solution~$\omega_t(z)$ of~\eqref{equation omega loin support} when $z$ is close to~$E_*$ and $t\leq t^{*}$ is close to $t^{*}$. We set 
\begin{align}\label{definition C1}
C_1(t)&:=-(t-1)\int_{\mathbb{R}} \frac{1}{(u-w_t)^3}\,\widehat{\mu}(\mathrm{d}u)
\end{align}
and first look at the regime where $| z-E_t |=o\big( |C_1(t)|^3 \big)$. 
We will derive asymptotics for the solution to \eqref{equation omega loin support} by performing an asymptotic analysis of Cardano's formula. We define
\begin{align}\label{z tilde}
\widetilde{z}-\widetilde{E_t}
:=
27\, \frac{C_2(t)^2}{C_1(t)^3}\, \Big(z-E_t+\mathcal{O}\big((\omega_t(z)-w_t)^4\big)\Big), 
\end{align} 
where
\begin{align}
C_2(t)&:=-(t-1)\int_{\mathbb{R}} \frac{1}{(u-w_t)^4}\,\widehat{\mu}(\mathrm{d}u),
\end{align}
and the order term in \eqref{z tilde} corresponds to the one in \eqref{equation omega loin support}. We first remark that as $z$ approaches $E_t$, the quantity $\widetilde{z}-\widetilde{E_t}$ tends to zero. Moreover, equation \eqref{equation omega loin support} is a perturbed cubic equation. Therefore, using Cardano's formula, we see that 
\begin{align}\label{solution omega regimes}
\omega_t(z)-w_t
=
-\frac{1}{3C_2(t)}\Big(C_1(t)+\Delta(t,z)+\frac{C_1(t)^2}{\Delta(t,z)}\Big),
\end{align}
where
\begin{align}\label{Delta}
\Delta(t,z) \nonumber
&:=
\sqrt[3]{\frac{2C_1(t)^3-C_1(t)^3\big(\widetilde{z}-\widetilde{E_t})+\sqrt{\big(2C_1(t)^3-C_1(t)^3\big(\widetilde{z}-\widetilde{E_t})\big)^2 -4C_1(t)^6}}{2}}\\
&=
\sqrt[3]{C_1(t)^3}\,\sqrt[3]{\frac{2-(\widetilde{z}-\widetilde{E_t})+\sqrt{-(\widetilde{z}-\widetilde{E_t})}\sqrt{4-(\widetilde{z}-\widetilde{E_t})}}{2}}.
\end{align}
The square and cubic roots are determined by the continuity of $\omega_t$ and the fact that $\omega_t$ maps the upper-half plane to itself. By Taylor expanding \eqref{Delta}, we get
\begin{align}\label{Delta regime 1}
\Delta(t,z)
=
\sqrt[3]{C_1(t)^3}
\bigg(
1 
+ \frac{\sqrt{-(\widetilde{z}-\widetilde{E_t})}}{3}
- \frac{\widetilde{z}-\widetilde{E_t}}{18}
+ \mathcal{O}\Big( (\widetilde{z}-\widetilde{E_t})^{\frac{3}{2}}\Big)
\bigg).
\end{align}
Plugging \eqref{Delta regime 1} into \eqref{solution omega regimes}, we obtain
\begin{align}\label{solution z tilde 1}
\omega_t(z)-w_t
=
\frac{1}{3 \sqrt{3}}\frac{C_1(t)}{ C_2(t)}\sqrt{\widetilde{z}-\widetilde{E_t}}
-\frac{1}{54}\frac{C_1(t)}{C_2(t)}\big(\widetilde{z}-\widetilde{E_t} \big)
+
\mathcal{O}\Big(C_1(t) \big(\widetilde{z}-\widetilde{E_t}\big)^{\frac{3}{2}} \Big).
\end{align}
Therefore, by combining \eqref{z tilde} and \eqref{solution z tilde 1}, we conclude that
\begin{align}
\widetilde{z}-\widetilde{E_t}
=
27 \frac{C_2(t)^2}{C_1(t)^3}(z-E_{t}) 
+ 
\mathcal{O}\bigg(\frac{(z-E_t)^2}{C_1(t)^5}\bigg),
\end{align}
and 
\begin{align}\label{solution finale omega 1}
\omega_t(z)-w_t
=
\sqrt{\frac{z-E_t}{C_1(t)}}
-\frac{1}{2}\frac{C_2(t)}{C_1(t)^2}(z-E_t)
+\mathcal{O}\bigg(\frac{(z-E_t)^{\frac{3}{2}}}{C_1(t)^{\frac{7}{2}}}\bigg),
\end{align}
where the square root is defined as the principal branch since the subordination function $\omega_t$ maps the upper-half plane to itself. We have therefore determined the asymptotics of $\omega_t(z)$ in the regime where $| z-E_t |=o\big( |C_1(t)|^3\big)$.

Next, we consider the regime where $|C_1(t)|^3=o\big( |z-E_t| \big)$ and perform a similar asymptotic analysis. We redefine
\begin{align}\label{tilde z 2}
\widetilde{z}-\widetilde{E_t}
:=
27\,C_2(t)^2 \bigg(z-E_t + \mathcal{O}\Big(\big(\omega_t(z)-w_t\big)^4\Big)\bigg)\ 
\text{ and }\ 
\widetilde{C}_1(t,z)
:=
\frac{C_1(t)}{\big|\sqrt[3]{\widetilde{z}-\widetilde{E_t}}\big|},
\end{align}
where the order term corresponds to the one in \eqref{equation omega loin support}. We remark that $\widetilde{C}_1(t,z)$ tends to zero as $z$ approaches $E_t$ and $t$ approaches $t^*$ in the considered regime. By Cardano's formula, 
\begin{align}
\Delta(t,z)
=
\sqrt[3]{\widetilde{C}_1(t,z)^3|\widetilde{z}-\widetilde{E_t}|-\frac{\widetilde{z}-\widetilde{E_t}}{2}-\frac{\widetilde{z}-\widetilde{E_t}}{2}\sqrt{1-4\,\widetilde{C}_1(t,z)^3\frac{|\widetilde{z}-\widetilde{E_t}|}{\widetilde{z}-\widetilde{E_t}}}},
\end{align}
and by expanding the square and cubic roots, we obtain
\begin{align}\label{delta second regime}
\Delta(t,z)
=
\sqrt[3]{-(\widetilde{z}-\widetilde{E_t})}
\bigg(1-\frac{2}{3}\widetilde{C}_1(t,z)^3\frac{|\widetilde{z}-\widetilde{E_t}|}{\widetilde{z}-\widetilde{E_t}}+\mathcal{O}\Big(\widetilde{C}_1(t,z)^6\Big)\bigg).
\end{align}
Combining \eqref{solution omega regimes}, \eqref{tilde z 2} and \eqref{delta second regime} then gives
\begin{align}
\widetilde{z}-\widetilde{E_t}=27\,C_2(t)^2(z-E_t)+\mathcal{O}\Big((z-E_t)^{\frac{4}{3}}\Big),
\end{align}
and 
\begin{align}\label{solution finale omega 2}
\omega_t(z)-w_t
=
\sqrt[3]{\frac{z-E_t}{C_2(t)}}
- 
\frac{1}{3}\frac{C_1(t)}{C_2(t)}
+
\mathcal{O}\bigg(\frac{C_1(t)^2}{\sqrt[3]{z-E_t}}\bigg)
+
\mathcal{O}\Big((z-E_t)^{\frac{2}{3}}\Big),
\end{align}
where the cubic root is determined by the continuity of $\omega_t$ and the fact that $\omega_t$ is a self-mapping of the upper-half plane. We have therefore obtained the asymptotics of $\omega_t(z)$ in the regime where $|C_1(t)|^3=o\big( | z-E_t| \big)$.

In order to complete our asymptotic analysis of $\omega_t(z)$, we next study the asymptotics of~$C_1(t)$ as $t$ approaches $t^*$. We recall the notations $w_t:=\omega_{t}(E_t)$ and $\omega_*:=\omega_{t^*}(E_*)$. By definition of $\omega_*$ and by Taylor expanding $C_1(t)$, we see that
\begin{align}\label{C1 etape 1}
C_1(t)
=
-3(t-1)\int_{\mathbb{R}}\frac{1}{(u-\omega_{*})^4}\,\widehat{\mu}(\mathrm{d}u)(w_t-\omega_{*}) + \mathcal{O}\big((w_t-\omega_{*})^2\big).
\end{align}
Moreover by definition of $\omega_*$ and by Proposition \ref{edgecaract},
\begin{align}
\frac{t}{t-1}
&= \nonumber
F_{\mu}'(w_t)
=
F_{\mu}'(\omega_*)+F_{\mu}''(\omega_*)(w_t-\omega_*)
+ \frac{1}{2} F_{\mu}'''(\omega_*)(w_t-\omega_*)^2
+ \mathcal{O}\big((w_t-\omega_*)^3\big)\\
&=
\frac{t^{*}}{t^{*}-1}+ \frac{1}{2} F_{\mu}'''(\omega_*)(w_t-\omega_*)^2
+ \mathcal{O}\big((w_t-\omega_*)^3\big),
\end{align}
and thus solving for $w_t-\omega_*$, we get
\begin{align}\label{ecart omega}
w_t-\omega_*
=
\sqrt{\frac{2}{F_{\mu}'''(\omega_*)}}\sqrt{\frac{t^{*}-t}{(t^*-1)(t-1)}}
+ \mathcal{O}\big((t^{*}-t)^\frac{3}{2}\big),
\end{align}
where the square root is the principal branch if $w_t>\omega_*$ and the other branch otherwise.~Thus
\begin{align}\label{C1 final}
C_1(t)
=
-\sqrt{3\,\frac{(t^*-t)(t-1)}{(t^*-1)}}\sqrt{\int_{\mathbb{R}}\frac{1}{(u-\omega_*)^4}\,\widehat{\mu}(\mathrm{d}u)}
+
\mathcal{O}( t^*-t ).
\end{align}

Plugging \eqref{C1 final} in \eqref{solution finale omega 1} and \eqref{solution finale omega 2} concludes our study of the asymptotics of $\omega_t(z)$, when~$z$ is close to $E_t$ and $t$ is close to $t^*$.
In order to determine the asymptotics of $\rho_t(x)$, as~$x$ approaches $E_t$ and $1<t<t^{*}$, we take a look at $I_{\mu}\big(\omega_t(z)\big)$. We recall that 
\begin{align}
\mathrm{Im}\, m_{\mu^{\boxplus t}}(z)
=
\mathrm{Im}\, \omega_t(z)\, I_{\mu}\big(\omega_t(z)\big), \qquad z\in\mathbb{C}^+,
\end{align}
and since $w_t$ is at positive distance from the support of $\widehat{\mu}$ by Lemma \ref{main results 2}, there exists a neighborhood $\mathcal{U}_t$ of $E_t$ such that for every $z\in\mathcal{U}_t$,
\begin{align}\label{asymptotics I hat 1}
I_{\mu}\big(\omega_t(z)\big)
&=
\int_{\mathbb{R}}\frac{1}{(u-w_t)^2}\,\mu(\mathrm{d}u)
+ 
\mathcal{O}\big( \omega_t(z)-w_t \big)\\ \label{asymptotics I hat 2}
&=
\int_{\mathbb{R}}\frac{1}{(u-\omega_{*})^2}\,\mu(\mathrm{d}u)
+ 
\mathcal{O}\big(\sqrt{t^*-t}\big)
+ 
\mathcal{O}\big( \omega_t(z)-w_t \big).
\end{align}
Therefore combining \eqref{solution finale omega 1}, \eqref{C1 final} and \eqref{asymptotics I hat 1}, we see that if $t<t^{*}$ and $C_1(t)<0$, then 
\begin{align}\label{omega solution finale 1}
\mathrm{Im}\,m_{\mu^{\boxplus t}}(x)
=
\frac{\int_{\mathbb{R}}\frac{1}{(u-w_t)^2}\,\mu(\mathrm{d}u) }{\sqrt{(t-1) \big|\int_{\mathbb{R}}\frac{1}{(u-w_t)^3}\,\widehat{\mu}(\mathrm{d}u)}\big|}\sqrt{x-E_t}
+
\mathcal{O}\big(x-E_t\big),\qquad x\geq E_t,
\end{align}
and $\mathrm{Im}\,m_{\mu^{\boxplus t}}(x)=0$, when $x<E_t$.  Identically if $1<t<t^*$ is fixed and $C_1(t)>0$, then 
\begin{align}\label{omega solution finale 2}
\mathrm{Im}\,m_{\mu^{\boxplus t}}(x)
=
\frac{ \int_{\mathbb{R}}\frac{1}{(u-w_t)^2}\,\mu(\mathrm{d}u) }{\sqrt{(t-1) \big|\int_{\mathbb{R}}\frac{1}{(u-w_t)^3}\,\widehat{\mu}(\mathrm{d}u)}\big|}\sqrt{E_t-x}
+
\mathcal{O}\big(E_t-x \big),\qquad x\leq E_t,
\end{align}
and $\mathrm{Im}\,m_{\mu^{\boxplus t}}(x)=0$, when $x>E_t$. This concludes the proof of \eqref{statement 1 2.2} and \eqref{statement 2 2.2}.

If we next consider $1<t<t^*$ converging to $t^*$ in a way such that $|x-E_t|^2=o\big( (t^*-t)^{3}\big)$ and if $C_1(t)<0$ as $t<t^{*}$ converges to $t^*$, then \eqref{solution finale omega 1}, \eqref{C1 final} and \eqref{asymptotics I hat 2} yield
\begin{align}\label{omega solution finale 3}
\mathrm{Im}\,m_{\mu^{\boxplus t}}(x)\nonumber
&=
 \sqrt[4]{\frac{t^*-1}{3(t^*-t)(t-1)}}\frac{ \int_{\mathbb{R}}\frac{1}{(u-\omega_*)^2}\,\mu(\mathrm{d}u)}{\sqrt[4]{\int_{\mathbb{R}}\frac{1}{(u-\omega_*)^4}\,\widehat{\mu}(\mathrm{d}u)}}
{\sqrt{x-E_t}}
+
\mathcal{O}\Big(\frac{x-E_t}{\sqrt{t^*-t}} \Big)\\
&+
\mathcal{O}\big(\sqrt{x-E_t}\sqrt[4]{t^*-t}\big),
\end{align}
when $x\geq E_t$ and $\mathrm{Im}\,m_{\mu^{\boxplus t}}(x)=0$, when $x<E_t$. Conversely, if $C_1(t)>0$ as $t<t^*$ converges to $t^*$, we obtain 
\begin{align}\label{omega solution finale 4}
\mathrm{Im}\,m_{\mu^{\boxplus t}}(x)\nonumber
&=
 \sqrt[4]{\frac{t^*-1}{3(t^*-t)(t-1)}}\frac{ \int_{\mathbb{R}}\frac{1}{(u-\omega_*)^2}\,\mu(\mathrm{d}u)}{\sqrt[4]{\int_{\mathbb{R}}\frac{1}{(u-\omega_*)^4}\,\widehat{\mu}(\mathrm{d}u)}}
{\sqrt{E_t-x}} 
+
\mathcal{O}\Big(\frac{E_t-x}{\sqrt{t^*-t}} \Big)\\
&+
\mathcal{O}\big(\sqrt{E_t-x}\sqrt[4]{t^*-t}\big),
\end{align}
when $x\leq E_t$ and $\mathrm{Im}\,m_{\mu^{\boxplus t}}(x)=0$, when $x>E_t$. The proof of \eqref{statement 3 2.2} and \eqref{statement 4 2.2} then follows from the Cauchy-Stieltjes inversion formula.
We now consider the regime where $(t^*-t)^{3}=o\big( |x-E_t|^2 \big)$. Combining \eqref{solution finale omega 2}, \eqref{C1 final} and \eqref{asymptotics I hat 2}, we obtain 
\begin{align}\label{solution omega finale 3}
\mathrm{Im}\,m_{\mu^{\boxplus t}}(x)\nonumber
&=
\frac{\sqrt{3}}{2}\frac{ \int_{\mathbb{R}}\frac{1}{(u-\omega_*)^2}\,\mu(\mathrm{d}u) }{\sqrt[3]{(t^*-1)\int_{\mathbb{R}}\frac{1}{(u-\omega_*)^4}\,\widehat{\mu}(\mathrm{d}u)}}
{\sqrt[3]{|x-E_t|}}
+ 
\mathcal{O}\Big( | x-E_t|^{\frac{2}{3}}\Big)\\
&+
\mathcal{O}\big( \sqrt[3]{|x-E_t|} \sqrt{t^*-t}  \big),
\end{align}
for every $x$ in an open neighborhood $\mathcal{U}_t$ of $E_t$ and where $E_t=E_*$ if $t=t^*$. 
Equation~\eqref{statement 5 2.2} then follows from the Cauchy-Stieltjes inversion formula. This proves Theorem~\ref{Theorem 1}.
\end{proof}

We next turn to the proof of Theorem \ref{Theorem 4}.

\begin{proof}[Proof of Theorem \ref{Theorem 4}]
Let $\mathcal{K}$ be an interval in $\mathbb{R}\backslash\mathrm{supp}(\mu_{\mathrm{ac}})$, $t^*>1$ and $E_*\in\mathcal{K}$ be such that 
(1.) $E_*$ is not an atom of $\mu$,
(2.) $E_*$ is an interior point of $\mathcal{V}_{t^*}$. 
We recall from \eqref{definition omega0} that $\omega_*:=\omega_{t^*}(E^*)$ verifies
\begin{align}
\int_{\mathbb{R}}\frac{1}{(u-\omega_*)^3}\,\widehat{\mu}(\mathrm{d}u)=0.
\end{align}
We consider the function 
\begin{align}
f(x,y):=\int_{\mathbb{R}}\frac{u-x}{\big((u-x)^2+y^2\big)^2}\,\widehat{\mu}(\mathrm{d}u).
\end{align}
$f$ is harmonic on the positive half-plane and therefore the set of its zeros forms a continuous curve $\gamma$. By definition, $f(\omega_*,0)=0$, so the point $(\omega_*,0)$ belongs to $\gamma$. Moreover
\begin{align}\label{derivative Re}
\frac{\partial}{\partial x}\, f(x,y)
=
\int_{\mathbb{R}}\frac{3(u-x)^2-y^2}{\big((u-x)^2+y^2\big)^3}\,\widehat{\mu}(\mathrm{d}u)
\end{align}  
is strictly positive when $y$ is sufficiently small and
\begin{align}\label{derivative Im}
\frac{\partial}{\partial y}\, f(x,y)
=
-4 y
\int_{\mathbb{R}}  \frac{u-x}{\big((u-x)^2+y^2\big)^3}\,\widehat{\mu}(\mathrm{d}u)
\end{align}
vanishes as $y$ tends to zero. Combining \eqref{derivative Re} and \eqref{derivative Im}, we get for any ${\omega}=(\omega_1,\omega_2)\in\gamma$,
\begin{align*}
0
=
f(\omega_1,\omega_2)
=
\Big(\left.\frac{\partial}{\partial x}\right|_{x=\omega_*, y=0}\, \hspace{-0.6cm}f(x,y)\Big) (\omega_1-\omega_*)
+
\frac{1}{2}\Big(\left.\frac{\partial^2}{\partial y^2}\right|_{x=\omega_*,y=0}\, \hspace{-0.6cm}f(x,y)\Big) {\omega_2}^2
+
\mathcal{O}\big({\omega_2}^3\big),
\end{align*}
and thence $|\omega_1-\omega_*|=\mathcal{O}\big({\omega_2}^2\big)$, when $\omega\in\gamma$ is sufficiently close to $\omega_*$.

Next, it follows from \eqref{solution finale omega 2} that when $t=t^*$,  $\mathrm{Im}\,\omega_{t^*}(z)$ behaves as a two-sided cubic root, when $z$ is close to $E_*$. Moreover by taking the real and imaginary parts of the cubic root in equation \eqref{solution finale omega 2}, we get
\begin{align}\label{range cube}
\mathrm{Im}\,\omega_{t^*}(z)\geq \sqrt{3}\big|\mathrm{Re}\,\omega_{t^*}(z)-\omega_*\big|,
\end{align}   
when $z$ is close to $E_*$ and equality is achieved when $z$ is real. Since any $\omega=(\omega_1,\omega_2)\in\gamma$ sufficiently close to $(\omega_*,0)$ verifies $|\omega_1-\omega_*|=\mathcal{O}\big( {\omega_2}^2 \big)$, we conclude that $\omega_1+\mathrm{i} \omega_2$ lies in the range of $\omega_{t^*}$. In addition, the function $I_{\mu}(\omega)$ decreases as $\mathrm{Im}\,\omega$ increases since 
\begin{align*}
\frac{\partial}{\partial y} 
\int_{\mathbb{R}}\frac{1}{(u-x)^2+y^2}\,\widehat{\mu}(\mathrm{d}u)
=
-2 y \int_{\mathbb{R}}\frac{1}{\big((u-x)^2+y^2\big)^2}\,\widehat{\mu}(\mathrm{d}u)
<
0,\qquad x\in\mathbb{R},\ y\in\mathbb{R}_+,
\end{align*}
and Proposition \ref{range omegat} implies the existence of a real number $\widetilde{E_t}$ such that 
\begin{align}\label{E_*}
f(\mathrm{Re}\widetilde{w}_t,\mathrm{Im}\widetilde{w}_t)=0, \text{ with } \widetilde{w}_t:=\omega_t(\widetilde{E_t}).
\end{align}
Therefore there exist $\varepsilon>0$ sufficiently small and a neighborhood $\mathcal{U}_{\varepsilon}$ of $E_*$ such that $\widetilde{E}_t\in\mathcal{U}_{\varepsilon}$ for every $0<t-t^*<\varepsilon$ and such that \eqref{range cube} and $|\mathrm{Re}\,\widetilde{w}_t - \omega_*| = \mathcal{O}\big( (\mathrm{Im}\,\widetilde{w}_t)^2 \big)$ hold. We will next derive precise asymptotics for $\omega_t(z)-\widetilde{w}_t$, when $0<t-t^*<\varepsilon$ with $\varepsilon$ sufficiently small and~$z$ lies in $\mathcal{U}_{\varepsilon}$. We fix $t>t^*$ and define the function 
\begin{align}\label{equation t big}
H(\omega):= \omega-(t-1)\int_{\mathbb{R}}\frac{1}{u-\omega}\,\widehat{\mu}(\mathrm{d}u),\qquad \omega\in\mathbb{C}^+\cup\mathbb{R}\backslash\mathrm{supp}(\widehat{\mu}).
\end{align}
By equation \eqref{inverse}, the function $H$ corresponds to the inverse of the subordination function~$\omega_t$ on $\omega_t\big( \mathbb{C}^+\cup\mathbb{R} \big)$. 
Moreover since $\widetilde{w}_t$ lies at positive distance from $\mathrm{supp}(\widehat{\mu})$, $H$ can be Taylor expanded at $\widetilde{w}_t$ and we get for any $\omega$ in a neighborhood of $\widetilde{w}_t$,  
\begin{align}\label{equation star}
\hspace{-0.355cm} H(\omega)-\widetilde{E_t}
= 
a(t,\widetilde{w}_t) (\omega-\widetilde{w}_t)
\hspace{-0.055cm}
+
\hspace{-0.055cm}
b(t,\widetilde{w}_t)(\omega-\widetilde{w}_t)^2
\hspace{-0.055cm}
+
\hspace{-0.055cm}
c(t,\widetilde{w}_t) (\omega-\widetilde{w}_t)^3
\hspace{-0.055cm}
+
\hspace{-0.055cm}
\mathcal{O}\big((\omega-\widetilde{w}_t)^4\big),
\end{align}
where we introduced the notations
\begin{align}
a(t,\omega)&:=1-(t-1)\int_{\mathbb{R}}\frac{1}{(u-\omega)^2}\,\widehat{\mu}(\mathrm{d}u),\\
b(t,\omega)&:=-
(t-1)\int_{\mathbb{R}}\frac{1}{(u-\omega)^3}\,\widehat{\mu}(\mathrm{d}u),\\
c(t,\omega)&:=-
(t-1)\int_{\mathbb{R}}\frac{1}{(u-\omega)^4}\,\widehat{\mu}(\mathrm{d}u).
\end{align}

In order to obtain the asymptotics of the solution of equation \eqref{equation star}, we will first obtain the asymptotics of $a(t,\widetilde{w}_t)$ and $b(t,\widetilde{w}_t)$, when $t>t^*$ is sufficiently close to $t^*$, so that $\widetilde{E_t}\in\mathcal{U}_\varepsilon$ and $|\mathrm{Re}\,\widetilde{w}_t-\omega_*| = \mathcal{O}\big((\mathrm{Im}\,\widetilde{w}_t)^2\big)$.  By Proposition \ref{range omegat}, we get 
\begin{align*}
&\frac{t^*-t}{(t^*-1)(t-1)} 
=
\int_{\mathbb{R}}\frac{1}{|u-\widetilde{w}_t|^2}\,\widehat{\mu}(\mathrm{d}u)
-
\int_{\mathbb{R}}\frac{1}{(u-\omega_{*})^2}\,\widehat{\mu}(\mathrm{d}u)\\ 
&=
3\int_{\mathbb{R}}\frac{1}{(u-\omega_*)^4}\,\widehat{\mu}(\mathrm{d}u)(\mathrm{Re}\,\widetilde{w}_t-\omega_*)^2
-
\int_{\mathbb{R}}\frac{1}{(u-\omega_*)^4}\,\widehat{\mu}(\mathrm{d}u)(\mathrm{Im}\,\widetilde{w}_t)^2
+
\mathcal{O}\big( (\mathrm{Im}\,\widetilde{w}_t)^3 \big),
\end{align*}
from which we deduce that $\mathrm{Im}\,\widetilde{w}_t$ and $\mathrm{Re}\,\widetilde{w}_t -\omega_*$ satisfy
\begin{align}\label{asymptotics Re and Im}
\mathrm{Im}\,\widetilde{w}_t
=
\sqrt{\frac{t^*-t}{c(t,\omega_*)(t^*-1)}}+\mathcal{O}(t-t^*),\quad
\mathrm{Re}\,\widetilde{w}_t-\omega_*
=
\mathcal{O}\big(t-t^*\big).
\end{align}
Moreover by Taylor expansion and by construction of $\omega_*$, 
\begin{align}\label{equations a and b}
a(t,\widetilde{w}_t)
&=
\frac{t^*-t}{t^*-1}+3c(t,\omega_*)(\widetilde{w}_t-\omega_*)^2
+
\mathcal{O}\big((\widetilde{w}_t-\omega_*)^3\big), \\
b(t,\widetilde{w}_t)
&=
3c(t,\omega_*)(\widetilde{w}_t-\omega_*)
+
\mathcal{O}\big((\widetilde{w}_t-\omega_*)^2\big).
\end{align}
Combining \eqref{asymptotics Re and Im} and \eqref{equations a and b}, we see that $a(t,\widetilde{\omega}_t)$ and $b(t,\widetilde{\omega}_t)$ verify
\begin{align}\label{asymptotics a}
a(t,\widetilde{w}_t)&=
\frac{2(t-t^*)}{t^*-1}
+
\mathcal{O}\big((t-t^*)^{\frac{3}{2}}\big),\\
b(t,\widetilde{w}_t)&= \label{asymptotics b}
-3\mathrm{i}\sqrt{\frac{c(t,\omega_*)(t^*-t)}{t^*-1}}+\mathcal{O}\big(t-t^*\big).
\end{align}
When $0<t-t^*<\varepsilon$ for some sufficiently small $\varepsilon$, we use  \eqref{equation star} with \eqref{asymptotics a} and \eqref{asymptotics b} to derive the asymptotic formula  
\begin{align}\label{equation E*}
\hspace{-0.155cm}\omega_t(z)-\widetilde{w}_t
=
\frac{z-\widetilde{E_t}}{a(t,\widetilde{w}_t)}
-
\frac{b(t,\widetilde{w}_t)(z-\widetilde{E_t})^2}{a(t,\widetilde{w}_t)^3}
+
\mathcal{O}\Big(\frac{b(t,\widetilde{w}_t)^2-c(t,\widetilde{w}_t)a(t,\widetilde{w}_t)}{a(t,\widetilde{w}_t)^5}(z-\widetilde{E_t})^3\Big)
\end{align}
that holds in the regime  where $z\in\mathcal{U}_{\varepsilon}$ and $|z-\widetilde{E_t}|^2=o\big( (t-t^*)^{{3}} \big)$. Using Proposition \ref{range omegat} and recalling that $\widetilde{w}_t$ lies in $\gamma$, we see that 
\begin{align}
&\mathrm{Re}\,\frac{1}{a(t,\widetilde{w}_t)} \label{real part}
=
\frac{1}{|a(t,\widetilde{w}_t)|^2} \Big(1-(t-1)\,\mathrm{Re}\,\int_{\mathbb{R}}\frac{1}{(u-\widetilde{w}_t)^2}\,\widehat{\mu}(\mathrm{d}u)\Big)>0,\\
&\mathrm{Im}\,\frac{1}{a(t,\widetilde{w}_t)} \label{imaginary part}
=
\frac{2(t-1)}{|a(t,\widetilde{w}_t)|^2} \mathrm{Im}\,\widetilde{w}_t \int_{\mathbb{R}}\frac{u-\mathrm{Re}\,\widetilde{w}_t}{|u-\widetilde{w}_t|^4}\,\widehat{\mu}(\mathrm{d}u)=0.
\end{align}
Therefore by equation~\eqref{equation E*},
\begin{align}\label{equation carrée}
\mathrm{Im}\,\omega_t(x)
=
\mathrm{Im}\,\widetilde{w}_t
-
\mathrm{Im}\,\frac{b(t,\widetilde{w}_t)}{a(t,\widetilde{w}_t)^3}(x-\widetilde{E_{t}})^2
+
\mathcal{O}\Big(\frac{(x-\widetilde{E_t})^3}{(t-t^*)^4}\Big),
\end{align}
whenever $x\in\mathcal{U}_{\varepsilon}$ is real and $|x-\widetilde{E_t}|^2=o\big( (t-t^*)^{{3}} \big)$.
Since
\begin{align}
\mathrm{Im}\,\frac{b(t,\widetilde{w}_t)}{a(t,\widetilde{w}_t)^3}
=
\mathrm{Im}\,b(t,\widetilde{w}_t) \Big(\mathrm{Re}\,\frac{1}{a(t,\widetilde{w}_t)}\Big)^3,
\end{align}
and by \eqref{asymptotics a} and \eqref{asymptotics b} $\mathrm{Re}\,\frac{1}{a(t,\widetilde{w}_t)}$ is strictly positive and $\mathrm{Im}\,b(t,\widetilde{w}_t)$
is strictly negative, we conclude that $-
\mathrm{Im}\,\frac{b(t,\widetilde{w}_t)}{a(t,\widetilde{w}_t)^3}$ in \eqref{equation carrée} is strictly positive. Therefore, using \eqref{connexionomegat}, the Cauchy-Stieltjes inversion formula and the fact that $\widetilde{w}_t$ is at positive distance from the support of~$\mu$, there exists a positive constant ${C_t}$ such that whenever $t-t^*<\varepsilon$ with $\varepsilon$ sufficiently small,
\begin{align}
\rho_t(x)-\rho_t(\widetilde{E_t})={C_{t}} (x-\widetilde{E_{t}})^2+ \mathcal{O}\Big((t-t^*)^{-4}{|x-\widetilde{E_{t}}|^3}\Big),\qquad x\in\mathcal{U}_\varepsilon,
\end{align}
in the regime where $|x-\widetilde{E_t}|^2=o\big((t-t^*)^{{3}}\big)$. Moreover it follows from equations \eqref{asymptotics a} and~\eqref{asymptotics b} that $C_t$ behaves as $(t-t^*)^{-\frac{5}{2}}$.

In the case where $(t-t^*)^3=o\big( |{x-\widetilde{E_t}}|^{{2}} \big)$, it follows from Cardano's formula that the solution~$\omega_t$ of \eqref{equation star} satisfies 
\begin{align}
\omega_t(z)-\widetilde{w}_t
=
\sqrt[3]{\frac{z-\widetilde{E_t}}{c(t,\widetilde{w}_t)}}
+
\mathcal{O}\bigg( \frac{b(t,\widetilde{w}_t)}{c(t,\widetilde{w}_t)} \bigg)
+
\mathcal{O}\bigg( {a(t,\widetilde{w}_t)} \big(z-\widetilde{E_t}\big)^{-\frac{1}{3}} \bigg)
\end{align}
for any $z$ in a small neighborhood $\mathcal{U}_\varepsilon$ of $\widetilde{E_t}$. Since the real part of $c(t,\widetilde{w}_t)$ remains positive when $t-t^*$ is sufficiently small, the Cauchy-Stieltjes inversion formula, \eqref{asymptotics a} and \eqref{asymptotics b} entail the existence of two strictly positive constants $\widetilde{C_t}$ and $\varepsilon$ such that whenever $x\in\mathcal{U}_\varepsilon$, $0<t-t^*<\varepsilon$ and  $(t-t^*)^3=o\big( |{x-\widetilde{E_t}}|^{{2}} \big)$,
\begin{align}
\rho_t(x)-\rho_t(\widetilde{E_t})
=
\widetilde{C_t}\sqrt[3]{{x-\widetilde{E_t}}}
+
\mathcal{O}\big( \sqrt{t-t^*} \big).
\end{align}
This concludes the proof of Theorem~\ref{Theorem 4}.
\end{proof}

\section{Proofs of Theorems \ref{Theorem 2} and \ref{Theorem 22}}\label{section: Proofs of Theorems}

We turn to the study of the points at which $\mu^{\boxplus t}$ admits an atom or at which the density~$\rho_t$  diverges. Our proof of Theorems \ref{Theorem 2} and \ref{Theorem 22} relies on the following lemma.

\begin{lemma}\label{Plemelj}
Let $\nu$ be an absolutely continuous finite measure on the real line with a density~$\rho_{\nu}$ verifying \eqref{Jacobi}. For each interval $[E_{-},E_{+}]$ in the support of $\nu$, we denote by $\beta^{-}, \beta^{+}$ the exponents of $\rho_\nu$ in \eqref{Jacobi} corresponding to $E_{-}$ and $E_{+}$ respectively. If
\begin{align}
f(x)
:=
\frac{\rho_\nu(x)}{(x-E_{-})^{\beta^{-}}(E_{+}-x)^{\beta^{+}}},\qquad x\in[E_{-},E_{+}]
\end{align}
is Hölder continuous, then for any point $E$ in the interior of the support of $\nu$, there exists a  constant $C_1$ such that
\begin{align}\label{Plemelj0}
\int_{\mathbb{R}} \frac{1}{u-\omega}\,\nu(\mathrm{d}u)
=
\int_{\mathbb{R}} \frac{1}{u-E}\,\nu(\mathrm{d}u)
+
C_1(\omega-E)
+
\mathcal{O}\big((\omega-E)^2\big).
\end{align}
Moreover, for each interval $[E_{-},E_{+}]$ in the support of $\nu$,  
\begin{enumerate}
\item if the exponent $\beta^{-}$ is strictly positive, there exists a constant $C_2$ such that
\begin{align}\label{Plemelj1}
\int_{\mathbb{R}}\frac{1}{u-\omega}\,\nu(\mathrm{d}u)
=
\int_{\mathbb{R}}\frac{1}{u-E_{-}}\,\nu(\mathrm{d}u)
+
C_2(E_{-}-\omega)^{\beta^{-}}
+
\mathcal{O}\Big((E_{-}-\omega)^{1+\beta^{-}}\Big),
\end{align}

\item if the exponent $\beta^{-}$ is strictly negative, there exists a constant $C_3$ such that 
\begin{align}\label{Plemelj2}
\int_{\mathbb{R}}\frac{1}{u-\omega}\,\nu(\mathrm{d}u)
=
C_3(E_{-}-\omega)^{\beta^{-}}
+
\mathcal{O}\Big((E_{-}-\omega\big)^{1+\beta^{-}}\Big),
\end{align}

\item if the exponent $\beta$ is zero, there exists a constant $C_4$ such that
\begin{align}\label{Plemelj3}
\int_{\mathbb{R}}\frac{1}{u-\omega}\,\nu(\mathrm{d}u)
=
C_4\log{(E_{-}-\omega)}
+
\mathcal{O}(1).
\end{align}

\end{enumerate}
The powers in \eqref{Plemelj1} and \eqref{Plemelj2} and the logarithm in \eqref{Plemelj3} are principal branches and the statements of Lemma \ref{Plemelj} hold for the endpoint $E_{+}$ by replacing $(E_{-}-\omega)$ by $(\omega-E_{+})$.
\end{lemma}
We refer e.g.\ to Theorem 14.11 in \cite{Henrici} for a proof of this result. Next, we move on to proving Theorem \ref{Theorem 2}.

\begin{proof}[Proof of Theorem \ref{Theorem 2}]
Let $E_t$ be a point in $\mathcal{S}_t$, i.e.\ a real point such that $\mu(\lbrace E_0 \rbrace)\geq 1-\frac{1}{t}$, with $E_0:=\frac{E_t}{t}=\omega_t(E_t)$. Understanding the asymptotics of $\rho_t(x)$, when $x$ is close to $E_t$, comes to understanding the asymptotics of $\omega_t(z)$ when $z$ is close to $E_t$.  We will therefore prove Theorem \ref{Theorem 2} by looking at the different positions of $E_0$. Namely, we will distinguish the cases where
\begin{enumerate}[leftmargin=1.6cm]
\item[(Case 1.)] $E_0$ is at positive distance from $\mathrm{supp}(\mu_{\mathrm{ac}})$, 
\item[(Case 2.)] $E_0$ lies in the interior of $\mathrm{supp}(\mu_{\mathrm{ac}})$,
\item[(Case 3.)] $E_0$ is the left endpoint $E_{-}$ of an interval $[E_{-},E_+]$ in $\mathrm{supp}(\mu_{\mathrm{ac}})$ and $\beta$ denotes the corresponding exponent of $\rho$ in \eqref{Jacobi}.
\end{enumerate} 
We recall the definition of the inverse function $H$ in \eqref{inverse}. 
\begin{align}\label{inverseinverse}
H(\omega)
:=
\omega-(t-1)\int_{\mathbb{R}}\frac{1}{u-\omega}\,\widehat{\mu}(\mathrm{d}u)
=
t\omega-(t-1)F_{\mu}(\omega),\qquad \omega\in\mathbb{C}^+,
\end{align}
and the second equality in \eqref{inverseinverse} follows from the Nevanlinna representation in \eqref{Nevan}. Since~$E_0$ is an atom of $\mu$, equation \eqref{freeconv} entails
\begin{align}\label{equation z inside support}
H(\omega)-E_t\nonumber
&=
t\big(\omega-E_0\big)-(t-1)F_{\mu}(\omega)\\
&= \nonumber
t\big(\omega-E_0\big)-(t-1)\frac{-1}{\frac{\mu(\lbrace E_0 \rbrace)}{E_0-\omega}+\int_{\mathbb{R}\backslash\lbrace E_0 \rbrace}\frac{1}{u-\omega}\,\mu(\mathrm{d}u)}\\
&=
(\omega-E_0)\Bigg(\frac{t\mu(\lbrace E_0 \rbrace)-(t-1)-t\int_{\mathbb{R}\backslash\lbrace E_0\rbrace}\frac{1}{u-\omega}\,\mu(\mathrm{d}u)(\omega-E_0)}{\mu(\lbrace E_0 \rbrace) - \int_{\mathbb{R}\backslash\lbrace E_0 \rbrace}\frac{1}{u-\omega}\,\mu(\mathrm{d}u)(\omega-E_0)} \Bigg).
\end{align}
By definition of $a(t)$ and $b(\omega)$ in \eqref{notation A} and \eqref{notation B},
we see from~\eqref{equation z inside support} that $H$ verifies
\begin{align}\label{equation omega support}
tb(\omega)(\omega-E_0)^2
-
\Big(a(t)+b(\omega)\big( H(\omega) - E_t \big) \Big) (\omega-E_0)
+
\mu(\lbrace E_0 \rbrace)\big(H(\omega)-E_t\big)
=
0,
\end{align}
whenever $\omega$ lies in a small neighborhood of $\omega_0$. 
The solution of equation \eqref{equation omega support} depends on the speed at which $|H(\omega)-E_t|$ and $a(t)$ vanish. Depending on that speed of convergence, the leading term of the solution will either be of size $\frac{H(\omega)-E_t}{a(t)}$ or will correspond to some root depending on the position of $E_0$ on the real line. Here we will focus on the scenario where $|H(\omega)-E_t|$ vanishes significantly faster than $a(t)$. The scenario where $a(t)$ vanishes faster will be addressed in Theorem~\ref{Theorem 22}.
Using \eqref{inverse}, we can rewrite equation~\eqref{equation omega support} as
\begin{align}\label{discriminant}
\omega_t(z)-E_0
&=\nonumber
\frac{a(t)+b( \omega_t(z) )(z-E_t)}{2tb(\omega_t(z))}\\
&-
\frac{\sqrt{\big( a(t)+b( \omega_t(z) )(z-E_t) \big)^2 
-
4 tb(\omega_t(z)) \mu(\lbrace E_0 \rbrace)(z-E_t)}}{2tb(\omega_t(z))},
\end{align}
where the square root is the principal branch. It is determined by the fact that $\omega_t$ is continuous and maps the upper-half plane to itself. 
There is a competition between $a(t)^2$ and~$b(\omega_t(z))(z-E_t)$ in equation \eqref{discriminant}. Here we are interested in the regime where 
\begin{align}\label{regime 1}
\big|b(\omega_t(z))(z-E_t)\big|=o\big( a(t)^{2} \big).
\end{align}
By Taylor expanding the square root in \eqref{discriminant}, we see that
\begin{align}\label{cinquieme solution}
\omega_t(z)-E_0 
&=
\frac{\mu(\lbrace E_0 \rbrace)}{a(t)}(z-E_t)
+ 
\frac{(t-1)\mu(\lbrace E_0 \rbrace)b(\omega_t(z))}{a(t)^3}(z-E_t)^2
+
\mathcal{O}\bigg(\frac{b(\omega_t(z))^2(z-E_t)^3}{a(t)^5}\bigg).
\end{align}
In order to~understand the asymptotics of $\omega_t(z)-E_0$ in the regime \eqref{regime 1}, we need to~determine the asymptotics of $b(\omega)$, as $\omega$ approaches $E_0$.
We first prove (Case 1.). We see~that
\begin{align}\label{positive distance}
b(\omega_t(z))
=
b(E_0)
+
\int_{\mathbb{R}\backslash\lbrace E_0 \rbrace}\frac{1}{(u-E_0)^2}\,\mu(\mathrm{d}u)\big( \omega_t(z)-E_0 \big)
+
\mathcal{O}\big( (\omega_t(z)-E_0)^2 \big)
\end{align}
holds as long as $|z-E_t|=o( a(t) )$. Therefore if $b(E_0)$ is nonzero, then by \eqref{cinquieme solution} and \eqref{positive distance},
\begin{align}\label{premiere solution}
\omega_t(z)-E_0 
=
&\frac{\mu\big(\lbrace E_0 \rbrace\big)}{a(t)}(z-E_t) 
+  
\frac{(t-1)\mu(\lbrace E_0 \rbrace)b(E_0)}{a(t)^3}(z-E_t)^2
+
\mathcal{O}\bigg(\frac{(z-E_t)^3}{a(t)^5}\bigg),
\end{align}
and we see that the regime in \eqref{regime 1} corresponds to  $|z-E_t|=o\big( a(t)^2 \big)$.
If $b(E_0)=0$, then the regime in \eqref{regime 1} corresponds to  $|z-E_t|^2=o\big( a(t)^3\big)$. By equations \eqref{cinquieme solution} and \eqref{positive distance}, 
\begin{align}\label{troisieme solution}
\hspace{-0.1cm}
\omega_t(z)-E_0 
=
\frac{\mu(\lbrace E_0 \rbrace)}{a(t)} ( z-E_t )
+  
\frac{(t-1)\mu(\lbrace E_0 \rbrace)^2 b'(E_0)}{a(t)^4} ( z-E_t )^3
+
\mathcal{O}\bigg(\frac{(z-E_t)^4}{a(t)^5}\bigg)
.
\end{align}
We have therefore obtained the asymptotics of $\omega_t(z)$ when $E_0$ is at positive distance from the support of the absolutely continuous part of $\mu$. As a side remark, we also observe that 
\begin{align}\label{second derivative}
H''(E_0) 
=
-2(t-1)F_{\mu}''(E_0)  
=
-4(t-1)\frac{b(E_0)}{\mu(\lbrace E_0 \rbrace)^2},
\end{align}
so the condition $b(E_0)=0$ is equivalent to the condition $H''(E_0)=0$ from Theorem \ref{Theorem 1}.

Next, we move on to (Case 2.) and use Lemma \ref{Plemelj} to determine the asymptotics of~$b(\omega_t(z))$. We recall that the Cauchy-Stieltjes inversion formula states that
\begin{align}\label{positive imaginary}
\mathrm{Im}\,b(E_0)=\mathrm{Im}\,\int_{\mathbb{R}\backslash\lbrace E_0 \rbrace}\frac{1}{u-E_0}\,\mu(\mathrm{d}u)
=
\pi\rho(E_0)>0.
\end{align} 
As a consequence, we see that \eqref{regime 1} corresponds to the regime where $|z-E_t|=o\big( a(t)^2 \big)$ and that \eqref{Plemelj0} holds for $\omega=\omega_t(z)$ as long as $|z-E_t|=o( a(t) )$. Using \eqref{Plemelj0}, we thus get
\begin{align}\label{cinquieme solution2}
\omega_t(z)-E_0 
=
\frac{\mu(\lbrace E_0 \rbrace)}{a(t)}(z-E_t)
+  
\frac{(t-1)\mu(\lbrace E_0 \rbrace)b(E_0)}{a(t)^3}(z-E_t)^2
+
\mathcal{O}\bigg(\frac{(z-E_t)^3}{a(t)^5}\bigg),
\end{align}

To treat (Case 3.), we combine \eqref{Plemelj1}, \eqref{Plemelj2} and \eqref{Plemelj3} in Lemma \ref{Plemelj} with \eqref{cinquieme solution}. 
If the exponent $\beta$ is strictly positive and if $b(E_0)$ is nonzero, then the regime in \eqref{regime 1} corresponds to $|z-E_t|=o\big( a(t)^2\big)$ and \eqref{Plemelj1}  holds for $\omega=\omega_t(z)$ as long as $|z-E_t|=o( a(t) )$. Therefore
\begin{align}\label{cinquieme solution3}\nonumber
\omega_t(z)-E_0 
&=
\frac{\mu(\lbrace E_0 \rbrace)}{a(t)}(z-E_t)
+  \nonumber
\frac{(t-1)\mu(\lbrace E_0 \rbrace)b(E_0)}{a(t)^3}(z-E_t)^2\\ 
&+
\frac{(t-1)\mu(\lbrace E_0 \rbrace)^{1+\beta}C_2}{a(t)^{3+\beta}}\big(E_t-z\big)^{2+\beta}
+
\mathcal{O}\bigg(\frac{(z-E_t)^{3}}{a(t)^5}\bigg),
\end{align}
where $C_2$ is given in \eqref{Plemelj1} in Lemma \ref{Plemelj}. Identically, if $b(E_0)$ is zero, then the regime in~\eqref{regime 1} corresponds to $|z-E_t|^{1+\beta}=o\big( a(t)^{2+\beta} \big)$ and by \eqref{Plemelj1}, we get
\begin{align}\label{cinquieme solution33}
\omega_t(z)-E_0
=
\frac{\mu(\lbrace E_0 \rbrace)}{a(t)}(z-E_t)
+  
\frac{(t-1)\mu(\lbrace E_0 \rbrace)^{1+\beta}C_2}{a(t)^{3+\beta}}\big(E_t-z\big)^{2+\beta}
+
\mathcal{O}\bigg(\frac{(z-E_t)^{3+\beta}}{a(t)^{5+\beta}}\bigg).
\end{align}

Next, we assume that the exponent $\beta$ is strictly negative. We see that equation \eqref{Plemelj2} holds for $\omega=\omega_t(z)$ whenever $|z-E_t|=o( a(t))$ and that the regime in~\eqref{regime 1} corresponds to  $|z-E_t|^{1-|\beta|}=o\big( a(t)^{2-|\beta|}\big)$. Equation \eqref{cinquieme solution} then  becomes 
\begin{align}\label{cinquieme solution4}
\omega_t(z)-E_0
=
\frac{\mu(\lbrace E_0 \rbrace)}{a(t)}(z-E_t)
+  
\frac{(t-1)\mu(\lbrace E_0 \rbrace)^{1-|\beta|}C_3}{a(t)^{3-|\beta|}}\big(E_t-z\big)^{2-|\beta|}
+
\mathcal{O}\bigg(\frac{(z-E_t)^{3-2|\beta|}}{a(t)^{5-2|\beta|}}\bigg),
\end{align}
where $C_3$ is given in \eqref{Plemelj2} in Lemma \ref{Plemelj}. 
Finally,  we use \eqref{Plemelj3} to treat the case where the exponent $\beta$ is zero. \eqref{Plemelj3} holds for $\omega=\omega_t(z)$ as long as $|z-E_t|=o( a(t))$ and the regime in~\eqref{regime 1} corresponds to  $\big|\log(\frac{z-E_t}{a(t)})(z-E_t)\big|=o\big( a(t)^2 \big)$. As a consequence,
\begin{align}\label{cinquieme solution5}\nonumber
\omega_t(z)-E_0 
&=
\frac{\mu(\lbrace E_0 \rbrace)}{a(t)}\big(z-E_t\big)
+
\frac{(t-1)\mu(\lbrace E_0 \rbrace)C_4}{a(t)^3}\log\bigg( \frac{\mu(\lbrace E_0 \rbrace)(E_t-z)}{a(t)} \bigg)(z-E_t)^2\\  
&+
\mathcal{O}\bigg(\frac{(z-E_t)^2}{a(t)^3}\bigg)
+
\mathcal{O}\bigg(\frac{(z-E_t)^3}{a(t)^5}\log\Big(\frac{z-E_t}{a(t)}\Big)^2\bigg),
\end{align}
where $C_4$ is given in \eqref{Plemelj3} in Lemma \ref{Plemelj}.
We have therefore determined the asymptotic behavior of $\omega_t(z)-E_0$ in (Case 1.), (Case 2.) and (Case 3.).
Next, we derive asymptotic formulas for $\mathrm{Im}\,m_{\mu^{\boxplus t}}(x)$ when $x$ is close to $E_t$. We first recall that
\begin{align}
\mathrm{Im}\,m_{\mu^{\boxplus t}}(E_t+\mathrm{i}\eta)
=
\frac{\mu^{\boxplus t}(\lbrace E_t \rbrace)}{\eta}
+
o\big({\eta}^{-1}\big)
=
\frac{a(t)}{\eta}
+
o\big({\eta}^{-1}\big),\qquad \eta>0,
\end{align}
so $\mathrm{Im}\,m_{\mu^{\boxplus t}}(E_t+\mathrm{i}\eta)$ diverges when $\eta=o\big( a(t) \big)$ and when $t$ is fixed, $E_t$ is an atom of~$\mu^{\boxplus t}$. Moreover, we recall from equation \eqref{connexionomegat} that
\begin{align}
\mathrm{Im}\,m_{\mu^{\boxplus t}}(z)=I_{\mu}\big( \omega_t(z) \big)\,\mathrm{Im}\,\omega_t(z),\qquad z\in\mathbb{C}^+\cup\mathbb{R}.
\end{align} 
Therefore since
\begin{align}
I_{\mu}(\omega):=
\int_{\mathbb{R}}\frac{1}{|u-\omega|^2}\,\mu(\mathrm{d}u)
=
\frac{\mu(\lbrace E_0 \rbrace )}{|E_0-\omega|^2}
+
\int_{\mathbb{R}\backslash\lbrace E_0 \rbrace}\frac{1}{|u-\omega|^2}\,\mu(\mathrm{d}u),
\end{align}
we obtain the identity
\begin{align}\label{equation im m mu}
\mathrm{Im}\,m_{\mu^{\boxplus t}}(z)
=
\mu(\lbrace E_0 \rbrace )\frac{\mathrm{Im}\,\omega_t(z)}{|\omega_t(z)-E_0|^2}
+
\int_{\mathbb{R}\backslash\lbrace E_0 \rbrace}\frac{\mathrm{Im}\,\omega_t(z)}{|u-\omega_t(z)|^2}\,\mu(\mathrm{d}u),\quad z\in\mathbb{C}^+\cup\mathbb{R}.
\end{align}

Combining the asymptotics in \eqref{cinquieme solution}, \eqref{premiere solution}, \eqref{troisieme solution}, \eqref{cinquieme solution2}, \eqref{cinquieme solution3}, \eqref{cinquieme solution4} and \eqref{cinquieme solution5}, and using \eqref{equation im m mu} and the Cauchy-Stieltjes inversion formula, we get the following. Let $\mathcal{U}_t$ be a small neighborhood of $E_t$ and let $x\in\mathcal{U}_t$.

\begin{enumerate}
\item In (Case 1.), when $|x-E_t|=o(a(t)^2)$, equations \eqref{premiere solution} and \eqref{troisieme solution} imply that 
\begin{align} \label{solution 0}
\mathrm{Im}\,m_{\mu^{\boxplus t}}(x)=0.
\end{align}

\item In (Case 2.), when $|x-E_t|=o(a(t)^2)$, equation \eqref{cinquieme solution2} implies that  
\begin{align}\label{solution 1}
\mathrm{Im}\,m_{\mu^{\boxplus t}}(x)
=
\frac{\pi t \mu(\lbrace E_0 \rbrace)  \rho(E_0) }{a(t)}+
\mathcal{O}\bigg( \frac{|x-E_t|}{a(t)^3}\bigg).
\end{align}

\item In (Case 3.), if $\beta$ is strictly positive and $|x-E_t|=o(a(t)^2)$, or if $\beta$ is strictly negative and $|x-E_t|^{1+\beta}=o\big( a(t)^{2+\beta}\big)$, then \eqref{cinquieme solution3}, \eqref{cinquieme solution33} and \eqref{cinquieme solution4} entail the existence of a positive constant $C_{\beta}$ that is independent of~$t>1$ such that when $x\geq E_t$
\begin{align}\label{solution 2}
\hspace{-0.18cm}
\mathrm{Im}\,m_{\mu^{\boxplus t}}(x)
= 
\frac{\pi t \mu(\lbrace E_0 \rbrace)^{1+\beta} C_\beta}{a(t)^{1+\beta}}(x-E_t)^{\beta}
+
\mathcal{O}\bigg( \frac{(x-E_t)^{1+\beta}}{a(t)^{3+\beta}}\bigg)
+
\mathcal{O}\bigg( \frac{(x-E_t)^{1+2\beta}}{a(t)^{3+2\beta}}\bigg),
\end{align}
and when $x<E_t$, $\mathrm{Im}\,m_{\mu^{\boxplus t}}(x)=0$.

\item In (Case 3.), if ~$\beta$ is zero and $\big|\log(\frac{x-E_t}{a(t)})(x-E_t)\big|=o\big( a(t)^2 \big)$,  then by \eqref{cinquieme solution5} there exists a positive constant~$C_{0}$ that is independent of $t>1$ such that when $x\geq E_t$,
\begin{align} \label{solution 3}
\mathrm{Im}\,m_{\mu^{\boxplus t}}(x)
=
\frac{\pi t \mu(\lbrace E_0 \rbrace) C_0}{a(t)}
+
\mathcal{O}\bigg(\frac{x-E_t}{a(t)^3}\log\Big(\frac{x-E_t}{a(t)}\Big)^2\bigg),
\end{align}
and when $x<E_t$, $\mathrm{Im}\,m_{\mu^{\boxplus t}}(x)=0$.
\end{enumerate}

The same statements hold if $E_0$ is the right endpoint $E_{+}$ of an interval $[E_{-},E_{+}]$ in $\mathrm{supp}(\mu_{\mathrm{ac}})$, by replacing $x-E_t$ by $E_t-x$.  This concludes the proof of Theorem \ref{Theorem 2}.
\end{proof}

The constants in the leading terms of \eqref{solution 1}, \eqref{solution 2} and \eqref{solution 3}  diverge as $a(t)$ converges to zero. We therefore carry on a separate asymptotic analysis  of the density $\rho_t$ in a regime where $a(t)$ decays sufficiently fast. In particular, we will also determine the asymptotics of~$\rho_t(x)$, when $x$ is close to $E_0$ and $a(t)=0$.

\begin{proof}[Proof of Theorem \ref{Theorem 22}]
As in the proof of Theorem \ref{Theorem 2}, we will distinguish the cases where
\begin{enumerate}[leftmargin=1.6cm]
\item[(Case 1.)] $E_0$ is at positive distance from $\mathrm{supp}(\mu_{\mathrm{ac}})$, 
\item[(Case 2.)] $E_0$ lies in the interior of $\mathrm{supp}(\mu_{\mathrm{ac}})$,
\item[(Case 3.)] $E_0$ is the left endpoint $E_{-}$ of an interval $[E_{-},E_+]$ in $\mathrm{supp}(\mu_{\mathrm{ac}})$ and $\beta$ denotes the corresponding exponent of $\rho$ in \eqref{Jacobi}.
\end{enumerate} 
As we will see, the behavior of the density $\rho_t$ around the point $E_t$ will depend on the value of $b(E_0)$ and also on the exponent $\beta$ of $\rho$ in (Case 3.).
Recalling the notations of $a(t)$ and~$b(t)$ in \eqref{notation A} and \eqref{notation B}, and solving equation \eqref{equation omega support}, we  conclude that
\begin{align}\label{discriminant2}
\omega_t(z)-E_0
&=\nonumber
\frac{a(t)+b( \omega_t(z) )(z-E_t)}{2tb(\omega_t(z))}\\
&+
\frac{\sqrt{\big( a(t)+b( \omega_t(z) )(z-E_t) \big)^2 
-
4 tb(\omega_t(z)) \mu(\lbrace E_0 \rbrace)(z-E_t)}}{2tb(\omega_t(z))},
\end{align}
where the square root is determined by the fact that $\omega_t$ is continuous and maps the upper-half plane to itself. We consider the regime where 
\begin{align}\label{regime 2}
a(t)^2=o\big( |b(\omega_t(z))(z-E_t)| \big).
\end{align}
Then by Taylor expanding the square root in \eqref{discriminant2} we obtain 
\begin{align}\label{equation generale}
\omega_t(z)-E_0 \nonumber
=
&\sqrt{\frac{\mu(\lbrace E_0 \rbrace)(E_t-z)}{tb(\omega_t(z))}}
+
\frac{z-E_t}{2t}
+
\frac{1}{2t}\frac{a(t)}{b(\omega_t(z))}
+
\mathcal{O}\Big( b(\omega_t(z))^{\frac{1}{2}}(z-E_t)^{\frac{3}{2}} \Big)\\
&+
\mathcal{O}\bigg( \frac{a(t) ({z-E_t})^{\frac{1}{2}}}{{b(\omega_t(z))}^{\frac{1}{2}}} \bigg)
+
\mathcal{O}\bigg( \frac{a(t)^2}{b(\omega_t(z))^{\frac{3}{2}}(z-E_t)^{\frac{1}{2}}} \bigg).
\end{align}
The branch of the square root in \eqref{equation generale} is determined by the fact that $\omega_t$ is a self-mapping of the upper-half plane and that its real part is increasing. Moreover it follows from Lemma~\ref{Plemelj} and \eqref{regime 2} that the leading term of equation~\eqref{equation generale}
is given by $\sqrt{\frac{\mu(\lbrace E_0 \rbrace)(E_t-z)}{tb(\omega_t(z))}}$. 
We will use Lemma \ref{Plemelj} to determine the asymptotics of $b\big(\omega_t(z)\big)$ in (Case 1.), (Case 2.) and (Case~3.).

We first consider (Case 1.). If $b(E_0)$ is nonzero, then the regime in \eqref{regime 2} corresponds to~$a(t)^2=o\big(|z-E_t|\big)$. Moreover by \eqref{equation generale}, 
\vspace{-0.09cm}
\begin{align}\label{case 1 1}
\omega_t(z) -E_0
\nonumber
&=
\sqrt{\frac{\mu(\lbrace E_0 \rbrace)(E_t-z)}{t b(E_0)}}
+
\frac{1}{2t}\Big(1+\frac{\mu(\lbrace E_0 \rbrace)b'(E_0)}{b(E_0)^2}\Big)(z-E_t)
+
\frac{1}{2t}\frac{a(t)}{b(E_0)}\\
&+
\mathcal{O}\Big((z-E_t)^{\frac{3}{2}}\Big)
+
\mathcal{O}\big(a(t)\sqrt{z-E_t}\big)
+
\mathcal{O}\Big(\frac{a(t)^2}{\sqrt{z-E_t}}\Big).
\end{align}
If we next assume that $b(E_0)$ is zero, then \eqref{regime 2} corresponds to  $a(t)^3=o\big( |z-E_t|^{2}\big)$ and we deduce from \eqref{equation generale} that
\begin{align}\label{case 1 2}
&\big(\omega_t(z)-E_0\big)^{\frac{3}{2}} \nonumber
=
\sqrt{\frac{\mu(\lbrace E_0 \rbrace)(E_t-z)}{t b'(E_0)}}
-
\frac{1}{4}\frac{b''(E_0)}{b'(E_0)}\Big(\frac{\mu(\lbrace E_0 \rbrace)(E_t-z)}{tb'(E_0)} \Big)^{\frac{5}{6}}
+\mathcal{O}\Big((z-E_t)^{\frac{7}{6}}\Big)\\ 
&+
\frac{a(t)}{2t b'(E_0)} \Big(\frac{\mu(\lbrace E_0 \rbrace)(E_t-z)}{tb'(E_0)} \Big)^{-\frac{1}{6}}
+
\mathcal{O}\Big( a(t) (z-E_t)^{\frac{1}{6}} \Big)
+
\mathcal{O}\Big( \frac{a(t)^2}{(z-E_t)^\frac{5}{6}} \Big).
\end{align}

We now consider (Case 2.). We first observe that the regime in \eqref{regime 2} corresponds to  $a(t)^2=o\big(|z-E_t|\big)$, since the imaginary part of $b(E_0)$ is strictly positive by equation \eqref{positive imaginary}. Moreover, using \eqref{Plemelj0} in Lemma \ref{Plemelj}, we obtain 
\begin{align}\label{case 2}
&\omega_t(z)-E_0
\nonumber
=
\sqrt{\frac{\mu(\lbrace E_0 \rbrace)(E_t-z)}{t b(E_0)}} 
+
\frac{1}{2t}(z-E_t)\bigg(1-\frac{C_1\mu(\lbrace E_0 \rbrace)}{b(E_0)^{2}}\bigg)\\
&+
\frac{1}{2t}\frac{a(t)}{b(E_0)}
+
\mathcal{O}\Big((z-E_t)^{\frac{3}{2}}\Big)
+
\mathcal{O}\Big(a(t)\sqrt{z-E_t}\Big)
+
\mathcal{O}\Big(\frac{a(t)^2}{\sqrt{z-E_t}}\Big),
\end{align}
where $C_1$ is the constant in \eqref{Plemelj0}. 

In (Case 3.), when $\beta$ is strictly positive and $b(E_0)$ is nonzero, the regime in \eqref{regime 2} corresponds to $a(t)^2=o\big( |z-E_t| \big)$ and by Lemma~\ref{Plemelj},  \eqref{regime 2} and \eqref{equation generale}, 
\begin{align}\label{case 3 1}
&\omega_t(z)-E_0 \nonumber
=
\sqrt{\frac{\mu(\lbrace E_0 \rbrace)(E_t-z)}{t b(E_0)}}
+
\frac{1}{2}\frac{C_2}{b(E_0)}\Bigg( \hspace{-0.15cm}-\hspace{-0.1cm} \sqrt{\frac{\mu(\lbrace E_0 \rbrace)(E_t-z)}{tb(E_0)}} \Bigg)^{{1+\beta}}
\hspace{-0.4cm}
+
\frac{1}{2t}\frac{a(t)}{b(E_0)}\\
&+
\mathcal{O}\Big( (z-E_t)^{\beta+\frac{1}{2}} \Big)
+
\mathcal{O}\big( z-E_t \big)
+
\mathcal{O}\Big( a(t)(z-E_t)^{\frac{\beta}{2}} \Big)
+
\mathcal{O}\Big( \frac{a(t)^2}{\sqrt{z-E_t}} \Big),
\end{align}
where $C_2$ is the constant in \eqref{Plemelj1}.

In (Case 3.), when $\beta$ is strictly positive and $b(E_0)$ is zero, the regime in \eqref{regime 2} corresponds to $a(t)^{2+\beta}=o\big(|z-E_t|^{1+\beta}\big)$ and by Lemma \ref{Plemelj}, \eqref{regime 2} and \eqref{equation generale}, we obtain 
\begin{align}\label{case 3 2}
&\big(E_0-\omega_t(z)\big)^{1+\frac{\beta}{2}} \nonumber
=
-
\sqrt{\frac{\mu(\lbrace E_0 \rbrace)(E_t-z)}{t C_2}}
-
\frac{1}{2t}\frac{a(t)}{C_2}\Bigg(\hspace{-0.15cm}-\hspace{-0.1cm}\sqrt{\frac{\mu( \lbrace E_0 \rbrace )(E_t-z)}{tC_2}} \Bigg)^{-\frac{\beta}{2+\beta}}\\
&+ 
\mathcal{O}\Big( (z-E_t)^{\frac{4+\beta}{4+2\beta}} \Big)
+
\mathcal{O}\Big( a(t)(z-E_t)^{\frac{2-\beta}{4+2\beta}} \Big)
+
\mathcal{O}\Big({a(t)^2}{(z-E_t)^{-\frac{2+3\beta}{4+2\beta}}} \Big),
\end{align}
where $C_2$ is the constant in \eqref{Plemelj1}. We remark that since  $\beta$ lies in $(0,1)$, the exponents of the leading terms in the right-hand side  of \eqref{case 3 2} are all in the range $(-1,-\frac{1}{3})\cup(\frac{1}{3},1)$. Therefore the branches of the different roots are determined by the fact that the real part of $\omega_t$ is increasing and the fact that $\omega_t$ maps the upper-half plane to itself.

Next, when $\beta$ is strictly negative, it follows from Lemma \ref{Plemelj} that  \eqref{regime 2} corresponds to the regime $a(t)^{2+\beta}=o\big(|z-E_t|^{{1+\beta}}\big)$. Moreover,  using Lemma \ref{Plemelj}, \eqref{regime 2} and \eqref{equation generale}, we get
\begin{align}\label{case 3 3}
&\big(E_0-\omega_t(z)\big)^{1+\frac{\beta}{2}} \nonumber
\hspace{-0.07cm}
=
\hspace{-0.07cm}
-\sqrt{\frac{\mu(\lbrace E_0 \rbrace)(E_t-z)}{t C_3}}
+
\frac{1}{2t}(E_t-z)\Bigg(\hspace{-0.15cm}-\hspace{-0.1cm}\sqrt{\frac{\mu(\lbrace E_0 \rbrace)(E_t-z)}{tC_3}} \Bigg)^{\frac{\beta}{2+\beta}}\\ \nonumber
&-
\frac{1}{2t}\frac{a(t)}{C_3}\Bigg(\hspace{-0.15cm}-\hspace{-0.1cm}\sqrt{\frac{\mu( \lbrace E_0 \rbrace )(E_t-z)}{tC_3}} \Bigg)^{-\frac{\beta}{2+\beta}}
\hspace{-0.3cm}
+
\mathcal{O}\Big( (z-E_t)^\frac{4+\beta}{4+2\beta} \Big)
+
\mathcal{O}\Big( (z-E_t)^\frac{6+5\beta}{4+2\beta} \Big)\\
&+
\mathcal{O}\Big( a(t)(z-E_t)^{\frac{1}{2}} \Big)
+
\mathcal{O}\Big({a(t)^2}{(z-E_t)^{-\frac{2+3\beta}{4+2\beta}}} \Big)
\end{align}
and the constant $C_3$ is the same as in equation \eqref{Plemelj2}.

Finally, we turn to the case where the exponent $\beta$ is zero. The regime in \eqref{regime 2} corresponds to $a(t)^2=o\big({|(z-E_t)\log(z-E_t)|}\big)$ and combining equations \eqref{Plemelj3} with \eqref{equation generale}, we obtain
\begin{align}\label{case 3 4}
&\sqrt{\log\big(E_0-\omega_t(z)\big)}\big(E_0-\omega_t(z)\big) \nonumber
=
-\sqrt{\frac{\mu(\lbrace E_0 \rbrace)(E_t-z)}{tC_4}} 
+
\mathcal{O}\bigg( \sqrt{z-E_t}\, \frac{\log\big( \log( z-E_t ) \big)}{ \log( z-E_t )} \bigg)\\
&-
\hspace{-0.03cm}
\frac{1}{2t}\frac{a(t)}{C_4}  \Bigg( \log\bigg(\hspace{-0.1cm}-\hspace{-0.1cm}\sqrt{\frac{2\mu(\lbrace E_0 \rbrace)}{tC_4}\frac{E_t-z}{\log(E_t-z)}} \bigg) \hspace{-0.07cm} \Bigg)^{-\frac{1}{2}}
\hspace{-0.4cm}
+
\mathcal{O}\Big( \frac{a(t)}{\log(z-E_t)^{\frac{3}{2}}} \Big)
+
\mathcal{O}\Big( \frac{a(t)^2 (z-E_t)^{-\frac{1}{2}}}{{\log(z-E_t)}}\Big)
\end{align}
and the constant $C_4$ corresponds to the constant in \eqref{Plemelj3}.  We have therefore determined the asymptotics of $\omega_t(z)$, when $z$ is close to $E_t$ in (Case 1.), (Case 2.) and (Case 3.).

In order to derive the asymptotics of $\mathrm{Im}\,m_{\mu^{\boxplus t}}(z)$, whenever $z$ lies in an open neighborhood~$\mathcal{U}_t$ of $E_t$, we recall from \eqref{connexionomegat} that
\begin{align}
\mathrm{Im}\,m_{\mu^{\boxplus t}}(z)=I_{\mu}\big( \omega_t(z) \big)\,\mathrm{Im}\,\omega_t(z),\qquad z\in\mathbb{C}^+\cup\mathbb{R}.
\end{align} 
Therefore since
\begin{align}
I_{\mu}(\omega):=
\int_{\mathbb{R}}\frac{1}{|u-\omega|^2}\,\mu(\mathrm{d}u)
=
\frac{\mu(\lbrace E_0 \rbrace )}{|E_0-\omega|^2}
+
\int_{\mathbb{R}\backslash\lbrace E_0 \rbrace}\frac{1}{|u-\omega|^2}\,\mu(\mathrm{d}u),
\end{align}
we see that
\begin{align}\label{asympt im m}
\mathrm{Im}\,m_{\mu^{\boxplus t}}(z)
=
\mu(\lbrace E_0 \rbrace )\frac{\mathrm{Im}\,\omega_t(z)}{|\omega_t(z)-E_0|^2}
+
\int_{\mathbb{R}\backslash\lbrace E_0 \rbrace}\frac{\mathrm{Im}\omega_t(z)}{|u-\omega_t(z)|^2}\,\mu(\mathrm{d}u),\quad z\in\mathbb{C}^+\cup\mathbb{R}.
\end{align}
Combining \eqref{case 1 1}, \eqref{case 1 2}, \eqref{case 2}, \eqref{case 3 1}, \eqref{case 3 2}, \eqref{case 3 3},  \eqref{case 3 4} and \eqref{asympt im m}, we will derive asymptotics for $\mathrm{Im}\,m_{\mu^{\boxplus t}}(x)$, as $x$ lies in a neighborhood $\mathcal{U}_t$ of $E_t$ and such that ~\eqref{regime 2} holds.

We first look at (Case 1.) and assume that $b(E_0)$ is nonzero. If $b(E_0)$ is strictly positive, using \eqref{case 1 1} and \eqref{asympt im m}, we find that  for every $x\in\mathcal{U}_t$ such that  $a(t)^2=o\big( |x-E_t| \big)$,
\begin{align}
\mathrm{Im}\,m_{\mu^{\boxplus t}}(x)
=
\sqrt{\frac{t\mu(\lbrace E_0 \rbrace)b(E_0)}{x-E_t}}
+
\mathcal{O}\big(1\big)
+
\mathcal{O}\Big(\frac{a(t)}{{x-E_t}}\Big), 
\end{align}
when $x\geq E_t$ and $\mathrm{Im}\,m_{\mu^{\boxplus t}}(x)=0$, when $x<E_t$.
Similarly, if $b(E_0)$ is strictly negative, 
\begin{align}
\mathrm{Im}\,m_{\mu^{\boxplus t}}(x)
=
\sqrt{\frac{t\mu(\lbrace E_0 \rbrace)b(E_0)}{x-E_t}}
+
\mathcal{O}\big( 1 \big)
+
\mathcal{O}\Big(\frac{a(t)}{{E_t-x}}\Big), 
\end{align}
when $x\leq E_t$ and $\mathrm{Im}\,m_{\mu^{\boxplus t}}(x)=0$, when $x>E_t$.

If we next assume that $b(E_0)$ is zero, then by equations  \eqref{case 1 2} and \eqref{asympt im m}, we have for any $x\in\mathcal{U}_t$ in the regime where $a(t)^3=o\big( |x-E_t|^{2} \big)$, 
\begin{align}
\mathrm{Im}\,m_{\mu^{\boxplus t}}(x)
=
\sin(\frac{\pi}{3})\sqrt[3]{\frac{t\mu(\lbrace E_0 \rbrace)^2 b'(E_0)}{|x-E_t|}}
+
\mathcal{O}(1)
+
\mathcal{O}\bigg(\frac{a(t)}{|x-E_t|}\bigg).
\end{align}

We consider (Case 2.). In the regime where $a(t)^2=o\big( |x-E_t| \big)$, there exists some positive constants $S_1^{+}$ and $S_1^{-}$ that are independent of~$t>1$ and such that for every $x\in\mathcal{U}_t$,
\begin{align}
\mathrm{Im}\,m_{\mu^{\boxplus t}}(x)
=
{S_1^{\pm}}{\pi}\,\sqrt{\frac{t\mu(\lbrace E_0 \rbrace)|b(E_0)|}{|x-E_t|}}
+
\mathcal{O}(1)
+
\mathcal{O}\bigg(\frac{a(t)}{|x-E_t|}\bigg), 
\end{align}
where $S_1^{\pm}=S_1^{+}$ if $x\geq E_t$ and $S_1^{\pm}=S_1^{-}$ if $x<E_t$.
We further remark that the constants~$S_1^{+}$ and $S_1^{-}$ depend only on the value of the argument of $b(E_0)$.

We turn to (Case 3.) and assume that the exponent $\beta$ is strictly positive. If $b(E_0)$ is nonzero, we can combine \eqref{case 3 1} and \eqref{asympt im m} to determine the asymptotics of $\mathrm{Im}\,m_{\mu^{\boxplus t}}$. 
When~$b(E_0)$ is strictly positive and $a(t)^2=o\big( |x-E_t| \big)$, we have for every  $x\in\mathcal{U}_t$, 
\begin{align}
\mathrm{Im}\,m_{\mu^{\boxplus t}}(x)
=
\sqrt{\frac{t\mu(\lbrace E_0 \rbrace)b(E_0)}{x-E_t}} 
+
\mathcal{O}\Big( (x-E_t)^{\frac{\beta-1}{2}}\Big)
+
\mathcal{O}\bigg( \frac{a(t)}{x-E_t} \bigg),
\end{align}
when $x\geq E_t$ and $\mathrm{Im}\,m_{\mu^{\boxplus t}}(x)=0$, when $x<E_t$.
In addition, when $b(E_0)$ is strictly negative and $a(t)^2=o\big( |x-E_t| \big)$, there exists a positive constant $S_2$ that is independent of $t>1$ and such that for every real number $x$ in $\mathcal{U}_t$, 
\begin{align}
\mathrm{Im}\,m_{\mu^{\boxplus t}}(x)\nonumber
&= 
S_2 \pi \Bigg({\frac{\mu(\lbrace E_0 \rbrace)}{|b(E_0)|}}\Bigg)^{\frac{1+\beta}{2}}
\Bigg({\frac{t}{x-E_t}} \Bigg)^{\frac{1-\beta}{2}} 
\hspace{-0.5cm}
+
\mathcal{O}\Big( (x-E_t)^{\frac{2\beta-1}{2}}\Big)
+
\mathcal{O}\bigg( \frac{a(t)}{(x-E_t)^{1-\frac{\beta}{2}}} \bigg)\\
&+
\mathcal{O}\bigg( \frac{a(t)^2}{(x-E_t)^{\frac{3}{2}}} \bigg),
\end{align}
when $x\geq E_t$ and
\begin{align}
\mathrm{Im}\,m_{\mu^{\boxplus t}}(x)
\hspace{-0.05cm}
= 
\hspace{-0.05cm}
\sqrt{\frac{t\mu(\lbrace E_0 \rbrace)b(E_0)}{x-E_t}} 
\hspace{-0.05cm}
+
\hspace{-0.05cm}
\mathcal{O}\Big( (E_t-x)^{\frac{\beta-1}{2}}\Big)
\hspace{-0.05cm}
+
\hspace{-0.05cm}
\mathcal{O}\bigg( \frac{a(t)}{E_t-x} \bigg),
\end{align}
when $x<E_t$.
Next, if the exponent $\beta$ is strictly positive and $b(E_0)$ is zero, we obtain using \eqref{case 3 2} and \eqref{asympt im m} the existence of some positive constants $S_3^{+}$, $S_3^{-}$ such that when $x\in\mathcal{U}_t$ and $a(t)^{2+\beta}=o\big( |x-E_t|^{1+\beta}\big)$,
\begin{align}
\mathrm{Im}\,m_{\mu^{\boxplus t}}(x)
=
S_3^{\pm} \pi  
\Bigg( \frac{t \mu(\lbrace E_0 \rbrace)^{1+\beta}}{|x-E_t|} \Bigg)^{\frac{1}{2+\beta}}
\hspace{-0.5cm}
+
\mathcal{O}\Big( |x-E_t|^{\frac{\beta}{2+\beta}} \Big)
+
\mathcal{O}\big(1\big)
+
\mathcal{O}\bigg( \frac{a(t)}{|x-E_t|} \bigg),
\end{align}
where $S_3^{\pm}$ stands for $S_3^{+}$ if $x\geq E_t$ and for $S_3^{-}$ if $x<E_t$. Moreover if the exponent~$\beta$ is strictly negative,  it follows from   \eqref{case 3 3} and \eqref{asympt im m} that there exists some positive constant~$S_4$ such that when $x\in\mathcal{U}_t$ and $a(t)^{2+\beta}=o\big( |x-E_t|^{1+\beta}\big)$,
\begin{align}
\mathrm{Im}\,m_{\mu^{\boxplus t}}(x)
=
S_4 \pi  
\Bigg( \frac{t \mu(\lbrace E_0 \rbrace)^{1+\beta}}{x-E_t} \Bigg)^{\frac{1}{2+\beta}}
\hspace{-0.5cm}
+
\mathcal{O}\Big( (x-E_t)^{\frac{\beta}{2+\beta}} \Big)
+
\mathcal{O}\big(1\big)
+
\mathcal{O}\bigg( \frac{a(t)}{x-E_t} \bigg),
\end{align}
when $x\geq E_t$ and $\mathrm{Im}\,m_{\mu^{\boxplus t}}(x)=0$, when $x<E_t$. 
Finally, we consider the case where the exponent $\beta$ is zero. Using \eqref{case 3 4} and \eqref{asympt im m}, we see that in the regime where $a(t)^2=o\big( |(x-E_t)\log(x-E_t)| \big)$, there exists some positive constant $S_5$ such that for every $x$ in $\mathcal{U}_t$, 
\begin{align}
\mathrm{Im}\,m_{\mu^{\boxplus t}}(x)\nonumber
&=
S_5\pi\sqrt{\frac{t\mu(\lbrace E_0 \rbrace)|\log(x-E_t)|}{2(x-E_t)}}
+
\mathcal{O}\Big( \frac{\big|\log | \log(x-E_t) | \big|}{\sqrt{(x-E_t)|\log(x-E_t)|}} \Big)\\
&+
\mathcal{O}\bigg( \frac{a(t)\sqrt{|\log(x-E_t)|}}{x-E_t} \bigg),
\end{align}
when $x\geq E_t$ and $\mathrm{Im}\,m_{\mu^{\boxplus t}}(x)=0$, when $x<E_t$. 

The same statements hold if $E_0$ is the right endpoint $E_{+}$ of an interval $[E_{-},E_+]$ in the support of $\mu_{\mathrm{ac}}$, by replacing $x-E_t$ by $E_t-x$. The proof of Theorem \ref{Theorem 22} then follows from the Cauchy-Stieltjes inversion formula.
\end{proof}

\section{Free additive convolution of $\mu_{\alpha}$ and $\mu_{\beta}$}\label{section: free addition definition}

In this section, we review the definition of the free additive convolution via analytic subordination and then prove Theorem \ref{Theorem 3}. 

\begin{proposition}(\cite{Bia98, Voi93}, Theorem 4.1 in \cite{Bel Ber} and Theorem 2.1 in \cite{Chis})\label{definition subalpha}
Let $\mu_{\alpha}$, $\mu_{\beta}$ be two probability measures on the real line.
There exist two unique analytic functions $\omega_\alpha,$ $\omega_\beta:\mathbb{C}^+\to\mathbb{C}^+$ such that for every $z\in\mathbb{C}^+$,
\begin{align}\label{freeadd1}
&\mathrm{Im}\,\omega_\alpha(z)\geq \mathrm{Im}\,z\text{ and }\frac{\omega_\alpha(\mathrm{i}\eta)}{\mathrm{i}\eta}\to 1\text{ as }\eta\to\infty,\\ \label{freeadd3}
&\mathrm{Im}\,\omega_\beta(z)\geq \mathrm{Im}\,z\text{ and }\frac{\omega_\beta(\mathrm{i}\eta)}{\mathrm{i}\eta}\to 1\text{ as }\eta\to\infty,\\ \label{freeadd2}
&\omega_\alpha(z)+\omega_\beta(z)-z
=
F_{\mu_\alpha}\big(\omega_\beta(z)\big)
=
F_{\mu_\beta}\big(\omega_\alpha(z)\big).
\end{align}
The function $F(z):=F_{\mu_{\alpha}}\big(\omega_\beta(z)\big)$ is a self-map of the upper half-plane and by~\eqref{freeadd3},  
\begin{align}
\lim_{\eta\to\infty}\frac{F(\mathrm{i}\eta)}{\mathrm{i}\eta}=
\lim_{\eta\to\infty}\frac{F_{\mu_{\alpha}}\big(\omega_{\beta}(\mathrm{i}\eta)\big)}{\omega_{\beta}(\mathrm{i}\eta)}\frac{\omega_{\beta}(\mathrm{i}\eta)}{\mathrm{i}\eta}=1.
\end{align}
Therefore there exists a probability measure on the real line, denoted by $\mu_{\alpha}\boxplus\mu_{\beta}$, such that $F(z):=F_{\mu_\alpha}\big(\omega_\beta(z)\big)$ is its negative reciprocal Cauchy-Stieltjes transform. This probability measure is the free additive convolution of $\mu_{\alpha}$ and $\mu_{\beta}$.
\end{proposition}

In \cite{Bel Ber}, the existence and the uniqueness of the subordination functions $\omega_\alpha$ and $\omega_\beta$ are proven by using the theory of Denjoy-Wolff fixed points:  
by equation \eqref{freeadd2}, we observe that for any $z\in\mathbb{C}^+$, $\omega_\alpha(z)$ and $\omega_\beta(z)$ are precisely the Denjoy-Wolff points of the analytic functions $f_{\alpha},f_{\beta}:\mathbb{C}^+\rightarrow \mathbb{C}^+\cup\mathbb{R}$ given by 
\begin{align*}
&f_{\alpha}(z,\omega):=F_{\mu_\alpha}\big(F_{\mu_\beta}(\omega)-\omega+z\big)-F_{\mu_\beta}(\omega)+\omega,\\
&f_{\beta}(z,\omega):=F_{\mu_\beta}\big(F_{\mu_\alpha}(\omega)-\omega+z\big)-F_{\mu_\alpha}(\omega)+\omega.
\end{align*}

\begin{remark}
Let $\mu$ be any finite measure on the real line and recall from~\eqref{definition Imu} the definition
\begin{align*}
I_{\mu}(\omega)
:=
\int_{\mathbb{R}} \frac{1}{|u-\omega|^2}\,\mu(\mathrm{d}u), \qquad \omega\in\mathbb{C}^+\cup\mathbb{R}\backslash{\mathrm{supp}(\mu)}.
\end{align*}
Using equation \eqref{freeadd2}, we  get
\begin{align}\label{identité plus petit de base}
\mathrm{Im}\,m_{\mu_{\alpha}\boxplus\mu_{\beta}}(z)
=
\mathrm{Im}\,\omega_{\alpha}(z)\,I_{\mu_{\beta}}\big(\omega_{\alpha}(z)\big)
=
\mathrm{Im}\,\omega_{\beta}(z)\,I_{\mu_{\alpha}}\big(\omega_{\beta}(z)\big),\qquad z\in\mathbb{C}^+.
\end{align} 
Taking the imaginary part in \eqref{freeadd2} and using the Nevanlinna representation in \eqref{Nevan} yield
\begin{align*}
I_{\widehat{\mu}_{\alpha}}\big(\omega_{\beta}(z)\big)=\frac{\mathrm{Im}\,\omega_{\alpha}(z)-\mathrm{Im}\,z}{\mathrm{Im}\,\omega_{\beta}(z)},\quad 
I_{\widehat{\mu}_{\beta}}\big(\omega_{\alpha}(z)\big)=\frac{\mathrm{Im}\,\omega_{\beta}(z)-\mathrm{Im}\,z}{\mathrm{Im}\,\omega_{\alpha}(z)},\qquad z\in\mathbb{C}^+.
\end{align*}
Therefore using \eqref{freeadd1} and \eqref{freeadd3}, we have for any $z\in\mathbb{C}^+$ that
\begin{align}\label{inegalité edge}
I_{\widehat{\mu}_{\alpha}}\big(\omega_{\beta}(z)\big)I_{\widehat{\mu}_{\beta}}\big(\omega_{\alpha}(z)\big)
=
\Big(\frac{\mathrm{Im}\,\omega_\beta(z)}{\mathrm{Im}\,\omega_\alpha(z)}-\frac{\mathrm{Im}\,z}{\mathrm{Im}\,\omega_{\alpha}(z)}\Big)\Big(\frac{\mathrm{Im}\,\omega_\alpha(z)}{\mathrm{Im}\,\omega_\beta(z)}-\frac{\mathrm{Im}\,z}{\mathrm{Im}\,\omega_{\beta}(z)}\Big)
\leq 
1.
\end{align}
\end{remark}

We next recall two fundamental results on the subordination functions and the free additive convolution due to Belinschi.

\begin{proposition}(Theorem 3.3 in \cite{Bel1}, Theorem 6 in \cite{Bel2}). Let $\mu_{\alpha}$ and $\mu_{\beta}$ be two compactly supported probability measures on the real line, none of them being a single point mass. Then the subordination functions $\omega_{\alpha}$ and $\omega_{\beta}$ extend continuously to $\mathbb{C}^+\cup\mathbb{R}$ as functions with values in $\mathbb{C}^+\cup\mathbb{R}\cup{\lbrace \infty \rbrace}$.
\end{proposition}

\begin{proposition}(Theorem 4.1 in \cite{Bel1})\label{extension analytic}
Let $\mu_{\alpha}$ and $\mu_{\beta}$ be two absolutely continuous probability measures that are compactly supported on the real line. The free additive  convolution $\mu_{\alpha}\boxplus\mu_{\beta}$ is absolutely continuous with respect to the Lebesgue measure. Its density is bounded, continuous and real-analytic wherever it is strictly positive. Moreover since $\mu_{\alpha}$ and $\mu_{\beta}$ are compactly supported, so is $\mu_{\alpha}\boxplus\mu_{\beta}$.
\end{proposition}

We now turn to the Nevanlinna representations of the negative reciprocal Cauchy-Stieltjes transforms of $\mu_\alpha$ and $\mu_\beta$.
We use Lemma \ref{Pick} in the same fashion as before.  
By Assumption~\ref{main assumption2} and by the Nevanlinna representation in Proposition \ref{Nevanlinna representation}, there exist two finite measures $\widehat{\mu}_{\alpha}$ and $\widehat{\mu}_{\beta}$ on the real line such that for every $\omega\in\mathbb{C}^+$,
\begin{align}\label{Nevan2}
F_{\mu_\alpha}(\omega)-\omega=\int_{\mathbb{R}} \frac{1}{u-\omega}\,\widehat{\mu}_\alpha(\mathrm{d}u)\ \text{ and }\ 
F_{\mu_\beta}(\omega)-\omega=\int_{\mathbb{R}} \frac{1}{u-\omega}\,\widehat{\mu}_\beta(\mathrm{d}u).
\end{align}
Moreover these measures are finite since
\begin{align}\label{second moment}
0<\widehat{\mu}_{\alpha}(\mathbb{R})=\int_{\mathbb{R}} u^2\mu_\alpha(\mathrm{d}u)<\infty
\ \text{ and }\ 
0<\widehat{\mu}_{\beta}(\mathbb{R})=\int_{\mathbb{R}} u^2\mu_\beta(\mathrm{d}u)<\infty.
\end{align}
Proposition \ref{Nevanlinna representation} entails the following result on the supports of $\widehat{\mu}_\alpha$ and $\widehat{\mu}_\beta$.
\begin{corollary}(Corollary 4.1 in \cite{MoSc22})\label{Nevan support}
Let $\mu_\alpha$ and $\mu_\beta$ satisfy Assumption~\ref{main assumption2}. There exist two integers $1\leq m_{\alpha} \leq n_{\alpha}-1$ and $1\leq m_{\beta} \leq n_{\beta}-1$ such that the atomic parts of $\widehat{\mu}_\alpha$ and of $\widehat{\mu}_\beta$ consist of $m_\alpha$ and $m_\beta$ points, respectively. More precisely, we have
\begin{align*}
\mathrm{supp}(\widehat{\mu}_\alpha)&=\mathrm{supp}({\mu}_\alpha)\cup\lbrace E_1^{\alpha},...,E_{m_{\alpha}}^{\alpha}\rbrace,\\
\mathrm{supp}(\widehat{\mu}_\beta)&=\mathrm{supp}({\mu}_\beta)\cup\lbrace E_1^{\beta},...,E_{m_{\beta}}^{\beta}\rbrace,
\end{align*}
and each $E\in\big\lbrace E_1^{\alpha},...,E_{m_{\alpha}}^{\alpha} \big\rbrace$ is a zero of $m_{\mu_\alpha}$ and each $E\in\big\lbrace E_1^{\beta},...,E_{m_\beta}^{\beta} \big\rbrace$ is a zero of ~$m_{\mu_\beta}.$ 
\end{corollary}

By Remark \ref{pure point de mu hat}, recall that every zero of $m_{\mu_\alpha}$, that lies at positive distance from $\mathrm{supp}({\mu}_{\alpha})$, is an atom of $\widehat{\mu}_{\alpha}$, and the same statement applies to $m_{\mu_\beta}$.

In order to study the support of the free additive convolution $\mu_{\alpha}\boxplus\mu_{\beta}$, 
let $\mathcal{E}$ be a compact interval on the real line such that $\mathrm{supp}(\mu_{\alpha}\boxplus\mu_{\beta})\subset\mathcal{E}$. Then consider the spectral domain
\begin{align*}
\mathcal{J}:=\big\lbrace z=E+\mathrm{i}\eta:\ E\in\mathcal{E}\text{ and }0\leq \eta \leq 1 \big\rbrace.
\end{align*}

As shown in \cite{MoSc22}, the subordination functions $\omega_\alpha$ and $\omega_\beta$ inherit several properties from the fact that $\mu_\alpha$ and $\mu_\beta$ verify Assumption \ref{main assumption2}. In particular, they remain bounded, whenever~$z$ lies in a compact domain that does not contain any zero of the Cauchy-Stieltjes transforms of $\mu_{\alpha}$ and $\mu_\beta$. Furthermore $\omega_\alpha$ stays at positive distance from the support of ${\mu}_\beta$ and $\omega_\beta$ remains at positive distance from the support of ${\mu}_\alpha$. We gather the main properties of the subordination functions in the following proposition. 

\begin{proposition}(Lemmas 4.4, 4.6, 4.9 and Proposition 4.7 in \cite{MoSc22})\label{general results alpha beta}
Let $\mu_{\alpha}$ and $\mu_{\beta}$ be two probability measures verifying Assumption~\ref{main assumption2} and let $\mathcal{U}_\alpha$ be an open set containing $\lbrace E_1^{\alpha},...,E_{m_{\alpha}}^{\alpha}\rbrace$ and $\mathcal{U}_\beta$ be an open set containing $\lbrace E_1^{\beta},...,E_{m_{\beta}}^{\beta}\rbrace$.
\begin{enumerate}
\item There exist two strictly positive constants $C_{1}$ and $C_{2}$ depending on the measures $\mu_{\alpha}$ and $\mu_{\beta}$, on~$\mathcal{J}$, and on the open sets $\mathcal{U}_\alpha$ and $\mathcal{U}_\beta$ such that 
\begin{align*}
\sup_{z\in\mathcal{J}\backslash \mathcal{U}_\alpha}|\omega_{\alpha}(z)|\leq C_{1},\quad
\sup_{z\in\mathcal{J}\backslash \mathcal{U}_\beta}|\omega_\beta(z)|\leq C_{2}. 
\end{align*}
\item There exist strictly positive constants $C_3$ and $C_4$ such that 
\begin{align*}
\inf_{z\in\mathcal{J}}\mathrm{dist}\big(\omega_\alpha(z),\mathrm{supp}(\mu_\beta)\big)\geq C_3,
\quad\inf_{z\in\mathcal{J}}\mathrm{dist}\big(\omega_\beta(z),\mathrm{supp}(\mu_\alpha)\big)\geq C_4.
\end{align*}
\item There exist two constants $C_5\geq 1$ and $C_6\geq 1$, depending on the measures $\mu_{\alpha}$ and $\mu_{\beta}$, on $\mathcal{J}$ and on the open sets $\mathcal{U}_\alpha$ and $\mathcal{U}_\beta$, such that 
\begin{align*}
C_5^{-1}\,\mathrm{Im}\, m_{\mu_{\alpha}\boxplus\mu_{\beta}}(z) 
&\leq 
\mathrm{Im}\, \omega_\alpha(z) 
\leq 
C_5\,\mathrm{Im}\, m_{\mu_{\alpha}\boxplus\mu_{\beta}}(z),\qquad z\in\mathcal{J}\backslash\mathcal{U}_\alpha,\\
C_6^{-1}\,\mathrm{Im}\, m_{\mu_{\alpha}\boxplus\mu_{\beta}}(z) 
&\leq 
\mathrm{Im}\, \omega_\beta(z) 
\leq 
C_6\,\mathrm{Im}\, m_{\mu_{\alpha}\boxplus\mu_{\beta}}(z),\qquad z\in\mathcal{J}\backslash\mathcal{U}_\beta.
\end{align*} 
\item An atom $E^{\alpha}$ of $\widehat{\mu}_{\alpha}$ lies in the range of $\omega_{\beta}$ if and only if $\widehat{\mu}_{\alpha}(\lbrace E^{\alpha} \rbrace)\geq\widehat{\mu}_{\beta}(\mathbb{R})$. Identically, an atom $E^{\beta}$ of $\widehat{\mu}_{\beta}$ lies in the range of $\omega_{\alpha}$ if and only if $\widehat{\mu}_{\beta}(\lbrace E^{\beta} \rbrace)\geq\widehat{\mu}_{\alpha}(\mathbb{R})$.
\end{enumerate}
\end{proposition}
We hence remark that if $E\in\mathcal{E}$ is not a vanishing point of the density of~$\mu_\alpha\boxplus\mu_\beta$, then the Cauchy-Stieltjes inversion formula in Lemma \ref{singular continuous} implies that $E$ lies in the support of~$\mu_{\alpha}\boxplus\mu_{\beta}$  if and only if $\mathrm{Im}\,\omega_\alpha(E)>0$ and $\mathrm{Im}\,\omega_\beta(E)>0$. Moreover item (3) of Proposition~\ref{general results alpha beta} also implies that $\mathrm{Im}\,\omega_\alpha$ and $\mathrm{Im}\,\omega_\beta$ either both vanish, or are both strictly positive.
In the case of the free additive convolution semigroup of a single measure $\mu$,  the subordination function~$\omega_t$ always stays at positive distance from the zeros of $m_{\mu}$. As pointed out in \cite{MoSc22}, in the case of the free additive convolution of two distinct probability measures, the subordination functions can approach zeros of the Cauchy-Stieltjes transform. More specifically,  an atom~$E^{\alpha}$ of $\widehat{\mu}_{\alpha}$ lies in the range of the subordination function~$\omega_{\beta}$ if and only if $E^{\alpha}$ is an atom of~$\widehat{\mu}_{\alpha}$ such that $\widehat{\mu}_{\beta}(\mathbb{R})\leq\widehat{\mu}_{\alpha}(\lbrace E^{\alpha}\rbrace)$. When it happens, the asymptotics below hold.

\begin{proposition}\label{divergence subordination} 
Let $\mu_{\alpha}$ and $\mu_{\beta}$ be two probability measures satisfying Assumption \ref{main assumption2} and~$E^{\alpha}$ be an atom of $\widehat{\mu}_{\alpha}$ that verifies the condition $\widehat{\mu}_{\beta}(\mathbb{R})\leq\widehat{\mu}_{\alpha}(\lbrace E^{\alpha} \rbrace)$. Then as $z$ approaches $E^{\alpha}$, we have 
\begin{align}\label{asymptotics sub}
\omega_\alpha(z)-z
&=
\frac{\widehat{\mu}_\alpha\big(\lbrace E^{\alpha} \rbrace\big)}{E^{\alpha}-\omega_\beta(z)}\big(1+\mathcal{O}(E^{\alpha}-\omega_{\beta}(z))\big),\\
E^{\alpha}-z 
&= \nonumber
\gamma_0 \Big(\frac{E^{\alpha}-\omega_{\beta}(z)}{\widehat{\mu}_{\alpha}(\lbrace E^{\alpha}\rbrace)}\Big)
\hspace{-0.03cm}
-
\hspace{-0.03cm}
\gamma_1 \Big(\frac{E^{\alpha}-\omega_{\beta}(z)}{\widehat{\mu}_{\alpha}(\lbrace E^{\alpha}\rbrace)}\Big)^2
\hspace{-0.1cm}
-
\hspace{-0.03cm}
\gamma_2 \Big(\frac{E^{\alpha}-\omega_{\beta}(z)}{\widehat{\mu}_{\alpha}(\lbrace E^{\alpha}\rbrace)}\Big)^3 
\hspace{-0.1cm}
+
\hspace{-0.03cm}
\mathcal{O}\big((E^{\alpha}-\omega_{\beta}(z))^4 \big),
\end{align}
where 
\begin{align}\label{gamma0}
\gamma_0&:=
\widehat{\mu}_{\alpha}(\lbrace E^{\alpha}\rbrace)-\widehat{\mu}_{\beta}(\mathbb{R}),\\ \label{gamma1}
\gamma_1&:=
\int_{\mathbb{R}}u\widehat{\mu}_{\beta}(\mathrm{d}u)-\widehat{\mu}_{\beta}(\mathbb{R}) j(E^{\alpha}),\\ \label{gamma2}
\gamma_2&:= \nonumber
\int_{\mathbb{R}}u^2\,\widehat{\mu}_{\beta}(\mathrm{d}u)
-2j(E^{\alpha})\int_{\mathbb{R}}u\,\widehat{\mu}_{\beta}(\mathrm{d}u)
-\widehat{\mu}_{\beta}(\mathbb{R})\\
&\times\Big( 
\widehat{\mu}_{\beta}(\mathbb{R})
-j(E^{\alpha})^2
 - \widehat{\mu}_{\alpha}(\lbrace E^{\alpha} \rbrace)\big(1+\int_{\mathbb{R}\backslash\lbrace E^{\alpha} \rbrace}\frac{1}{(u-E^{\alpha})^2}\,\widehat{\mu}_{\alpha}(\mathrm{d}u)\big) 
\Big).
\end{align}
and 
\begin{align}
{j}(E^{\alpha}):=E^{\alpha}+\int_{\mathbb{R}\backslash\lbrace E^{\alpha}\rbrace} \frac{1}{u-E^{\alpha}}\,\widehat{\mu}_{\alpha}(\mathrm{d}u).
\end{align}
The same result holds if we interchange $\alpha$ and $\beta$.
\end{proposition}

\begin{remark} Our proof can be extended to any probability measures satisfying Assumption~\ref{main assumption2}, but not necessarily the power law behavior in \eqref{Jacobi2} and \eqref{Jacobi22} if we assume that~$E^{\alpha}$ lies at positive distance from the support of $\mu_{\alpha}$. Indeed the asymptotics in \eqref{asymptotics sub} rely on a Taylor expansion of $j(z)$. However such analysis would no longer hold if $E^{\alpha}$ is an endpoint of the support of $\mu_{\alpha}$ and the asymptotic analysis of $j(z)$ would in general yield different exponents in \eqref{asymptotics sub}.
\end{remark}

\begin{proof}
The condition $\widehat{\mu}_{\beta}(\mathbb{R})\leq\widehat{\mu}_{\alpha}(\lbrace E^{\alpha} \rbrace)$ ensures that the atom $E^{\alpha}$ belongs to the range of~$\omega_{\beta}$ and by Proposition \ref{general results alpha beta} that it lies at positive distance from the support of $\mu_{\alpha}$.
By the subordination equation in \eqref{freeadd2}, we have
\begin{align}\label{subequation 1}
\omega_{\alpha}(z)&=z+\int_{\mathbb{R}}\frac{1}{u-\omega_\beta(z)}\,\widehat{\mu}_{\alpha}(\mathrm{d}u),\qquad z\in\mathbb{C}^+\cup\mathbb{R},\\
\label{subequation 2}
\omega_{\beta}(z)&=z+\int_{\mathbb{R}}\frac{1}{u-\omega_\alpha(z)}\,\widehat{\mu}_{\beta}(\mathrm{d}u),\qquad z\in\mathbb{C}^+\cup\mathbb{R}
\end{align}
and by Corollary \ref{Nevan support}, $\mathrm{supp}((\mu_{\alpha})_{\mathrm{ac}})=\mathrm{supp}\big((\widehat{\mu}_{\alpha})_{\mathrm{ac}}\big)$. Therefore $E^{\alpha}$ lies at positive distance from the support of $(\widehat{\mu}_\alpha)_{\mathrm{ac}}$. We set 
\begin{align}
g(z):=\frac{1}{\widehat{\mu}_\alpha\big(\lbrace E^{\alpha} \rbrace\big)}\int_{\mathbb{R}\backslash \lbrace E^{\alpha}\rbrace}\frac{E^{\alpha}-\omega_\beta(z)}{u-\omega_\beta(z)}\,\widehat{\mu}_\alpha(\mathrm{d}u)
\end{align}
and 
\begin{align}
h(z):=\frac{1}{\widehat{\mu}_\beta(\mathbb{R})}\int_{\mathbb{R}}\frac{1}{1-\frac{u-z}{\omega_\alpha(z)-z}}\,\widehat{\mu}_\beta(\mathrm{d}u)-1,
\end{align}
and observe, using equations \eqref{subequation 1} and \eqref{subequation 2}, that
\begin{align}\label{asympt 1}
\omega_\alpha(z)-z=\frac{\widehat{\mu}_\alpha\big(\lbrace E^{\alpha} \rbrace\big)}{E^{\alpha}-\omega_\beta(z)}\big(1+g(z)\big)
\end{align}
and 
\begin{align}\label{asympt 2}
\omega_\beta(z)-z 
=
\frac{\widehat{\mu}_\beta(\mathbb{R})}{z-\omega_\alpha(z)}\big( 1+h(z) \big).
\end{align}
Combining \eqref{asympt 1} and \eqref{asympt 2}, we then see that 
\begin{align*}
E^{\alpha}-\omega_\beta(z)
=
E^{\alpha}-z+\frac{\widehat{\mu}_\beta(\mathbb{R})}{\widehat{\mu}_\alpha\big(\lbrace E^{\alpha} \rbrace \big)}\big(E^{\alpha}-\omega_\beta(z)\big)
\frac{1+h(z)}{1+g(z)},
\end{align*} 
which entails
\begin{align}\label{equation z subordination}
\big(E^{\alpha}-\omega_\beta(z)\big)\Big(1-\frac{\widehat{\mu}_\beta(\mathbb{R})}{\widehat{\mu}_\alpha\big(\lbrace E^{\alpha}\rbrace\big)}  \frac{1+h(z)}{1+g(z)} \Big)
=
E^{\alpha}-z.
\end{align}
Next we will determine the asymptotics of the ratio $\frac{1+h(z)}{1+g(z)}$, when $z$ approaches $E^{\alpha}$. We notice that
\begin{align}\label{asympt 1plush}
\frac{1+h(z)}{1+g(z)}
&=\nonumber
\frac{1}{\widehat{\mu}_\beta\big(\mathbb{R}\big)}
\frac{\omega_\alpha(z)-z}{1+g(z)}\int_{\mathbb{R}}\frac{1}{\omega_\alpha(z)-u}\,\widehat{\mu}_\beta(\mathrm{d}u)\\
&=\nonumber
\frac{\widehat{\mu}_\alpha\big(\lbrace E^{\alpha}\rbrace\big)}{\widehat{\mu}_\beta\big(\mathbb{R}\big)}
\frac{1}{E^{\alpha}-\omega_\beta(z)}
\int_\mathbb{R}\frac{1}{m_{\widehat{\mu}_\alpha}(\omega_\beta(z))-(u-z)}\,\widehat{\mu}_\beta(\mathrm{d}u)\\
&=
\frac{\widehat{\mu}_\alpha\big(\lbrace E^{\alpha}\rbrace\big)}{\widehat{\mu}_\beta\big(\mathbb{R}\big)}
\frac{-1}{E^{\alpha}-\omega_\beta(z)}\,
m_{\widehat{\mu}_\beta}\Big(m_{\widehat{\mu}_\alpha}\big(\omega_\beta(z)\big)+z\Big).
\end{align}
Since $\omega_{\alpha}(z)$ diverges as $z$ approaches $E^{\alpha}$, we see that
\begin{align}
m_{\widehat{\mu}_{\beta}}\big(m_{\widehat{\mu}_{\alpha}}(\omega_{\beta}(z))+z\big)
=
-\sum_{k\geq 0} \frac{\int_{\mathbb{R}}u^k\,\widehat{\mu}_{\beta}(\mathrm{d}u)}{(m_{\widehat{\mu}_{\alpha}}(\omega_{\beta}(z))+z)^{k+1}}
\end{align}
when $z$ is sufficiently close to $E^{\alpha}$. Moreover if we set 
\begin{align}
j(z):=z+\int_{\mathbb{R}\backslash\lbrace E^{\alpha} \rbrace}\frac{1}{u-\omega_{\beta}(z)}\,\widehat{\mu}_{\alpha}(\mathrm{d}u), 
\end{align}
we see that for any $k\geq 0$, 
\begin{align}\label{asympt mmuhat}\nonumber
&\Big(\frac{1}{z+m_{\widehat{\mu}_{\alpha}}(\omega_{\beta}(z))}\Big)^{k+1}
=
\Big(\frac{E^{\alpha}-\omega_{\beta}(z)}{\widehat{\mu}_{\alpha}(\lbrace E^{\alpha} \rbrace)}\Big)^{k+1}
-
(k+1)j(z)\Big(\frac{E^{\alpha}
-
\omega_{\beta}(z)}{\widehat{\mu}_{\alpha}(\lbrace E^{\alpha} \rbrace)}\Big)^{k+2}\\ 
&+
\frac{1}{2}(k+1)(k+2)j(z)^2\Big( \frac{E^{\alpha}-\omega_{\beta}(z)}{\widehat{\mu}_{\alpha}( \lbrace E^{\alpha} \rbrace)}\Big)^{k+3}
+
\mathcal{O}\big((E^{\alpha}-\omega_{\beta}(z))^{k+4}\big)
\end{align}
and combining \eqref{asympt mmuhat} with \eqref{asympt 1plush} and \eqref{equation z subordination}, we obtain
\begin{align}\label{asympt11}
E^{\alpha}-z\nonumber
&=
\Big( \frac{E^{\alpha}-\omega_{\beta}(z)}{\widehat{\mu}_{\alpha}(\lbrace E^{\alpha} \rbrace)}\Big) 
\big(\widehat{\mu}_{\alpha}(\lbrace E^{\alpha}\rbrace)-\widehat{\mu}_{\beta}(\mathbb{R})\big)\\ \nonumber
&-
\Big( \frac{E^{\alpha}-\omega_{\beta}(z)}{\widehat{\mu}_{\alpha}(\lbrace E^{\alpha} \rbrace)}\Big)^2
\big(\int_{\mathbb{R}}u\widehat{\mu}_{\beta}(\mathrm{d}u)-\widehat{\mu}_{\beta}(\mathbb{R}) j(z)\big)\\ \nonumber
&- 
\Big( \frac{E^{\alpha}-\omega_{\beta}(z)}{\widehat{\mu}_{\alpha}(\lbrace E^{\alpha} \rbrace)}\Big)^3
\big(
\widehat{\mu}_{\beta}(\mathbb{R})j(z)^2
-2j(z)\int_{\mathbb{R}}u\,\widehat{\mu}_{\beta}(\mathrm{d}u)+\int_{\mathbb{R}}u^2\,\widehat{\mu}_{\beta}(\mathrm{d}u)
\big)\\
&+
\mathcal{O}\big( (E^{\alpha}-\omega_{\beta}(z))^4 \big).
\end{align}
Using \eqref{asympt11}, we then see that as $z$ approaches $E^{\alpha}$,
\begin{align}
j(z)
&=
j(E^{\alpha})\nonumber
-
\Big( \frac{E^{\alpha}-\omega_{\beta}(z)}{\widehat{\mu}_{\alpha}(\lbrace E^{\alpha} \rbrace)} \Big)
\Big( 
\widehat{\mu}_{\alpha}(\lbrace E^{\alpha} \rbrace)\big(1
+
\int_{\mathbb{R}\backslash \lbrace E^{\alpha}\rbrace}\frac{1}{(u-E^{\alpha})^2}\,\widehat{\mu}_{\alpha}(\mathrm{d}u)\big)
-
\widehat{\mu}_{\beta}(\mathbb{R})
\Big)\\
&+
\mathcal{O}\big( (E^{\alpha}-\omega_{\beta}(z))^2\big).
\end{align}
We therefore conclude that 
\begin{align}
E^{\alpha}-z\nonumber
=
\gamma_0
 \frac{E^{\alpha}-\omega_{\beta}(z)}{\widehat{\mu}_{\alpha}(\lbrace E^{\alpha} \rbrace)} 
-
\gamma_1
\Big( \frac{E^{\alpha}-\omega_{\beta}(z)}{\widehat{\mu}_{\alpha}(\lbrace E^{\alpha} \rbrace)}\Big)^2 
- 
\gamma_2
\Big( \frac{E^{\alpha}-\omega_{\beta}(z)}{\widehat{\mu}_{\alpha}(\lbrace E^{\alpha} \rbrace)}\Big)^3
+
\mathcal{O}\big( (E^{\alpha}-\omega_{\beta}(z))^4 \big),
\end{align}
where $\gamma_0$, $\gamma_1$ and $\gamma_2$ are defined in \eqref{gamma0}, \eqref{gamma1} and \eqref{gamma2}.
This concludes the proof of this proposition.
\end{proof}

We next remark that by \eqref{freeadd1} and \eqref{freeadd2}, the subordination functions $\omega_\alpha$ and $\omega_\beta$ inherit Nevanlinna representations.

\begin{lemma} \label{Nevan omega alpha}
Let $\mu_\alpha$ and $\mu_\beta$ be two probability measures on the real line satisfying Assumption~\ref{main assumption2}.
Then there exist two finite measures $\nu_{\alpha}$ and $\nu_{\beta}$ supported on the real line~such~that
\begin{align}\label{Nevan omega alpha 1}
&\omega_\alpha(z)
=
z+\int_{\mathbb{R}}\frac{1}{u-z}\,\nu_{\alpha}(\mathrm{d}u),\qquad z\in\mathbb{C}^+\cup\mathbb{R}\backslash\mathrm{supp}(\nu_\alpha),\\ \label{Nevan omega alpha 2}
&\omega_\beta(z)
=
z+\int_{\mathbb{R}}\frac{1}{u-z}\,\nu_{\beta}(\mathrm{d}u),\qquad z\in\mathbb{C}^+\cup\mathbb{R}\backslash\mathrm{supp}(\nu_\beta).
\end{align}
Moreover, 
\begin{align}\label{support alpha}
\mathrm{supp}(\nu_{\alpha})
&=
\mathrm{supp}(\mu_{\alpha}\boxplus\mu_{\beta})
\cup
\big\lbrace x\in\mathrm{supp}\big((\widehat{\mu}_{\alpha})_{\mathrm{pp}}\big):\ \widehat{\mu}_{\beta}(\mathbb{R})\leq\widehat{\mu}_{\alpha}(\lbrace x \rbrace) \big\rbrace,\\ \label{support beta}
\mathrm{supp}(\nu_{\beta})
&=
\mathrm{supp}(\mu_{\alpha}\boxplus\mu_{\beta})
\cup
\big\lbrace x\in\mathrm{supp}\big((\widehat{\mu}_{\beta})_{\mathrm{pp}}\big):\ \widehat{\mu}_{\alpha}(\mathbb{R})\leq\widehat{\mu}_{\beta}(\lbrace x \rbrace) \big\rbrace.
\end{align}
In particular, the real part of $\omega_{\alpha}$ is strictly increasing on each interval in $\mathbb{R}\backslash\mathrm{supp}(\nu_{\alpha})$ and the real part of $\omega_{\beta}$ is strictly increasing on each interval in $\mathbb{R}\backslash\mathrm{supp}(\nu_{\beta})$.
\end{lemma}

\begin{proof}
The proof of equations \eqref{Nevan omega alpha 1} and \eqref{Nevan omega alpha 2} is identical to the one of Lemma 4.4 in~\cite{Bao20}. Regarding \eqref{support alpha} and \eqref{support beta}, it follows from Proposition \ref{divergence subordination} that $\omega_\alpha$ only admits a pole at the atoms $E^{\alpha}$ of $\widehat{\mu}_{\alpha}$ such that $\widehat{\mu}_{\alpha}(\lbrace E^{\alpha} \rbrace)\geq\widehat{\mu}_{\beta}(\mathbb{R})$. Since $\omega_{\alpha}$ preserves the upper and lower half-planes, we conclude that $E^{\alpha}$ is an atom of $\nu_{\alpha}$. The same argument holds for $\omega_\beta$.
\end{proof}

As in the study of $\mu^{\boxplus t}$, the endpoints of the support of $\mu_\alpha\boxplus\mu_\beta$ can be characterized explicitly in terms of the derivatives of the reciprocal Cauchy-Stieltjes transforms of $\mu_\alpha$ and~$\mu_\beta$. The two-dimensional analogue is the following.

\begin{proposition}\label{edge charact}[Proposition 4.3 in \cite{Bao20}]
Let $\mu_{\alpha}$ and $\mu_{\beta}$ be two probability measures on the real line satisfying Assumption \ref{main assumption2}. 
If we let $\rho_{\alpha\boxplus\beta}$  denote the density function of $\mu_\alpha\boxplus\mu_\beta$ and if we define~by
\begin{align*}
\mathcal{V}_{\alpha\boxplus\beta}:=\partial\lbrace x\in\mathbb{R}:\ \rho_{\alpha\boxplus\beta}(x)>0 \rbrace,
\end{align*}
the set of points at which $\rho_{\alpha\boxplus\beta}$ vanishes, then
\begin{align*}
\big|(F'_{\mu_\alpha}(\omega_\beta(z))-1)(F'_{\mu_\beta}(\omega_\alpha(z))-1)\big|\leq 1,\qquad z\in\mathbb{C}^+\cup\mathbb{R}.
\end{align*}
Moreover equality is achieved for $z=E+\mathrm{i}\eta$ if and only if 
\begin{align*}
E\in\mathcal{V}_{\alpha\boxplus\beta}\quad \text{ and }\quad \eta=0.
\end{align*}
In fact we have that for such a $z$,
$(F'_{\mu_\alpha}(\omega_\beta(z))-1)(F'_{\mu_\beta}(\omega_\alpha(z))-1)=1$.
\end{proposition}

\begin{proof}
When $E\in\mathcal{V}_{\alpha\boxplus\beta}$ is not an atom of $\widehat{\mu}_{\alpha}$ or $\widehat{\mu}_{\beta}$, the argument is identical to the proof of Proposition 4.3 in \cite{Bao20}. When $E$ is an atom of $\widehat{\mu}_{\alpha}$ or $\widehat{\mu}_\beta$, the claim is a consequence of Proposition \ref{divergence subordination} since $E\in\mathcal{V}_{\alpha\boxplus\beta}$ if and only if $\gamma_0=0$.
\end{proof}

We are now prepared for the proof of Theorem \ref{Theorem 3}. We will perform an asymptotic analysis of the subordination functions $\omega_{\alpha}$ and $\omega_{\beta}$ that heavily relies on Assumption \ref{main assumption2}. Indeed, we observe that when \eqref{Jacobi2} and \eqref{Jacobi22} no longer hold, $\omega_{\alpha}$ and $\omega_{\beta}$ can approach the supports of $\mu_{\alpha}$ and $\mu_{\beta}$. When it happens, their asymptotics involve in general different exponents (see Plemelj's formula in Lemma \ref{Plemelj}) that will impact the local behavior of $\rho_{\alpha\boxplus\beta}$. 

\begin{proof}[Proof of Theorem \ref{Theorem 3}] 
Let $E_0\in\mathcal{V}_{\alpha\boxplus\beta}$. We split the proof in the two cases where $E_0$ is either an atom of $\widehat{\mu}_\alpha$ or $\widehat{\mu}_\beta$, or is not an atom of either measures. 
First, we assume that~$E_0\in\mathcal{V}_{\alpha\boxplus\beta}$ is an atom of $\widehat{\mu}_\alpha$. Then by Proposition \ref{divergence subordination}, $E_0$ is an exterior point of $\mathcal{V}_{\alpha\boxplus\beta}$ (see \eqref{exterior point}) if $\gamma_0=0$ and $\gamma_1\neq 0$.
Proposition \eqref{divergence subordination} then entails
\begin{align*}
\omega_\beta(z)-E_0
=
\widehat{\mu}_\alpha(\lbrace E_0 \rbrace)\sqrt{\frac{z-E_0}{\gamma_1}}
+
\mathcal{O}\big(z-E_0\big),
\end{align*}
where the square root is the principal branch, since $\omega_\beta$ is a self-mapping of the upper-half plane. In addition, $E_0$ being at positive distance from the support of $\mu_\alpha$ ensures that~$I_{\mu_\alpha}\big(\omega_\beta(z)\big)$ is strictly positive and remains bounded in a neighborhood $\mathcal{U}_0$~of~$E_0$. Therefore,  
since 
\begin{align}\label{approx I}
\int_{\mathbb{R}}\frac{1}{|u-\omega_{\beta}(z)|^2}\,{\mu_{\alpha}}(\mathrm{d}u)
=
&\int_{\mathbb{R}}\frac{1}{(u-E_0)^2}\,{\mu_{\alpha}}(\mathrm{d}u) \nonumber
+
2\int_{\mathbb{R}}\frac{1}{(u-E_0)^3}\,\mu_{\alpha}(\mathrm{d}u)\big(\mathrm{Re}\,\omega_{\beta}(z)-E_0\big)\\
&- \nonumber
\int_{\mathbb{R}}\frac{1}{(u-E_0)^4}\,{\mu_\alpha}(\mathrm{d}u)\big(\mathrm{Im}\,\omega_{\beta}(z)\big)^2\\
&+
\mathcal{O}\Big( \big(\mathrm{Re}\,\omega_{\beta}(z)-E_0\big)^2 \Big)
+
\mathcal{O}\Big( \big(\mathrm{Im}\,\omega_{\beta}(z)\big)^3 \Big),
\end{align}
we conclude from the identity 
\begin{align}\label{rappel comparaison}
\mathrm{Im}\,m_{\mu_\alpha\boxplus\mu_\beta}(z)
=
\mathrm{Im}\,\omega_{\beta}(z)\,I_{\mu_{\alpha}}\big(\omega_{\beta}(z)\big),\qquad z\in\mathbb{C}^+\cup\mathbb{R},
\end{align}
and from the Cauchy-Stieltjes inversion formula, the existence of a strictly positive constant~$C_1$ and of a neighborhood $\mathcal{U}_0$ of $E_0$ such that for every $x$ in $\mathcal{U}_0$, we either have 
\begin{align}
\rho_{\alpha\boxplus\beta}(x)
=
C_1\sqrt{x-E_0}+\mathcal{O}\big(x-E_0\big),
\end{align}
when $x\geq E_0$ and $\rho_{\alpha\boxplus\beta}(x)=0$, when $x<E_0$, in the case where $\gamma_1<0$, or we  have
\begin{align}
\rho_{\alpha\boxplus\beta}(x)
=
C_1\sqrt{E_0-x}+\mathcal{O}\big(E_0-x\big),
\end{align}
when $x<E_0$ and $\rho_{\alpha\boxplus\beta}(x)=0$, when $x>E_0$, in the case where $\gamma_1>0$. This proves \eqref{claim alpha beta 1} in the situation where the exterior point $E_0$ is an atom of $\widehat{\mu}_\alpha$.  
Next, still assuming that~$E_0\in\mathcal{V}_{\alpha\boxplus\beta}$ is an atom of $\widehat{\mu}_\alpha$, we observe that Proposition~\ref{divergence subordination} implies that $E_0$ is an interior point of $\mathcal{V}_{\alpha\boxplus\beta}$ (see \eqref{interior point}), if $\gamma_0=\gamma_1=0$.
Equation \eqref{asymptotics sub} then implies that
\begin{align}
\omega_\beta(z)-E_0
=
\widehat{\mu}_\alpha(\lbrace E_0 \rbrace)\sqrt[3]{\frac{E_0-z}{\gamma_2}}
+
\mathcal{O}\big((z-E_0)^{\frac{2}{3}}\big),
\end{align}
where the cubic root is determined by the fact that $\omega_\beta$ maps the upper-half plane to itself and that its real part is strictly increasing on each interval of $\mathbb{R}\backslash\mathrm{supp}(\nu_{\beta})$.
We deduce from \eqref{approx I}, \eqref{rappel comparaison}~and the Cauchy-Stieltjes inversion formula, the existence of a strictly positive constant $C_2$ such~that 
\begin{align}
\rho_{\alpha\boxplus\beta}(x)
=
C_2\sqrt[3]{|x-E_0|}+\mathcal{O}\big( |x-E_0|^{\frac{2}{3}} \big),
\end{align}
whenever $x$ is in a neighborhood $\mathcal{U}_0$ of $E_0$. By Cauchy-Schwarz, we observe that if \linebreak $\gamma_0=\gamma_1=0$, then $\gamma_2>0$. Indeed when $\gamma_0=\gamma_1=0$, we see that
\begin{align}
\gamma_2
=
\int_{\mathbb{R}}u^2\,\widehat{\mu}_{\beta}(\mathrm{d}u)
-
\frac{\big( \int_{\mathbb{R}}u\,\widehat{\mu}_{\beta}(\mathrm{d}u) \big)^2}{\widehat{\mu}_{\beta}(\mathbb{R})}
+\widehat{\mu}_{\beta}(\mathbb{R})^2\int_{\mathbb{R}\backslash\lbrace E_0 \rbrace}\frac{1}{(u-E_0)^2}\,\widehat{\mu}_{\alpha}(\mathrm{d}u)
>0.
\end{align} 
Therefore no other rate of decay can be observed for $\rho_{\alpha\boxplus\beta}$ when $E_0$ is an atom of either $\widehat{\mu}_{\alpha}$ or $\widehat{\mu}_{\beta}$.
This proves \eqref{claim alpha beta 2} in the case where the interior point $E_0$ of $\mathcal{V}_{\alpha\boxplus\beta}$ is an atom of $\widehat{\mu}_\alpha$. 

We move on to the second case where $E_0\in\mathcal{V}_{\alpha\boxplus\beta}$ is neither an atom of $\widehat{\mu}_\alpha$ nor $\widehat{\mu}_\beta$. 
It follows from Proposition \ref{general results alpha beta} that $\omega_\alpha(E_0)$ lies at positive distance from the support of $\widehat{\mu}_\alpha$ and $\omega_\beta(E_0)$ lies at positive distance from the support of $\widehat{\mu}_\beta$. Furthermore by \eqref{Nevan2}, 
\begin{align}
F_{\mu_\beta}'(\omega)=1+\int_{\mathbb{R}}\frac{1}{(u-\omega)^2}\,\widehat{\mu}_\beta(\mathrm{d}u)>0,\qquad 
\omega\in\mathbb{R}\backslash\mathrm{supp}(\widehat{\mu}_\beta).
\end{align}
By the analytic inverse function theorem, there exist a neighborhood $\mathcal{U}_{\alpha}$ of $\omega_\alpha(E_0)$ and a neighborhood $\mathcal{U}_{\beta}$ of $F_{\mu_\beta}\big(\omega_\alpha(E_0)\big)$ such that $F_{\mu_\beta}:\mathcal{U}_{\alpha}\rightarrow\mathcal{U}_{\beta}$ is analytic and invertible. Therefore the function
\begin{align}\label{inverse equation alpha}
H(\omega):=F_{\mu_\beta}^{-1}\big(F_{\mu_\alpha}(\omega)\big)+\omega-F_{\mu_\alpha}(\omega),\qquad \omega\in\mathcal{U}_{\alpha},
\end{align}
is well-defined and analytic. By equation \eqref{freeadd2}, $H\big(\omega_\alpha(z)\big)=z$, whenever $\omega_\alpha(z)$ is in $\mathcal{U}_{\alpha}$ and via Taylor expansion, $H$ can be expressed as  
\begin{align}\label{Taylor expansion}
H(\omega)
=
E_0+
\sum_{k=1}^3\frac{H^{(k)}\big(\omega_\alpha(E_0)\big)}{k!}\big(\omega-\omega_\alpha(E_0)\big)^k
+
\mathcal{O}\Big(\big(\omega-\omega_\alpha(E_0)\big)^4\Big),\qquad \omega\in\mathcal{U}_{\alpha}.
\end{align}
From \eqref{inverse equation alpha}, we successively compute the derivatives of the function $H$. For brevity, we use the notations $\omega_\alpha(E_0):=\omega_\alpha^{0}$ and $\omega_\beta(E_0):=\omega_\beta^{0}$, to find
\begin{align}
H'(\omega_\beta^0)&=\label{first}
\frac{F_{\mu_\beta}'(\omega_\alpha^0)+F_{\mu_\alpha}'(\omega_\beta^0)-F_{\mu_\beta}'(\omega_\alpha^0)F_{\mu_\alpha}'(\omega_\beta^0)}{F_{\mu_\beta}'(\omega_\alpha^0)},\\
H''(\omega_\beta^0)&=\label{second}
\frac{-F_{\mu_\alpha}'(\omega_\beta^0)^2 F_{\mu_\beta}''(\omega_\alpha^0)+F_{\mu_\beta}'(\omega_\alpha^0)^2 F_{\mu_\alpha}''(\omega_\beta^0) - F_{\mu_\beta}'(\omega_\alpha^0)^3 F_{\mu_\alpha}''(\omega_\beta^0)}{F_{\mu_\beta}'(\omega_\alpha^0)^3},\\
H'''(\omega_\beta^0)&= \label{third a}
\frac{3F_{\mu_\alpha}'(\omega_\beta^0)^3 F_{\mu_\beta}''(\omega_\alpha^0)^2-3F_{\mu_\beta}'(\omega_\alpha^0)^2 F_{\mu_\alpha}'(\omega_\beta^0)F_{\mu_\beta}''(\omega_\alpha^0)F_{\mu_\alpha}''(\omega_\beta^0)}{F_{\mu_\beta}'(\omega_\alpha^0)^5}\\
&\quad-\label{third b}
\frac{F_{\mu_\beta}'(\omega_\alpha^0)F_{\mu_\alpha}'(\omega_\beta^0)^3 F_{\mu_\beta}'''(\omega_\alpha^0) - F_{\mu_\beta}'(\omega_\alpha^0)^4 F_{\mu_\alpha}'''(\omega_\beta^0) + F_{\mu_\beta}'(\omega_\alpha^0)^5 F_{\mu_\alpha}'''(\omega_\beta^0)}{F_{\mu_\beta}'(\omega_\alpha^0)^5}.
\end{align} 
By Proposition \ref{edge charact}, $H'(\omega_\beta^0)=0$ and we will prove that if $H''(\omega_\beta^0)=0$, then $H'''(\omega_\beta^0)<0$. As we will see, this  implies that the density $\rho_{\alpha\boxplus\beta}$ only decays as a square root at any exterior point of $\mathcal{V}_{\alpha\boxplus\beta}$ or as a cubic root at any interior point of $\mathcal{V}_{\alpha\boxplus\beta}$. 

First, we assume that $H''(\omega_\beta^0)=0$ in order to rewrite \eqref{third a}. Using Proposition \ref{edge charact} and~\eqref{second}, we see that $H''(\omega_\beta^0)=0$ is equivalent to
\begin{align}
&F_{\mu_\beta}''(\omega_\alpha^0) F_{\mu_\alpha}'(\omega_\beta^0)^2 
=
F_{\mu_\alpha}''(\omega_\beta^0) F_{\mu_\beta}'(\omega_\alpha^0)^2 \big(1 - F_{\mu_\beta}'(\omega_\alpha^0)\big),\\
&F_{\mu_\alpha}''(\omega_\beta^0) F_{\mu_\beta}'(\omega_\alpha^0)^2 
=
F_{\mu_\beta}''(\omega_\alpha^0) F_{\mu_\alpha}'(\omega_\beta^0)^2 \big(1 - F_{\mu_\alpha}'(\omega_\beta^0)\big).
\end{align}
Therefore the second term in the numerator of \eqref{third a} can be expressed as
\begin{align}
3F_{\mu_\beta}'(\omega_\alpha^0)^2 F_{\mu_\alpha}'(\omega_\beta^0)F_{\mu_\beta}''(\omega_\alpha^0)F_{\mu_\alpha}''(\omega_\beta^0)
=
3 F_{\mu_\alpha}'(\omega_\beta^0)^3 F_{\mu_\beta}''(\omega_\alpha^0)^2 \big(1 - F_{\mu_\alpha}'(\omega_\beta^0)\big),
\end{align}
and the numerator of \eqref{third a} can be simplified to
\begin{align}\label{RED}
&3F_{\mu_\alpha}'(\omega_\beta^0)^3 F_{\mu_\beta}''(\omega_\alpha^0)^2\nonumber
-
3F_{\mu_\beta}'(\omega_\alpha^0)^2 F_{\mu_\alpha}'(\omega_\beta^0)F_{\mu_\beta}''(\omega_\alpha^0)F_{\mu_\alpha}''(\omega_\beta^0)\\ \nonumber
&= 
3F_{\mu_\alpha}'(\omega_\beta^0)^3 F_{\mu_\beta}''(\omega_\alpha^0)^2
-
3 F_{\mu_\alpha}'(\omega_\beta^0)^3 F_{\mu_\beta}''(\omega_\alpha^0)^2 \big(1 - F_{\mu_\alpha}'(\omega_\beta^0)\big)\\
&= 
3 F_{\mu_\alpha}'(\omega_\beta^0)^4 F_{\mu_\beta}''(\omega_\alpha^0)^2.
\end{align}
Next, we look at the numerator of \eqref{third b}. Since $H'(\omega_\beta^0)=0$, we have
\begin{align}\label{identite FF}
F_{\mu_\alpha}'(\omega_\beta^0)F_{\mu_\beta}'(\omega_\alpha^0)
=
F_{\mu_\beta}'(\omega_\alpha^0)+F_{\mu_\alpha}'(\omega_\beta^0),
\end{align}
and the term $F_{\mu_\alpha}'(\omega_\beta^0)^3 F_{\mu_\beta}'(\omega_\alpha^0)$ in \eqref{third b} can be rewritten using \eqref{identite FF} as follows,
\begin{align}\nonumber
&F_{\mu_\beta}'(\omega_\alpha^0)F_{\mu_\alpha}'(\omega_\beta^0)^3 F_{\mu_\beta}'''(\omega_\alpha^0) - F_{\mu_\beta}'(\omega_\alpha^0)^4 F_{\mu_\alpha}'''(\omega_\beta^0) + F_{\mu_\beta}'(\omega_\alpha^0)^5 F_{\mu_\alpha}'''(\omega_\beta^0)\\
&= \nonumber
F_{\mu_\beta}'(\omega_\alpha^0)F_{\mu_\alpha}'(\omega_\beta^0)^3 F_{\mu_\beta}'''(\omega_\alpha^0) 
+ 
\big(F_{\mu_\beta}'(\omega_\alpha^0) -1 \big) F_{\mu_\alpha}'''(\omega_\beta^0) F_{\mu_\beta}'(\omega_\alpha^0)^4  \\
&= \nonumber
F_{\mu_\beta}'''(\omega_\alpha^0) \big(
F_{\mu_\beta}'(\omega_\alpha^0)+
F_{\mu_\alpha}'(\omega_\beta^0)+
F_{\mu_\alpha}'(\omega_\beta^0)^2+
F_{\mu_\alpha}'(\omega_\beta^0)^3
\big)
+
\big(F_{\mu_\beta}'(\omega_\alpha^0) -1 \big) F_{\mu_\alpha}'''(\omega_\beta^0) F_{\mu_\beta}'(\omega_\alpha^0)^4  \\
&= \nonumber
F_{\mu_\beta}'''(\omega_\alpha^0) \big(
F_{\mu_\beta}'(\omega_\alpha^0)-1 +
\frac{1-F_{\mu_\alpha}'(\omega_\beta^0)^4}{1-F_{\mu_\alpha}'(\omega_\beta^0)}
\big)
+
\big(F_{\mu_\beta}'(\omega_\alpha^0) - 1\big) F_{\mu_\alpha}'''(\omega_\beta^0) F_{\mu_\beta}'(\omega_\alpha^0)^4  \\
&= \nonumber
\big( F_{\mu_\beta}'(\omega_\alpha^0) -1 \big)  F_{\mu_\beta}'''(\omega_\alpha^0)F_{\mu_\alpha}'(\omega_\beta^0)^4
+
\big(F_{\mu_\beta}'(\omega_\alpha^0) -1\big) F_{\mu_\alpha}'''(\omega_\beta^0) F_{\mu_\beta}'(\omega_\alpha^0)^4  \\
&= 
\big(F_{\mu_\beta}'(\omega_\alpha^0) -1\big) \big(  F_{\mu_\beta}'''(\omega_\alpha^0)F_{\mu_\alpha}'(\omega_\beta^0)^4 + F_{\mu_\alpha}'''(\omega_\beta^0) F_{\mu_\beta}'(\omega_\alpha^0)^4  \big),
\end{align}
where we used~\eqref{identite FF} iteratively to get the third equality, then summed the powers of $F'_{\mu_\alpha}(\omega_\beta^0)$ to get the fourth one, and then used $\big(F_{\mu_\beta}'(\omega_\alpha^0) - 1\big)\big(F_{\mu_\alpha}'(\omega_\beta^0) - 1\big)=1$ to obtain the second to last one.
Combining the last line with \eqref{RED} and $H''(\omega_\beta^0)=0$, we conclude that
\begin{align}\label{z'''}
H'''(\omega_\beta^0)
= \nonumber
&3 F_{\mu_\alpha}'(\omega_\beta^0)^4 F_{\mu_\beta}''(\omega_\alpha^0)^2
+
\big(F_{\mu_\beta}'(\omega_\alpha^0) -1\big) \big(  F_{\mu_\beta}'''(\omega_\alpha^0)F_{\mu_\alpha}'(\omega_\beta^0)^4 + F_{\mu_\alpha}'''(\omega_\beta^0) F_{\mu_\beta}'(\omega_\alpha^0)^4  \big)\\
=\nonumber
&F_{\mu_\alpha}'(\omega_\beta^0)^4 \Big( \frac{3}{2} F_{\mu_\beta}''(\omega_\alpha^0)^2 - \big(F_{\mu_\beta}'(\omega_\alpha^0)-1\big)F_{\mu_\beta}'''(\omega_\alpha^0)\Big)\\
&+
F'_{\mu_\beta}(\omega_\alpha^0)^4\frac{F'_{\mu_\beta}(\omega_\alpha^0)-1}{F'_{\mu_\alpha}(\omega_\beta^0)-1}\Big( \frac{3}{2} F_{\mu_\alpha}''(\omega_\beta^0)^2 - \big(F_{\mu_\alpha}'(\omega_\beta^0)-1\big)F_{\mu_\alpha}'''(\omega_\beta^0)\Big).
\end{align}

Now we are prepared to show that $H'''(\omega_\beta^0)<0$, if $H''(\omega_\beta^0)=0$. Using the Nevanlinna representations in \eqref{Nevan2}, we can write
\begin{align}\label{Nevanlina for CS 1}\nonumber
&\frac{3}{2} F_{\mu_\beta}''(\omega_\alpha^0)^2 - \big(F_{\mu_\beta}'(\omega_\alpha^0)-1\big)F_{\mu_\beta}'''(\omega_\alpha^0)\\ 
&=
6\bigg(\Big(\int_{\mathbb{R}}\frac{1}{(u-\omega_\alpha^0)^3}\,\widehat{\mu}_\beta(\mathrm{d}u)\Big)^2
-
\int_{\mathbb{R}}\frac{1}{(u-\omega_\alpha^0)^2}\,\widehat{\mu}_\beta(\mathrm{d}u)
\int_{\mathbb{R}}\frac{1}{(u-\omega_\alpha^0)^4}\,\widehat{\mu}_\beta(\mathrm{d}u)\bigg),
\end{align}
Using Cauchy-Schwarz on the integrals on the right side of \eqref{Nevanlina for CS 1}, we observe that
\begin{align}\label{Cauchy Schwarz}
&\frac{3}{2} F_{\mu_\beta}''(\omega_\alpha^0)^2 - \big(F_{\mu_\beta}'(\omega_\alpha^0)-1\big)F_{\mu_\beta}'''(\omega_\alpha^0)<0.
\end{align}
Identically, \eqref{Nevanlina for CS 1} and \eqref{Cauchy Schwarz} hold if we interchange $\alpha$ and $\beta$. It therefore follows from~\eqref{z'''} and~\eqref{Cauchy Schwarz}  that $H'''(\omega_\beta^0)<0$ if $H''(\omega_\beta^0)=0$. 

If $H''(\omega_\beta^0)=0$, \eqref{Taylor expansion} entails
\begin{align}
\omega_\beta(z)-\omega_\beta(E_0)=\sqrt[3]{6 \frac{z-E_0}{H'''\big(\omega_\beta(E_0)\big)}}+\mathcal{O}\big((z-E_0)^{\frac{2}{3}}\big).
\end{align} 
We remark that since $\mathrm{Im}\,\omega_\beta(x)>0$ whenever $x$ lies in the support of the free additive convolution  $\mu_\alpha\boxplus\mu_\beta$ and since $\mathrm{Re}\,\omega_\beta$ is strictly increasing on each interval of $\mathbb{R}\backslash\mathrm{supp}(\nu_\beta)$ by Lemma \ref{Nevan omega alpha}, the cubic root is uniquely determined and $E_0$ is an interior point of $\mathcal{V}_{\alpha\boxplus\beta}$.

If $H''(\omega_\beta^0)\neq 0$, then we find from~\eqref{Taylor expansion} that
\begin{align}
\omega_\beta(z)-\omega_\beta(E_0)=\sqrt{2 \frac{z-E_0}{z''\big(\omega_\beta(E_0)\big)}}+\mathcal{O}\big(z-E_0\big),
\end{align} 
and the choice of the square root is determined by the fact that $\omega_\beta$ is a self-mapping of the upper-half plane.   $E_0$ is thence an exterior point of $\mathcal{V}_{\alpha\boxplus\beta}$. 

Recalling that $\omega_\beta(E_0)$ is at positive distance from the support of ${\mu}_\alpha$, \eqref{approx I}, \eqref{rappel comparaison} and the Cauchy-Stieltjes inversion formula, we then obtain the existence of two strictly positive constants $C_1, C_2$ and of a neighborhood $\mathcal{U}_0$ of $E_0$ such that for every $x$ in $\mathcal{U}_0$, we either have 
\begin{align}
\rho_{\alpha\boxplus\beta}(x)
=
C_1\sqrt{x-E_0}+\mathcal{O}\big(x-E_0\big),
\end{align}
when $x\geq E_0$ and $\rho_{\alpha\boxplus\beta}(x)=0$, when $x<E_0$, in the case where $z''\big(\omega_\beta(E_0)\big)<0$, 
\begin{align}
\rho_{\alpha\boxplus\beta}(x)
=
C_1\sqrt{E_0-x}+\mathcal{O}\big(E_0-x\big),
\end{align}
when $x\leq E_0$ and $\rho_{\alpha\boxplus\beta}(x)=0$, when $x>E_0$, in the case where $H''\big(\omega_\beta(E_0)\big)>0$, or we have 
\begin{align}
\rho_{\alpha\boxplus\beta}(x)
=
C_2\sqrt[3]{|x-E_0|}+\mathcal{O}\big( |x-E_0|^{\frac{2}{3}} \big),
\end{align}
if $H''\big(\omega_\beta(E_0)\big)= 0$. This concludes the proof of Theorem \ref{Theorem 3}. 
\end{proof}

\textit{Acknowledgment: } P.M.\ acknowledges support from the Royal Swedish Academy of Sciences and from the Swedish Foundation for International Cooperation in Research and Higher Education, Grant No. PD2023-9315. P.M.\ is very grateful to Kevin Schnelli for useful discussions on this problem and comments on the first versions of this manuscript. P.M.\ is also very grateful to Christophe Charlier for useful remarks on some technicalities of the paper, and to the anonymous referees for useful remarks and suggestions.

\end{document}